%% file: main.tex
\documentclass[11pt]{article}

\usepackage[utf8]{inputenc} 
\usepackage[T1]{fontenc}    
\usepackage{url}            
\usepackage{booktabs}       
\usepackage{nicefrac}       
\usepackage{microtype}      
\usepackage{multirow}
\usepackage{array}
\usepackage[noblocks]{authblk}
\usepackage{amssymb,arydshln}
\usepackage[nodayofweek]{datetime}
\usepackage{float}
\usepackage{bm}

\usepackage[shortlabels]{enumitem}
\setlist[itemize]{itemindent=0ex,parsep=3pt,itemsep=0pt,leftmargin=\parindent,topsep=5pt,labelwidth=0.8em,labelsep=0.7em}
\setlist[enumerate]{label={\arabic*)},itemindent=0ex,parsep=3pt,itemsep=0pt,leftmargin=\parindent,topsep=5pt,labelwidth=0.9em,labelsep=0.6em}

\usepackage[font=small,labelfont=bf]{caption}
\usepackage{amsmath,amsfonts,nccmath,mathtools,mathrsfs}
\usepackage{tcolorbox}
\usepackage{xcolor}

\usepackage[margin=1in]{geometry}

\usepackage{pgfplots}
\usepgfplotslibrary{fillbetween}
\usetikzlibrary{arrows.meta}
\tikzset{>=latex}

\DeclarePairedDelimiterX\setv[2]{\{}{\}}{#1 \;\delimsize\vert\; #2}

\pgfmathdeclarefunction{cubic}{1}{%
  \pgfmathparse{-2*(#1+2)*(#1+2)*(#1-2)+40}%
}

\pgfmathdeclarefunction{bicuadratic}{1}{%
  \pgfmathparse{(#1-1)*(#1-1)*(#1-1)*(#1-1)+10}%
}

\pgfmathdeclarefunction{cuadratic}{1}{%
    \pgfmathparse{(-(2*#1)^2)+70}%
}

\usepackage{pifont}

\usepackage[numbers,merge,sort&compress]{natbib}
\usepackage{amsthm}
\usepackage{graphicx,color}
\usepackage{subcaption,setspace}

\usepackage[colorlinks = true,
            linkcolor = blue,
            urlcolor  = blue,
            citecolor = blue,
            anchorcolor = blue]{hyperref}

\newtheorem{theorem}{Theorem}[section]
\newtheorem{fact}{Fact}[section]
\newtheorem{proposition}{Proposition}[section]

\newtheorem{lemma}{Lemma}[section]
\newtheorem{corollary}{Corollary}[section]

\newtheorem{assumption}{Assumption}

\theoremstyle{definition}
\newtheorem{definition}{Definition}[section]
\newtheorem{example}{Example}[section]

\newtheorem{remark}{Remark}[section]

\usepackage[nameinlink]{cleveref}
\crefname{equation}{}{}
\crefname{theorem}{Theorem}{Theorems}
\crefname{corollary}{Corollary}{Corollaries}
\crefname{example}{Example}{Examples}
\crefname{assumption}{Assumption}{Assumptions}
\crefname{lemma}{Lemma}{Lemmas}
\crefname{proposition}{Proposition}{Propositions}
\crefname{figure}{Figure}{Figures}
\crefname{table}{Table}{Tables}
\crefname{fact}{Fact}{Facts}
\crefname{conjecture}{Conjecture}{Conjectures}
\crefname{section}{Section}{Sections}
\crefname{appendix}{Appendix}{Appendices}
\Crefname{equation}{}{}
\Crefname{theorem}{Theorem}{Theorems}
\Crefname{corollary}{Corollary}{Corollaries}
\Crefname{example}{Example}{Examples}
\Crefname{lemma}{Lemma}{Lemma}
\Crefname{proposition}{Proposition}{Proposition}
\Crefname{figure}{Figure}{Figures}
\Crefname{table}{Table}{Tables}
\Crefname{section}{Section}{Sections}
\Crefname{appendix}{Appendix}{Appendices}

\newcommand{\tr}{{{\mathsf T}}}
\newcommand{\her}{{{\mathsf H}}}
\newcommand{\mK}{{\mathsf{K}}}
\newcommand{\mZ}{{\mathsf{Z}}}

\newcommand{\ECL}{$\mathtt{ECL}$}

\newcommand{\LQG}{\mathtt{LQG}}

\usepackage{tikz}

\makeatletter
\newcommand{\removelatexerror}{\let\@latex@error\@gobble}
\makeatother

\title{\bf Benign Nonconvex Landscapes in Optimal and Robust Control, Part II: Extended Convex Lifting 
\thanks{The work of Y. Zheng and C. Pai is supported by NSF ECCS-2154650, NSF CMMI-2320697, and NSF CAREER 2340713. The work of Y. Tang is supported by NSFC through grant 72301008.}}
\author[1]{Yang Zheng}
\author[1]{Chih-Fan Pai}
\author[2]{Yujie Tang}

\affil[1]{\small Department of Electrical and Computer Engineering, University of California San Diego}
\affil[2]{\small Department of Industrial Engineering \& Management, Peking University}
\date{ \small \today } 

\begin{document}

\maketitle
\vspace{-5mm}

\begin{abstract}

Many optimal and robust control problems are nonconvex and potentially nonsmooth in their policy optimization forms. In Part II of this paper, we introduce a new and unified Extended Convex Lifting (\ECL) framework to reveal \textit{hidden convexity} in classical optimal and robust control problems from a modern optimization perspective. Our \ECL{} offers a bridge between nonconvex policy optimization and convex reformulations, enabling convex analysis for nonconvex problems.  Despite non-convexity and non-smoothness, the existence of an \ECL{} not only reveals that minimizing the original function is equivalent to a convex problem but also certifies a class of first-order \textit{non-degenerate} stationary points to be globally optimal. Therefore, no spurious stationarity exists in the set of non-degenerate policies. This \ECL{} framework can cover many benchmark control problems, including state feedback linear quadratic regulator (LQR), dynamic output feedback linear quadratic Gaussian (LQG) control, and $\mathcal{H}_\infty$ robust control. \ECL{} can also handle a class of distributed control problems when the notion of quadratic invariance~(QI)~holds. We further show that all static stabilizing policies are non-degenerate for state feedback LQR and  $\mathcal{H}_\infty$ control under standard assumptions. We believe that the new \ECL{} framework may be of independent interest for analyzing nonconvex problems beyond control.

\vspace{-1mm}

\end{abstract}

\tableofcontents

\input{sec_Intro.tex}

\input{sec_Examples.tex}

\input{sec_ECL.tex}

\input{sec_Applications.tex}

\section{Conclusion} \label{section:conclusion}

In this work, we have introduced a unified \texttt{Extended Convex Lifting} (\ECL{}) framework to reveal hidden convexity in nonconvex and potentially nonsmooth optimization problems. Our \ECL{} can be viewed as a bridge to connect nonconvex optimization problems with their convex reformulations, thus enabling convex analysis for benign nonconvex landscapes. Despite non-convexity and non-smoothness, the existence of an \ECL{} guarantees that minimizing the original function is equivalent to a convex problem (\Cref{theorem:convex-equivalency}), and also certifies a class of first-order {non-degenerate} stationary points to be globally optimal (\Cref{theorem:ECL-guarantee}). 

We have shown that \ECL{} allows for a unified analysis of the hidden convexity and global optimality in benchmark optimal and robust control problems from a modern optimization perspective. We have built explicit \ECL{}s for LQR, state feedback $\mathcal{H}_\infty$ control, LQG, dynamic output feedback $\mathcal{H}_\infty$ control, as well as a class of distributed control problems, many of which were analyzed on a case-by-case basis in the literature. The guarantees for LQG and dynamic output feedback $\mathcal{H}_\infty$ control are new and only reported in Part I of this paper \cite{zheng2023benign}. We believe this \ECL{} framework will be helpful to reveal hidden convexity for other control problems where LMI-based solutions have been revealed~\cite{scherer2000linear,dullerud2013course}. We hope the \ECL{} framework will be useful for analyzing nonconvex optimization problems in other areas beyond control.

\section*{Acknowledgement}
The authors would like to thank Chengkai Yao for his assistance with creating some illustrative figures, particularly \Cref{fig:non-convexity-illustration}(b), \Cref{fig:local-geometry,fig:simple-example-LQR-main-paper}, in this paper.

\addcontentsline{toc}{section}{References}
\bibliographystyle{unsrt}
\bibliography{ref.bib}

\newpage

\appendix

\numberwithin{equation}{section}
\numberwithin{example}{section}
\numberwithin{remark}{section}
\numberwithin{assumption}{section}

\vspace{10mm}
\noindent\textbf{\Large Appendix}
\vspace{5mm}

This appendix contains auxiliary results, additional discussions, and technical proofs. We divide it into four parts:
\begin{itemize}
 \item \cref{appendix:computational-details} presents some computational details for \cref{fig:non-convexity-illustration} and \Cref{example:LQR_ex}.
 \item \cref{appendix:stationary-point} reviews stationary points in the benchmark control problems in \Cref{section:applications}.
 \item \Cref{appendix:state-feedback} presents technical proofs for state feedback control problems (\cref{subsection:static-policies}).
 \item \Cref{appendix:dynamic-policies} presents technical proofs for output feedback control problems (\cref{subsection:dynamic-policies}).
\end{itemize}

\vspace{5mm}

\input{app_sec1.tex}

\input{app_sec2.tex}

\input{app_sec3.tex}

\input{app_sec4.tex}

\end{document}

%% file: sec_Intro.tex
\section{Introduction}

Many optimal and robust control problems, including linear quadratic regulator (LQR), linear quadratic Gaussian (LQG) control, and robust $\mathcal{H}_\infty$  control, have been extensively studied in classical control since the 1960s \cite{kalman1963theory,levine1970determination,doyle1978guaranteed}. It is well-known that almost all these control problems are naturally nonconvex in the space of controller (i.e., policy\footnote{We use the terminologies ``controller'' and ``policy'' interchangeably in this paper. Both of them are certain \textit{functions} that map system outputs to system inputs.}) parameters. Still, many fruitful results, such as optimal solution structures (static or dynamic), stability margin of LQR, the separation principle for LQG, and Riccati-based strategies for $\mathcal{H}_\infty$ control, have been obtained. We refer the interested reader to classical textbooks \cite{zhou1996robust,green2012linear,dullerud2013course} and excellent monographs \cite{scherer2000linear,boyd1994linear}.

Most classical techniques either rely on proper controller re-parameterizations, leading to convex problems in terms of linear matrix inequalities (LMIs) \cite{scherer2000linear,boyd1994linear,dullerud2013course}, or develop suitable Riccati-based strategies to characterize optimal or suboptimal controllers \cite{zhou1996robust,green2012linear}.  All these methods do not optimize over the space of policy parameters directly, and they often require an explicit underlying system model for re-parameterizing the controller space or solving Riccati equations. On the other hand, the optimization landscapes of optimal and robust control problems can also offer fruitful results (see \cite{lewis2007nonsmooth,hu2023toward, Talebi2024geometry} for surveys), in which we consider the control costs as functions of the policy parameters and study their analytical and geometrical properties. Indeed, this optimization perspective is naturally amenable for data-driven design paradigms such as reinforcement learning \citep{recht2019tour}, but it leads to nonconvex and potentially nonsmooth problems. 
Despite the lack of convexity~(and possibly smoothness), in Part I \cite{zheng2023benign}, we have established that in a class of \textit{non-degenerate} policies, all Clarke stationary points are globally optimal for both LQG and $\mathcal{H}_\infty$ control with output measurements. 

In Part II of this paper, we introduce a unified framework to reveal \textit{hidden convexity} in classical optimal and robust control problems from a modern optimization perspective. This framework~covers many iconic control problems, including state feedback LQR, dynamic output feedback LQG and $\mathcal{H}_\infty$ control, and a class of distributed control problems with quadratic invariance \cite{furieri2020learning}. The global optimality results for LQG and $\mathcal{H}_\infty$ control in Part I \cite[Thereoms 4.2 \& 5.2]{zheng2023benign} will be direct~corollaries.

\subsection{Non-convexity and Non-smoothness in Control} \label{subsection:nonconvexity}

For many iconic optimal control problems, nonconvexity naturally arises when we directly search over a suitably parameterized policy space to optimize their control objectives. For example, given a linear system $\dot{x} =  Ax + B u$, it is well-known that the set of static state feedback gains $K$ that stabilize the system via $u=Kx$ is a nonconvex set~\cite{fazel2018global}, meaning that we need to search over a nonconvex policy space to find an optimal feedback gain; see \Cref{fig:non-convexity-illustration}(a) for an illustration of the policy space and \Cref{fig:non-convexity-illustration}(b) for a typical nonconvex LQR landscape. If we consider output feedback controller synthesis such as linear quadratic Gaussian (LQG) control, then the parameterized set of dynamic policies can even be disconnected \cite{zheng2021analysis}. Furthermore, the LQG  objective function may have spurious stationary points \cite{zheng2022escaping}, and there can be uncountably infinitely many globally optimal policies that reside on a nonconvex manifold due to the symmetry induced by \emph{similarity transformations}~\cite{zheng2022escaping,zheng2021analysis,kraisler2024output}. In fact, one fundamental problem in output feedback controller synthesis is to re-parameterize the set of stabilizing dynamic controllers in a convex way, and a classical solution is the celebrated Youla parameterization \cite{youla1976modern} and two recent ones are the system-level parameterization \cite{wang2019system} and input-output parameterization~\cite{furieri2019input,zheng2021equivalence,zheng2022system}.

\begin{figure}
    \centering
    \includegraphics[width=0.26 \textwidth]{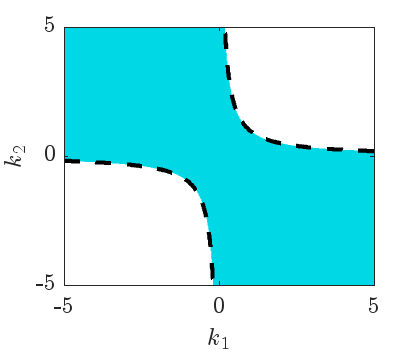} \hspace{5mm}
    \includegraphics[width=0.3 \textwidth]{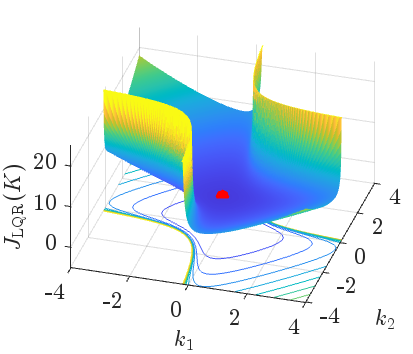} \hspace{5mm}
    \includegraphics[width=0.3 \textwidth]{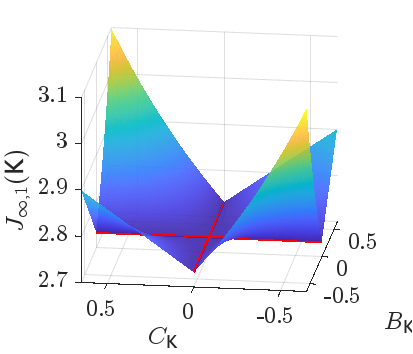}
    \caption{Illustration of nonconvexity and non-smoothness in control: (a) a slice of the set of static feedback gains $K$ such that the closed-loop system $A+BK$ is stable, where $A = 0, B = I_2$; (b) nonconvex and smooth LQR cost; (c) nonsmooth and nonconvex $\mathcal{H}_\infty$ cost. Computational details are given in \Cref{appendix:details-of-figure-1}.}
    \label{fig:non-convexity-illustration}
\end{figure}

In addition to non-convexity, non-smoothness may also arise when we consider robust control problems that address the worst-case performance against uncertainties or adversarial noises. In this case, a typical performance measure is the $\mathcal{H}_\infty$ norm of a certain closed-loop transfer function \cite{zhou1996robust}, which is known to be both \textit{nonconvex} and \textit{nonsmooth} in the policy space~\cite{apkarian2006nonsmooth,lewis2007nonsmooth}. In fact, robust control problems were one of the early motivations and applications for nonsmooth optimization \cite{lewis2007nonsmooth}. \Cref{fig:non-convexity-illustration}(c) illustrates the nonconvex and nonsmooth landscape in an $\mathcal{H}_\infty$ control~instance.

\begin{figure}[tbhp]
    \centering
    \setlength{\abovecaptionskip}{3pt}
        \includegraphics[height=2.8cm]{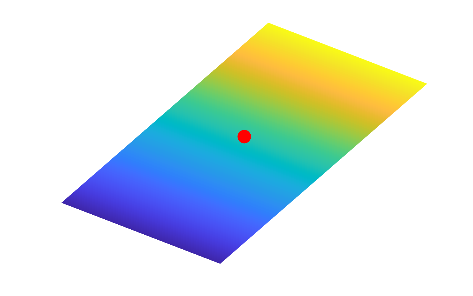} \hspace{-2.5mm}
        \includegraphics[height=2.8cm]{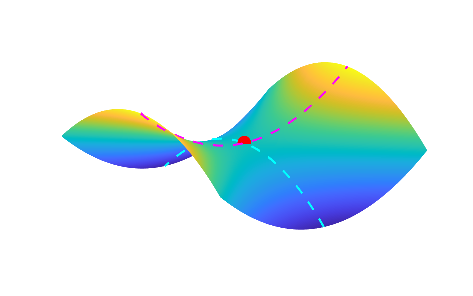} \hspace{1.5mm}
        \includegraphics[height=2.8cm]{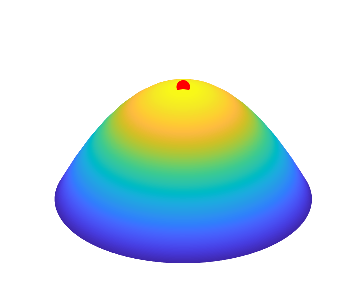} \hspace{1mm}
        \includegraphics[height=2.8cm]{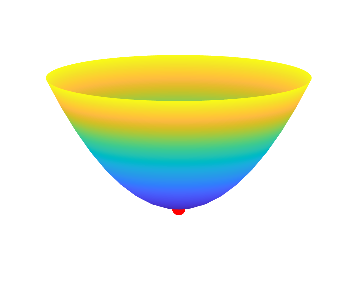} \\ \vspace{-6mm}
        
        {\footnotesize  \flushleft  \hspace{0.6cm } (a) \hspace{3.8cm}  (b) \hspace{3.6cm}  (c) \hspace{3.0cm}  (d)}
        \caption{Local geometry of a smooth function $\phi(x)$:
        (a) a non-stationary point with non-zero gradient $\nabla \phi(x) \neq 0$;
        (b) (strict) saddle; (c) local maximizer; (d) local minimizer. A general nonconvex function~$\phi(x)$ may have all these stationary points. If $\phi(x)$ has \textit{hidden convexity} and is equipped with an \ECL{}, all non-degenerate stationary points are global minimizers (see \Cref{theorem:ECL-guarantee}); no saddle and local maximizers exist. 
        }
    \label{fig:local-geometry}
\end{figure}

 For nonconvex and nonsmooth optimization problems, it is generally very hard to derive theoretical guarantees or certificates for local policy search algorithms. Intuitively, the local geometry of nonconvex problems does not give a global performance guarantee, and there may be spurious local minimizers (including saddles). \Cref{fig:local-geometry} illustrates possible local geometry of a smooth function.     
As summarized in Part I \cite{zheng2023benign}, a series of recent findings have revealed \textit{favorable benign nonconvex landscape properties} in many benchmark control problems, including LQR \cite{fazel2018global,mohammadi2021convergence,fatkhullin2021optimizing}, risk-sensitive control \cite{zhang2021policy}, LQG control \cite{zheng2021analysis,zheng2022escaping,duan2022optimization}, dynamic filtering~\cite{umenberger2022globally,zhang2023learning}, and $\mathcal{H}_\infty$ control \cite{hu2022connectivity,guo2022global,tang2023global}, etc. One goal of our work is to provide a unified framework that can explain the benign nonconvex landspace properties for all these iconic control problems.

\subsection{Our Techniques --- Extended Convex Lifting (\ECL)}

In this paper, we aim to develop a new and unified framework, called \texttt{Extended Convex Lifting (ECL)} (see \Cref{fig:ECL-process} for an illustration), which reconciles the gap between nonconvex policy optimization and convex reformulations. The remarkable (now classical) results in control theory reveal that many optimal and robust control problems admit ``convex reformulations''\footnote{For state feedback LQR and $\mathcal{H}_\infty$ control, convex reformulations are relatively straightforward; see \cite{boyd1994linear}. However, convex reformulations for output feedback control are non-trivial and subtle; see \Cref{remark:convex-reformulation,remark:saddle-points}.} in terms of LMIs, via a suitable change of variables~\cite{scherer2000linear,dullerud2013course,boyd1994linear}. These changes of variables enable us to analyze the nonconvex policy optimization problems through their convex reformulations.

Here we provide a simple example to demonstrate some intuition. Suppose $f(x)$ is a potentially nonconvex function, and there exists a smooth bijection $y = \Phi(x)$ such that $g(y) = f(\Phi^{-1}(y))$ is convex. Then, it is clear that minimizing $f(x)$ is equivalent to minimizing $g(y)$, which is a convex problem. Furthermore, all stationary points of $f(x)$ are globally optimal, and there exist no local maximizers or saddle points.  Interestingly, this simple process is almost sufficient to guarantee that any stationary point is globally optimal for a large class of LQR policy optimization problems (see \Cref{example:LQR_ex} and \Cref{remark:eliminating-lifting-variables}).

For many other control problems, however, such a simple process is insufficient and the bijection $\Phi$ does not exist due to the natural appearance of Lyapunov variables. Instead, 
we develop a generic template for \textit{extended} convex~reparameterization of epigraphs of nonconvex functions $f(x)$, which significantly generalizes the simple setup above. We call this generic template \texttt{Extended Convex Lifting (ECL)}. As illustrated in \Cref{fig:ECL-process}, we first \textit{lift} the (non-)strict epigraph to a higher dimensional set by appending auxiliary variables (which often corresponds to Lyapunov variables) and then assume a smooth bijection $\Phi$ that maps this lifted nonconvex set to a ``partially'' convex set $\mathcal{F}_{\mathrm{cvx}} \times \mathcal{G}_{\mathrm{aux}}$; see \Cref{definition:ECL} for a precise procedure. 
In many control problems, $\mathcal{F}_{\mathrm{cvx}}$ is a convex set represented by an LMI, while  $\mathcal{G}_{\mathrm{aux}}$ can be nonconvex and can be used to account for similarity transformations of {dynamic output feedback}~policies. 

\begin{figure}[t]
    \centering
    \includegraphics[width = 0.9\textwidth]{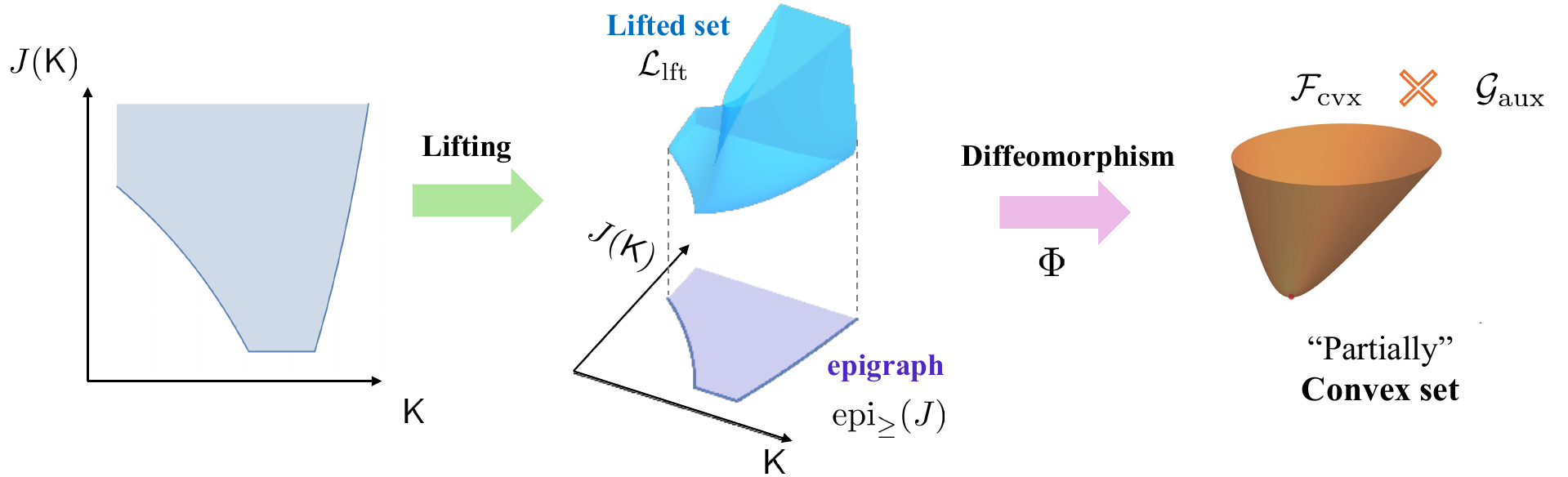}
    \caption{A schematic illustration of \ECL{}. Left figure: We begin with the epigraph $\operatorname{epi}_\geq(J) $ of a potentially \textit{nonconvex} and \textit{nonsmooth} function $J(\mK)$; 
    Middle figure: We lift the epigraph $\operatorname{epi}_\geq(J)$ to a set $\mathcal{L}_{\mathrm{lft}}$ of a higher dimension; Right figure: We identify a smooth and invertible mapping $\Phi$ that maps $\mathcal{L}_{\mathrm{lft}}$ to a ``partially'' convex set $\mathcal{F}_{\mathrm{cvx}} \times \mathcal{G}_{\mathrm{aux}}$. In many control problems, $\mathcal{F}_{\mathrm{cvx}}$ is a convex set represented by LMIs, while  $\mathcal{G}_{\mathrm{aux}}$ can be nonconvex and provides the flexibility of capturing similarity transformations (cf. \Cref{section:applications}). The schematic here serves an illustrative role, and the full \ECL{} details are elaborated in \Cref{section:ECL}.}
    \label{fig:ECL-process}
\end{figure}
 
Despite non-convexity and non-smoothness, the existence of an \ECL{} not only reveals that minimizing the original function is equivalent to a convex problem (\Cref{theorem:convex-equivalency}) but also identifies a class of first-order (\textit{non-degenerate}) stationary points to be globally optimal (\Cref{theorem:ECL-guarantee}). While our \ECL{} constructions and the corresponding guarantees in \Cref{theorem:convex-equivalency,theorem:ECL-guarantee} have a strong geometrical intuition, the details in the construction as well as in the proofs require resolving some subtleties. We thus believe our \ECL{} framework is of independent interest.

Our \texttt{ECL} framework presents a unified template for convex re-parameterizations of many iconic control problems~\cite{scherer2000linear,boyd1994linear,dullerud2013course}. Many recent results on global optimality of (non-degenerate) stationary points, such as LQR in \cite{fazel2018global,mohammadi2021convergence}, LQG in \cite{zheng2021analysis}, state feedback $\mathcal{H}_\infty$ control in \cite{guo2022global}, Kalman filter in \cite{umenberger2022globally}, some distributed control problem \cite{furieri2020learning}, are special cases once the corresponding \texttt{ECL} is constructed. We note that the exact \ECL{} constructions, especially those for output feedback control problems, require special treatments that have not been discussed in existing literature, which will be presented in detail in \Cref{section:applications}.  From our perspective of \ECL{}, standard state feedback and output feedback control problems are \textit{almost convex problems} in disguise; all the non-convexity and non-smoothness in \Cref{subsection:nonconvexity} are benign in this sense. 

We finally highlight that almost all existing convex reformulations \cite{scherer2000linear,boyd1994linear,dullerud2013course} for control problems rely on strict LMIs. They are sufficient for deriving suboptimal policies in control, but typically fail to characterize globally optimal policies. To characterize global optimality, our use of non-strict LMIs in the \ECL{} framework requires resolving non-trivial technical details, especially in the LQG and output feedback $\mathcal{H}_\infty$ cases (see \Cref{lemma:inclusion-projection-LQG,lemma:inclusion-projection-Hinf}).

\subsection{Related Work}

We here summarize some closely related work in two directions: 1) solutions to optimal and robust control, and 2) lifting strategies in analyzing nonconvex problems.\footnote{For other related work on policy optimization in control, please refer to Part I \cite{zheng2023benign} and the references therein}

\vspace{-2mm}

\paragraph{Solutions of optimal LQG control and $\mathcal{H}_\infty$ robust control.}

The LQG (or $\mathcal{H}_2$) optimal~control theory was heavily studied in the 1960s \cite{kalman1963theory}. Many structural properties are now well-understood, such as the existence of the optimal LQG controller and the separation principle of the controller structure. It is known that the globally optimal LQG controller is unique in the frequency domain, and can be found by solving two Riccati equations \cite[Chap. 14]{zhou1996robust}. It is also well-known that the LQG optimal controller has no guaranteed stability margin \cite{doyle1978guaranteed}, and thus robustness does not come free from linear quadratic optimization in the output feedback case. This fact partially motivated the development of $\mathcal{H}_\infty$ robust control, which was heavily studied in the 1980s \cite{zames1981feedback,francis1987course,glover2005state}. One key milestone was reached in the seminal paper \cite{doyle1988state}, where state-space solutions via solving two Riccati equations were derived to get all suboptimal $\mathcal{H}_\infty$ controllers. The suboptimal $\mathcal{H}_\infty$ controller in \cite{doyle1988state} also has a separation structure reminiscent of classical LQG theory. However, unlike the LQG theory, the optimal $\mathcal{H}_\infty$ controllers are not only generally non-unique in the frequency domain \cite[Page 406]{zhou1996robust}, \cite[Section 3.2]{GLOVER201715} (also see \cite[Example 5.4]{zheng2023benign}), but also quite subtle to characterize \cite[Section 5.1]{glover2005state}. Besides Riccati-based solutions, LMI-based approaches for both LQG and $\mathcal{H}_\infty$ control were developed in 1990s \cite{gahinet1994linear,scherer1997multiobjective}.

\vspace{-2mm}
\paragraph{Lifting strategies in analyzing nonconvex problems} 

Recent years have witnessed a surge of interest in analyzing nonconvex problems using convex analysis; see e.g.,  \cite{levin2022effect,guo2022global,umenberger2022globally,sun2021learning,mohammadi2021convergence,zheng2021analysis,tang2023global}. Most of these results rely on a suitable change of variables that makes the original problem simpler to deal with.  For example, a recent study in \cite{levin2022effect} examines landscapes in nonconvex optimization by smoothly parameterizing its domain. The work \cite{fawzi2022lifting} discusses concise descriptions of convex sets by exploiting the concept of lifting as simplified higher-dimensional representations that project onto the original sets. 
As for nonconvex control problems, many of them admit suitable convex reformulations using LMIs, which leads to the desired landscape properties for policy optimization. For instance, the authors of  \cite{mohammadi2021convergence} leveraged such a technique to study global convergence properties of policy optimization for continuous-time LQR. The analysis via convex parameterizations was also studied in \cite{sun2021learning} for state feedback problems with smooth objectives. More recently, it has been shown that for nonsmooth discrete-time static state feedback $\mathcal{H}_\infty$ control, all Clarke stationary points are global minima \cite{guo2022global}, in which the key proof idea is also based on convex reformulations. All these results focused on static state feedback policies.  
A differentiable convex lifting strategy was proposed in \cite{umenberger2022globally}, applicable to dynamic output feedback estimation problems (e.g., Kalman filter). Our previous work \cite{zheng2021analysis} also utilized a classical change of variables to reveal some nice geometry of stabilizing dynamic feedback policies, which is shown to have at most two path-connected components.

Inspired by these recent advances, our \ECL{} framework offers a unified way to bridge the gap between nonconvex policy optimization and classical convex reformulations in control. This \ECL{} framework is more general than  \cite{zheng2021analysis,umenberger2022globally,guo2022global,sun2021learning,mohammadi2021convergence}, in the sense that it can directly handle static state feedback and dynamic output feedback policies, as well as smooth and nonsmooth cost functions. More importantly, our \ECL{} framework naturally classifies degenerate and non-degenerate policies, which reflects the subtleties between strict and non-strict LMIs in control.

\subsection{Paper Outline}

The rest of this paper is organized as follows. \Cref{section:motivating-examples} presents three motivating examples to illustrate hidden convexity via a diffeomorphism. We present the full technical details of \ECL{} in \Cref{section:ECL}, including its construction and theoretical guarantees. In \Cref{section:applications}, we construct appropriate \ECL{}s for nonconvex (and potentially nonsmooth) policy optimization problems in benchmark control problems. We finally conclude the paper in \Cref{section:conclusion}. Some auxiliary results, additional discussions, and technical proofs are provided in \cref{appendix:computational-details,appendix:stationary-point,appendix:state-feedback,appendix:dynamic-policies}.

\paragraph{Notations.} 
The set of $k\times k$ real symmetric matrices is denoted by $\mathbb{S}^k$, and $\mathbb{S}^k_{+}$ (resp. $\mathbb{S}^k_{++}$) denotes the set of $k\times k$ positive semidefinite (resp. positive definite) matrices. We denote the set of $k \times k$ real invertible matrices by $\mathrm{GL}_k=\{T\in\mathbb{R}^{k\times k} \mid \det T\neq 0\}$. Given a matrix $M$, $M^\tr$ denotes its transpose, and $M^\her$ denotes its conjugate transpose. For any $M_1,M_2\in\mathbb{S}^k$, we use $M_1\prec (\preceq) M_2$ and $M_2\succ (\succeq) M_1$ to mean that $M_2-M_1$ is positive (semi)definite. We use $I_n$ (resp. $0_{m\times n}$) to denote the $n\times n$ identity matrix (resp. $m\times n$ zero matrix), and sometimes omit their subscripts when the dimensions can be inferred from the context. For a subset $S$ of a topological space, we let $\operatorname{int} S$ denote its interior and $\operatorname{cl}S$ denote its closure.

%% file: sec_Examples.tex
\section{Motivating Examples} \label{section:motivating-examples}

To motivate our discussion, we consider three simple examples in \cref{subsection:examples}: one academic example, and two state feedback control instances. All of them are essentially (or equivalent to) convex problems. In \cref{subsection:epigraphs}, we summarize a useful fact on hidden convexity via a diffeomorphism over epigraphs. This section serves as a better motivation and can be skipped if the reader wishes to see the full mathematical theory behind \ECL{}, which is summarized in \Cref{section:ECL}.

\subsection{Three Examples} \label{subsection:examples}

\begin{figure}[t]
    \centering
    \setlength{\abovecaptionskip}{6pt}
    \begin{subfigure}{0.22\textwidth}
        \includegraphics[width=\textwidth]{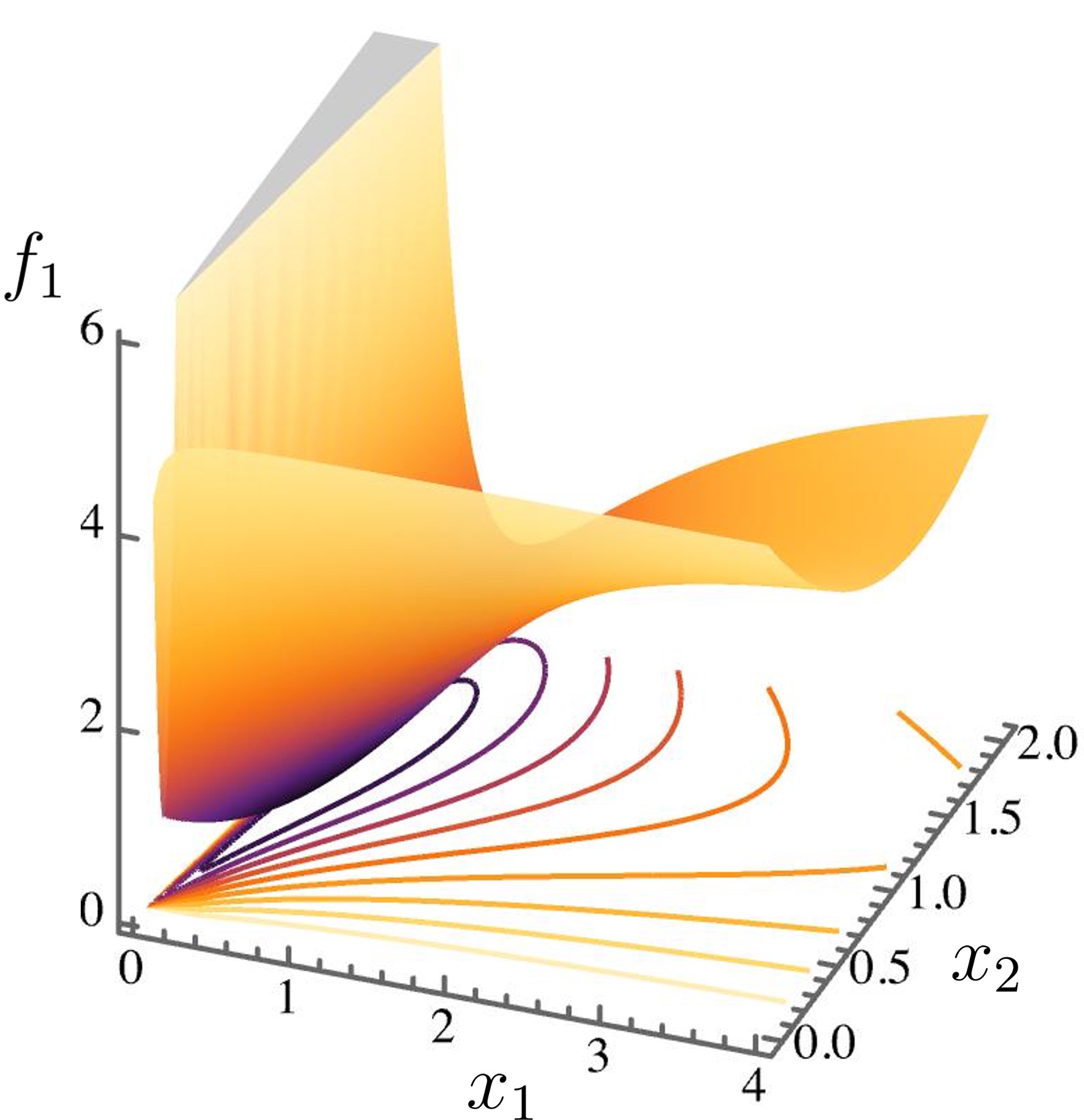}
        \caption{Nonconvex smooth}
        \label{fig:ex-a-1}
    \end{subfigure}
    \hspace{4mm}
    \begin{subfigure}{0.22\textwidth}
        \includegraphics[width=\textwidth]{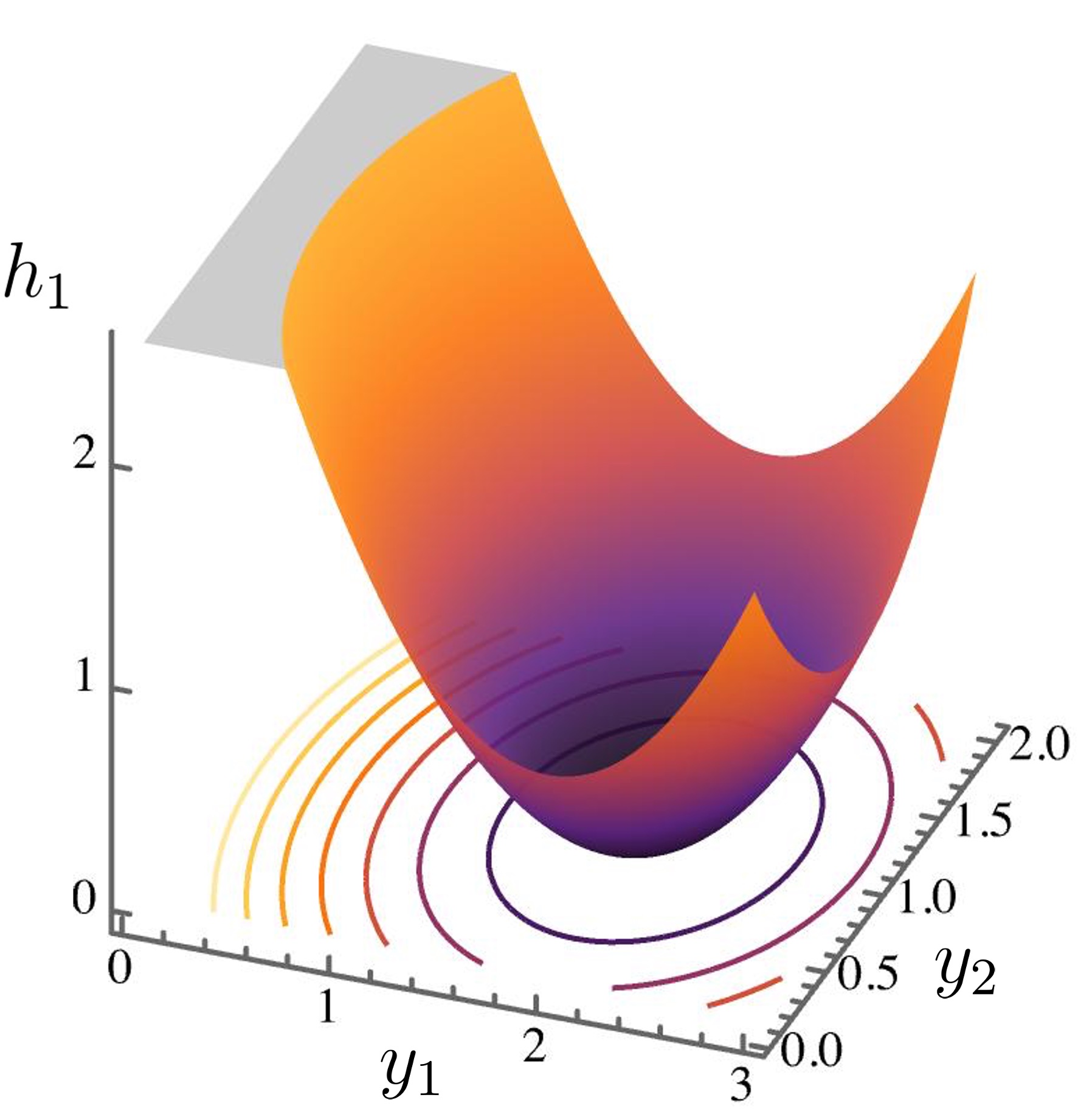}
        \caption{Convex smooth}
        \label{fig:ex-a-2}
    \end{subfigure}
    \hspace{4mm}
    \begin{subfigure}{0.22\textwidth}
        \includegraphics[width=\textwidth]{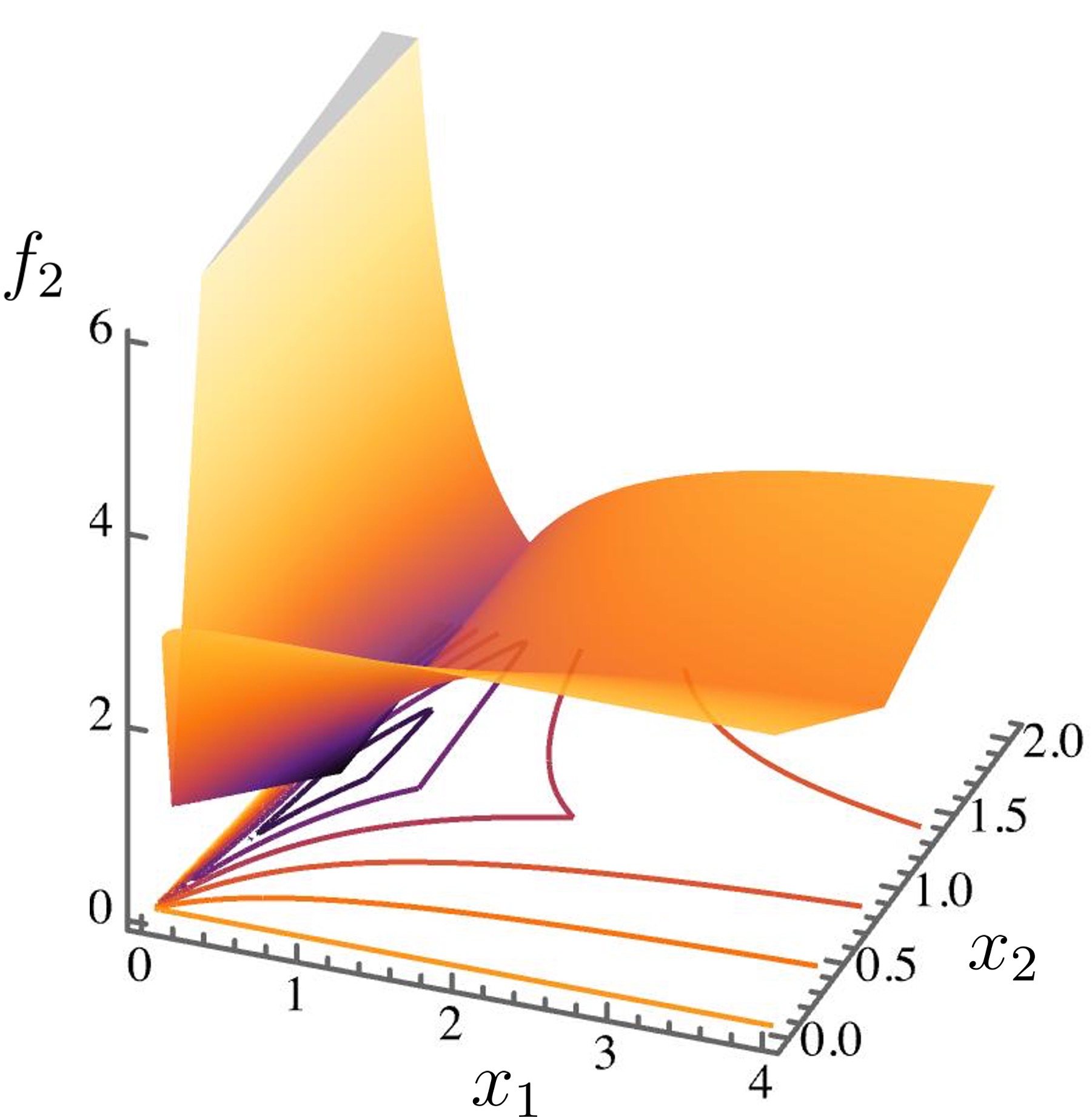}
        \caption{Nonconvex nonsmooth}     \label{fig:ex-b-1}
    \end{subfigure}
    \hspace{4mm}
    \begin{subfigure}{0.22\textwidth}
        \includegraphics[width=\textwidth]{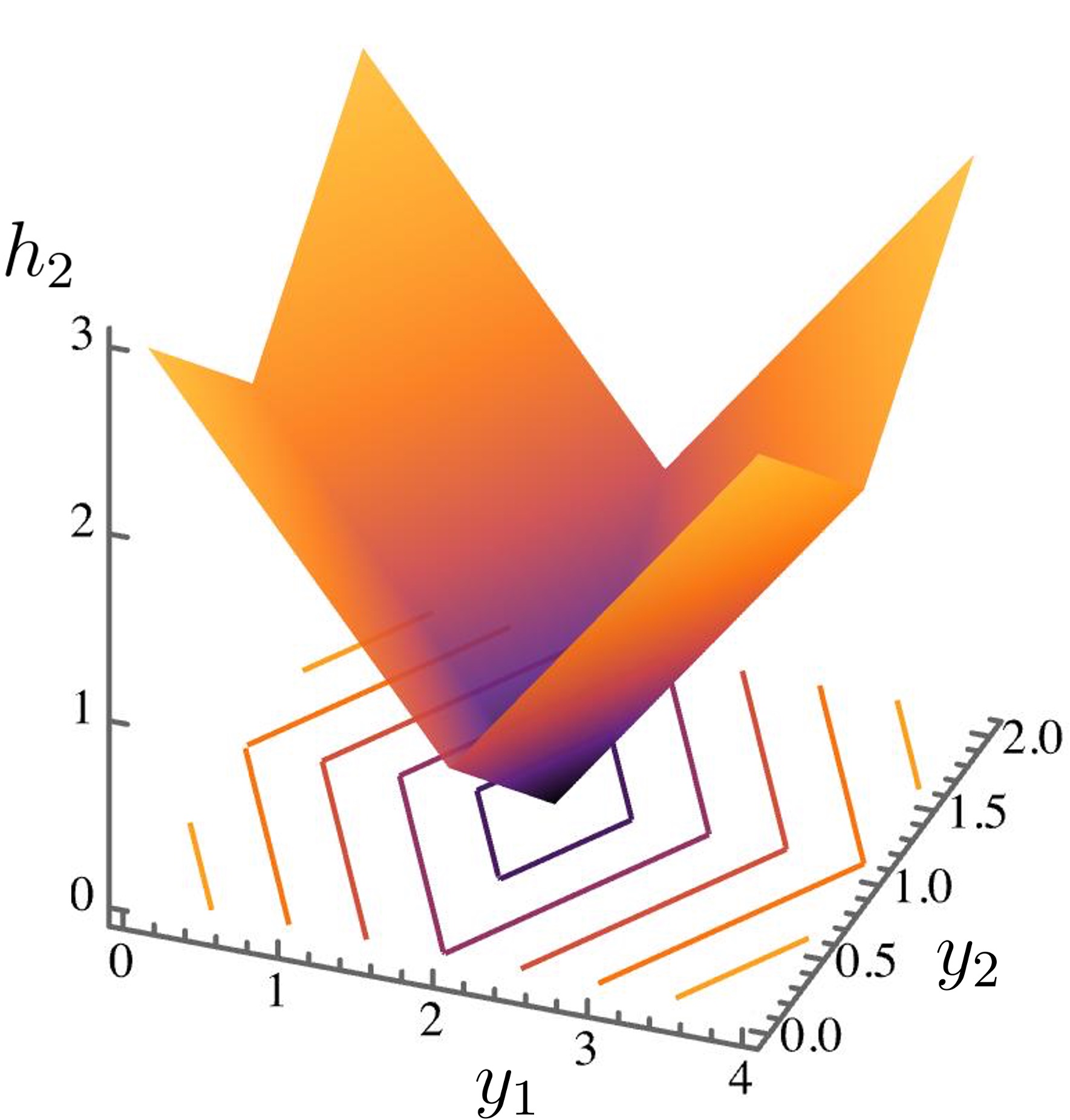}
        \caption{Convex nonsmooth}     \label{fig:ex-b-2}
    \end{subfigure}
    \caption{Illustration of nonconvex functions and their transformation in \cref{example:academic}.  
    }
    \label{fig:academic-example}
\end{figure}

Our first example highlights that a suitable change of variables (i.e., nonlinear coordinate transformation) can render a nonconvex function convex. This type of invertible, nonlinear, and smooth coordinate transformation will be a key component in our \ECL{} framework.    

\begin{example} \label{example:academic}
Consider a bivariate function 
\begin{subequations}
\begin{equation} \label{eq:academic-example}
    f_1(x_1,x_2) = \left(\frac{x_2}{x_1}-2\right)^2 + (x_2-1)^2, \qquad \mathrm{dom}(f_1)=\{x \in \mathbb{R}^2\mid x_1 > 0,x_2>0\}.  
\end{equation}
It is not difficult to see that this function $f_1$ is smooth but nonconvex (see \Cref{fig:ex-a-1}).~It is also clear that its global minimizer is $x^\ast = (0.5,1)$ by inspection. In fact, we can define a diffeomorphism 
\begin{equation}  \label{eq:academic-example-1}
g(x_1,x_2) := (x_2/x_1,x_2), \qquad \forall x_1 > 0,x_2>0,     
\end{equation} 
which leads to   
\begin{equation} \label{eq:academic-example-2}
    h_1(y_1,y_2):=f_1(g^{-1}(y_1,y_2)) = \left(y_1-2\right)^2 + (y_2-1)^2, \qquad \forall y_1 > 0, y_2 >0.  
\end{equation}
\end{subequations}
It is obvious that $h_1$ is strongly convex with the global minimizer as $y^* = (2,1)$, confirming the  global minimizer $x^\ast = (1/2,1)$ for $f_1$ in \cref{eq:academic-example}. We illustrate the shape of $h_1$ in \cref{fig:ex-a-2}.

Consider another nonsmooth and nonconvex bivariate function (see \cref{fig:ex-b-1}) 
\begin{subequations}
\begin{equation} \label{eq:nonsmooth-example-1}
    f_2(x_1,x_2) = \left|\frac{x_2}{x_1}-2\right| + |x_2-1|, \qquad \mathrm{dom}(f_2)=\{x \in \mathbb{R}^2\mid x_1 > 0,x_2>0\}. 
\end{equation}
By the same diffeomorphism \cref{eq:academic-example-1}, we have
\begin{equation} \label{eq:nonsmooth-example-2}
    h_2(y_1,y_2):=f_2(g^{-1}(y_2,y_2)) = \left|y_1-2\right| + |y_2-1|, \qquad \forall y_1 > 0, y_2 >0,  
\end{equation}
\end{subequations}
which becomes convex (but remains nonsmooth); see \cref{fig:ex-b-2}.  
Despite being nonconvex, both $f_1$ and $f_2$ enjoy hidden convexity properties, revealed by the diffeomorphism \cref{eq:academic-example-1}.
\hfill $\square$
\end{example}

While \cref{example:academic} serves an academic purpose and is purely illustrative, it is important to highlight that invertible and smooth coordinate transformations, similar to \cref{eq:academic-example-1}, indeed appear very often in feedback controller synthesis. We will detail several benchmark applications in \cref{section:applications}.

Before presenting those formally, consider a linear time-invariant (LTI) dynamical system 
\begin{equation} \label{eq:LTI-example}
        \dot{x}(t) = Ax(t) + Bu(t) + w(t),
    \end{equation}
    where $x(t) \in \mathbb{R}^n, u(t) \in \mathbb{R}^m, w(t) \in \mathbb{R}^n$ are system state, control input, and disturbance, respectively. 
A central control task is to design a suitable $u(t)$ to regulate the behavior of $x(t)$ in the presence of $w(t)$. For illustration, we next consider two simple benchmark control cases: 1) linear quadratic regulator (LQR) with stochastic noises, and   2) $\mathcal{H}_\infty$ robust control with adversarial noises. 

\begin{example}[A simple LQR instance] \label{example:LQR_ex}
    We here consider an LTI system \cref{eq:LTI-example} with problem data
    $$
    A = \begin{bmatrix}
        -2 & 0 \\ 0 & 1
    \end{bmatrix}, \; B = \begin{bmatrix}
        0 \\ 1
    \end{bmatrix},
    $$
    where $w(t)$ is a white Gaussian noise with an intensity matrix $\mathbb{E}[w(t)w(\tau)]=4I_2\delta(t-\tau)$. We aim to design a stabilizing state feedback policy $u(t) = kx(t)$ where $k = \begin{bmatrix}
        k_1 & k_2
    \end{bmatrix} \in \mathbb{R}^{1 \times 2}$ to minimize an average mean performance $\lim_{T \rightarrow \infty }\mathbb{E} \!\left[\frac{1}{T}\int_{0}^T \left(x_1(t)^2 + x_2(t)^2 + u(t)^2\right)dt\right]$.

    \begin{subequations}
    After simple algebra, we derive the LQR cost function in terms of $(k_1,k_2)$ as 
    \begin{equation} \label{eq:LQR-ex-1}
        J(k_1,k_2) = \frac{1-2k_2+3k_2^2 - 2k_2^3 - 2k_1^2k_2}{k_2^2 - 1}, \qquad \forall k_1\in \mathbb{R}, k_2 < - 1. 
    \end{equation}
    It is not obvious from \cref{eq:LQR-ex-1} to see whether it is convex or admits a convex parameterization\footnote{With some efforts, one can show that \cref{eq:LQR-ex-1} is actually convex.}.  
    However, following a classical change of variables (see \Cref{appendix:examples-of-ECL} for more computational details),
    we define a nonlinear mapping as 
    \begin{equation} \label{eq:LQR-ex-1-mapping}
    \begin{aligned}
  y=(y_1,y_2)=g(k_1,k_2):= \left(\frac{k_1}{1-k_2},  \frac{2k_2 - k_1^2 - 2k_2^2 }{k_2^2 - 1}\right) \qquad \forall k_1 \in \mathbb{R}, k_2 < -1. 
    \end{aligned}
    \end{equation}
This mapping turns out to be invertible, and it can be further verified that 
\begin{equation} \label{eq:LQR-ex-1-y}
    h(y_1,y_2): = J(g^{-1}(y_1,y_2)) = -y_2 -1 + y^\tr \begin{bmatrix}
         1 &  y_1 \\
y_1& -y_2 - 2
    \end{bmatrix}^{-1} y, \qquad \forall \begin{bmatrix}
         1 &  y_1 \\
y_1& -y_2 - 2
    \end{bmatrix} \succ 0. 
\end{equation}
It is now clear that $h$ is convex since its epigraph is convex, as shown below: 
    \begingroup
    \setlength\arraycolsep{2pt}
\def\arraystretch{0.9}
    \begin{equation} \label{eq:LQR-ex-1-cvx}
    \left\{(y,\gamma) \in \mathbb{R}^3 \left| h(y) \leq \gamma \right.\right\} = \left\{(y,\gamma) \in \mathbb{R}^3 \left|   \begin{bmatrix}
        \gamma +y_2 +1 & y^\tr \\
        y & \mathrm{aff}(y)
    \end{bmatrix} \succeq 0, \! 
    \mathrm{aff}(y) \!:=\! \begin{bmatrix}
         1 &  y_1 \\
y_1& -y_2 - 2
    \end{bmatrix} \!\succ\! 0
        \right.\right\},
    \end{equation}
    \endgroup
    \end{subequations}
    where we have applied the Schur complement. 
    Thanks to the smooth and invertible mapping \cref{eq:LQR-ex-1-mapping}, we see that minimizing \cref{eq:LQR-ex-1} is equivalent to a convex problem of minimizing \cref{eq:LQR-ex-1-y}. Consequently, any stationary point of \cref{eq:LQR-ex-1} is globally optimal (in fact, the stationary point of \cref{eq:LQR-ex-1} is unique). \Cref{fig:simple-example-LQR-main-paper} illustrates these two functions \Cref{eq:LQR-ex-1,eq:LQR-ex-1-y}, both of which are indeed convex.
    \hfill $\qed$
\end{example}

    \begin{figure}[t]
    \centering
    \setlength{\abovecaptionskip}{4pt}
    \begin{subfigure}{0.35\textwidth}
        \includegraphics[width=0.8\textwidth]{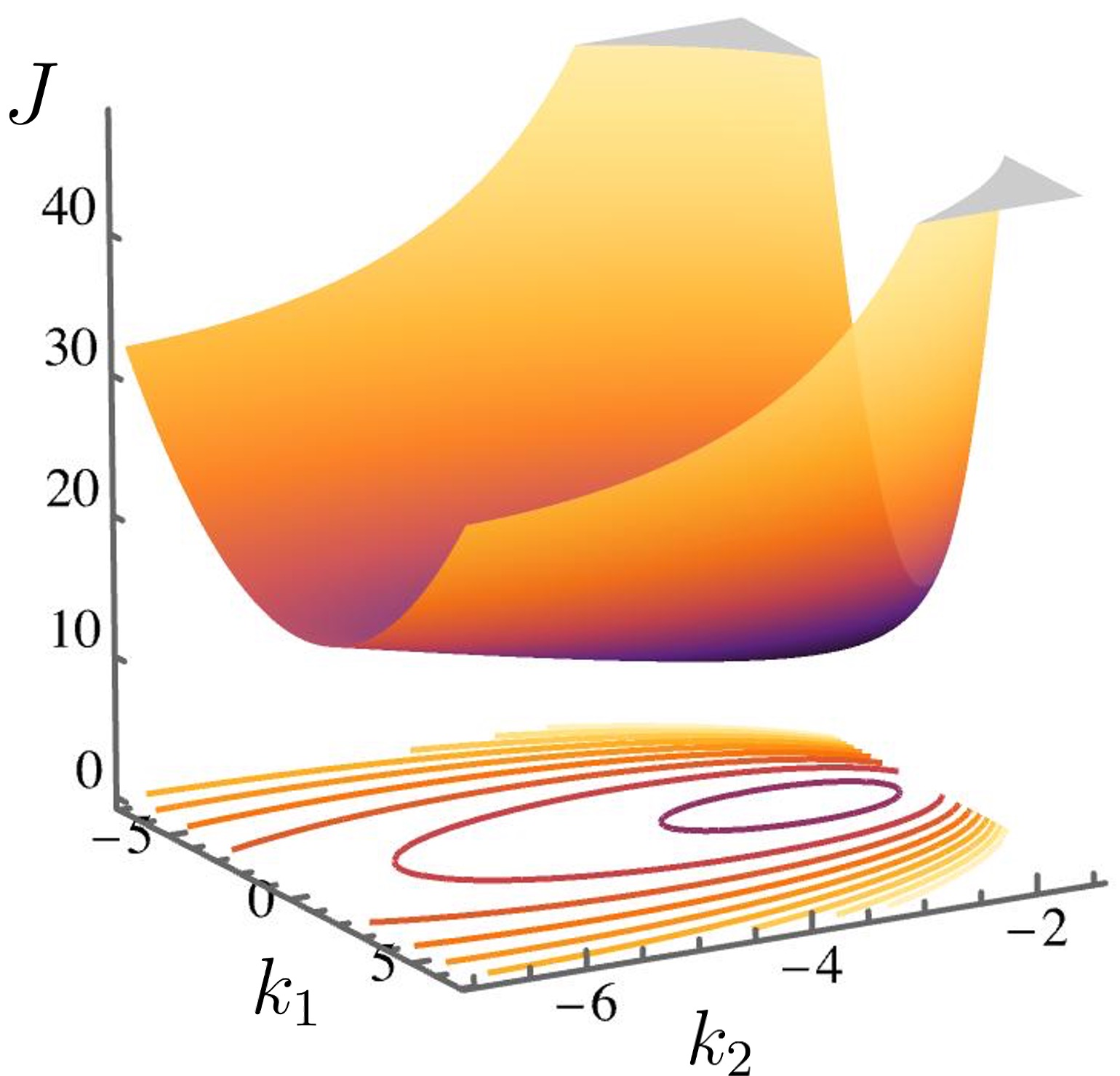}
        \caption{$J(k_1,k_2)$}     \label{fig:LQR-before}
    \end{subfigure}
    \hspace{10mm}
    \begin{subfigure}{0.35\textwidth}
        \includegraphics[width=0.8
\textwidth]{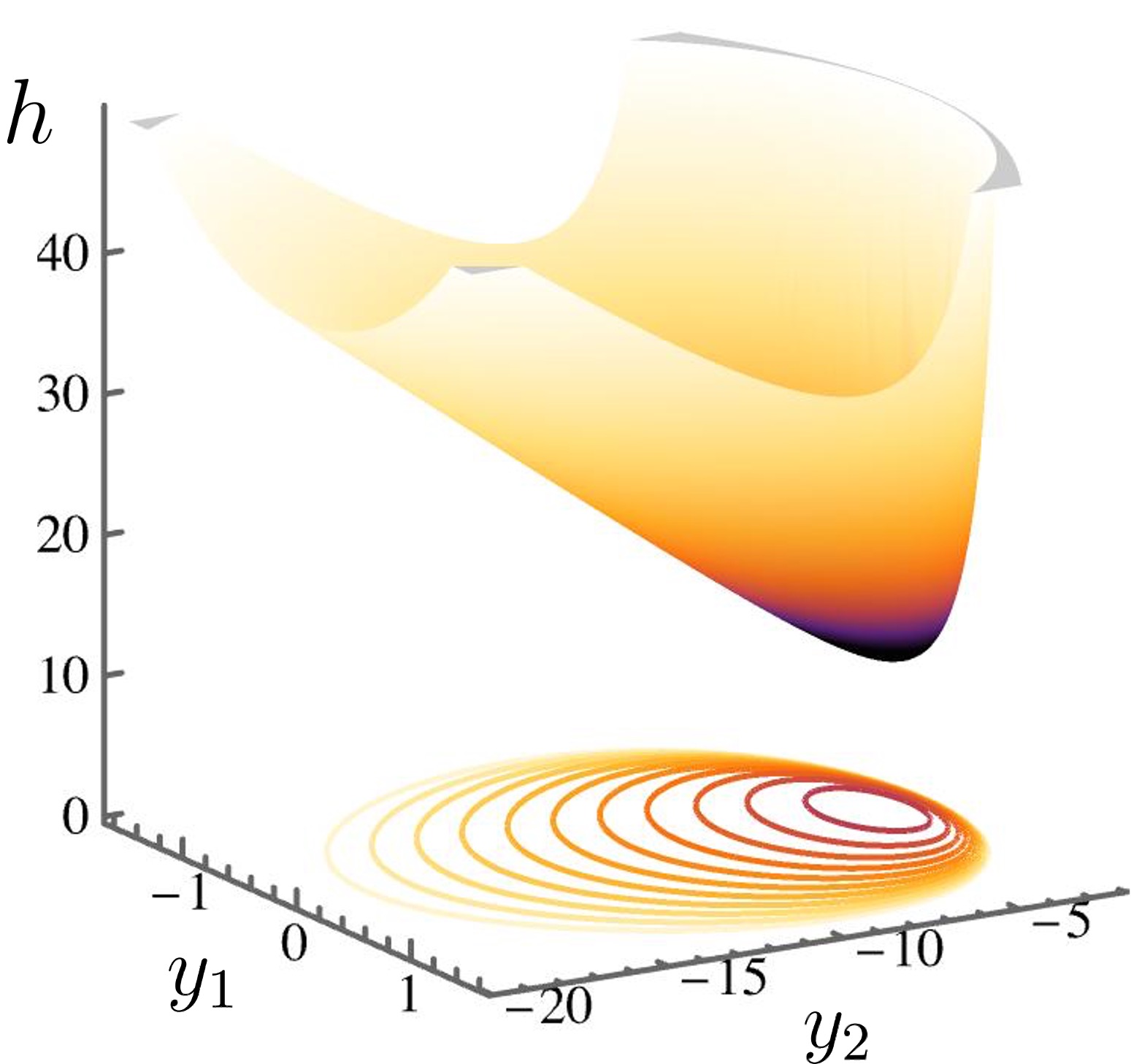}
        \caption{$h(y_1,y_2)$}     \label{fig:LQR-after}
    \end{subfigure}
    \caption{Optimization landscape of the LQR instance in \Cref{example:LQR_ex}: (a) the original LQR cost \cref{eq:LQR-ex-1}; (b) the cost function after convexification \cref{eq:LQR-ex-1-y}.
    }
    \label{fig:simple-example-LQR-main-paper}
\end{figure}

\begin{example}[A simple $\mathcal{H}_\infty$ instance] \label{example:Hinf}
    We consider another LTI system \cref{eq:LTI-example} with problem data
    $$
    A = -1, \; B = 1.
    $$
    Unlike the stochastic noise in \cref{example:LQR_ex}, we here consider an adversarial disturbance $w$ with bounded energy $\|w\|_{\ell_2}^2:=\int_0^\infty \|w(t)\|^2_2 dt \leq 1$. We aim to design a stabilizing state feedback policy $u(t) = kx(t)$ where $k \in \mathbb{R}$ to minimize the worst-case performance, given by the square root of $\max_{\|w\|_{\ell_2} \leq 1}  \int_{0}^\infty \left(0.1x(t)^2 + u(t)^2\right)dt$. From classical control theory, this cost function is the same as the $\mathcal{H}_\infty$ norm of the transfer function from $w$ to $z:= \begin{bmatrix}
        \sqrt{0.1}x \ \ kx
    \end{bmatrix}^\tr$. Some simple calculations lead to an analytical expression of the cost function
    \begin{equation}
\begin{aligned} \label{eq:Hinf-ex-2}
    J_\infty(k)=\|\mathbf{T}_{zw}(k)\|_{\mathcal{H}_\infty}=\frac{\sqrt{0.1+k^2}}{1-k},\quad \forall k<1.
\end{aligned} 
    \end{equation} 
In this case, the non-convexity of $J_\infty$ can be seen from the fact that the second derivative $J_\infty''(k)<0$ for some $k<1$.
Indeed, we can observe from \Cref{fig:Hinf-before} that $J_\infty$ is nonconvex. Note that this $\mathcal{H}_\infty$ cost function $J_\infty$ is also not coercive since $\lim_{k\to -\infty} J_\infty(k) = 1$.

Still, classical control theory \cite{boyd1991linear} has revealed the hidden convex structure in $J_\infty$, i.e., we have
\[
 \min_{k<1}\, J_\infty(k)=\min_{(\gamma, y, x)\in\mathcal{F}_{\infty}} \, \gamma,
 \]
where
 $   \mathcal{F}_{\infty}
=\left\{
(\gamma, y, x):x>0,\ 2\gamma(x-y)-(0.1x^2+y^2+1)\geq0\right\}$ is a convex set.  
The convexity can be seen from the equivalence between the second inequality in $\mathcal{F}_{\infty}$ and the LMI
\[
\begin{bmatrix}
    -2x+2y & 1 & \sqrt{0.1}x & y\\
    1 & -\gamma & 0 & 0\\
    \sqrt{0.1}x & 0 & -\gamma & 0\\
    y & 0 & 0 & -\gamma
\end{bmatrix} \preceq 0.
\]
In fact, the convex set 
$\mathcal{F}_{\infty}$ is derived from a classical change of variables $(p,k)\rightarrow(y,x):$
\begin{equation} \label{eq:Hinf-simple-example-change}
    y=kp^{-1},\quad x=p^{-1}, \qquad \forall p > 0, 
\end{equation}
where $p$ is an extra Lyapunov variable from the bounded real lemma (more details will be provided in \Cref{subsection:hinf-state-feedback}). Thus, we can solve the nonconvex $\mathcal{H}_\infty$ control with \cref{eq:Hinf-ex-2} via convex optimization. The convex set $\mathcal{F}_\infty$ is illustrated in \cref{fig:Hinf-Fcvx}, where the red dot denotes the optimal solution
$
(\gamma^\star,y^\star,x^\star)\approx(0.3015,-0.3015,3.015)
$ 
to the convex problem $\min_{(\gamma, y, x)\in\mathcal{F}_{\infty}} \, \gamma$ and $\gamma^\star\approx0.3015$ corresponds to the optimal value of the original nonconvex problem $\min_{k<1}\, J_\infty(k)$. Furthermore, since the change of variables in \cref{eq:Hinf-simple-example-change} is invertible, one can recover the stationary point $k^\star=-0.1$ of $J_\infty$ from $(y^\star,x^\star)$ by $k^\star=y^\star/x^\star$.
\hfill $\qed$
\end{example}

\begin{figure}[t]
    \centering
    \setlength{\abovecaptionskip}{4pt}
    \begin{subfigure}{0.35\textwidth}
        \includegraphics[width=0.8\textwidth]{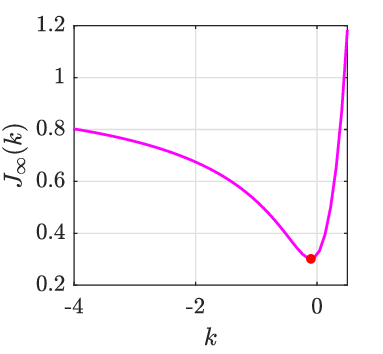}
        \caption{$J_\infty(k)$}     \label{fig:Hinf-before}
    \end{subfigure}
    \hspace{10mm}
    \begin{subfigure}{0.35\textwidth}
        \includegraphics[width=0.9
\textwidth]{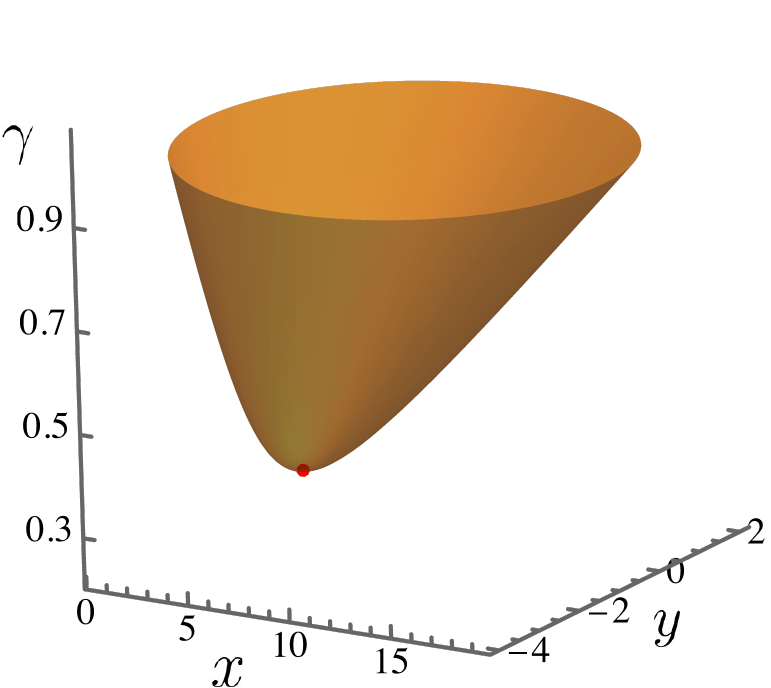}
        \caption{$\mathcal{F}_{\infty}$}     \label{fig:Hinf-Fcvx}
    \end{subfigure}
    \caption{Optimization landscape of the $\mathcal{H}_\infty$ instance in \Cref{example:Hinf}: (a) the $\mathcal{H}_\infty$ cost \cref{eq:Hinf-ex-2}, where the red dot denotes the stationary point $k^\star=-0.1$; (b) the convex set $\mathcal{F}_{\infty}$ after the lifting procedure, where the red dot denotes the optimal solution $(\gamma^\star,y^\star,x^\star)$ of the convex problem $\min_{(\gamma, y, x)\in\mathcal{F}_{\infty}} \, \gamma$.
    }
    \label{fig:simple-example-LQR}
\end{figure}

Unlike the LQR case in \cref{example:LQR_ex}, a direct convex parameterization for $\mathcal{H}_\infty$ cost function \Cref{eq:Hinf-ex-2} in \cref{example:Hinf} is not straightforward, and instead, we have introduced two extra variables ${p}>0$ and $\gamma$ to derive a convex representation for the nonconvex epigraph of \Cref{eq:Hinf-ex-2}. Similar to \cref{example:LQR_ex,example:academic}, the existence of such a lifted convex representation can further imply the global optimality of certain stationary points. More details will be provided in \Cref{section:ECL}.

\subsection{A Simple Fact on Epigraphs} \label{subsection:epigraphs}

All instances in \cref{example:academic,example:LQR_ex,example:Hinf} involve a diffeomorphism, which allows for an explicit and equivalent convex representation in a new set of coordinates. We here present a simple~fact~summarizing the key features in the previous examples from an epigraph perspective.

The fact below guarantees any stationary points are globally optimal for a class of continuous functions when their non-strict epigraphs admit \textit{direct convex reparameterization}. 

\begin{fact} \label{fact:epigraph-lift}
Let $f: \mathcal{D} \rightarrow \mathbb{R}$ be a continuous function.
Denote its non-strict epigraph by
$
\operatorname{epi}_\geq(f) := \{(x,\gamma) \in \mathcal{D}\times\mathbb{R} \mid \gamma \geq f(x)\}. 
$ 
Suppose that 
\begin{itemize}
    \item 
    there exists a convex set $\mathcal{F}_{\mathrm{cvx}} \subset \mathbb{R}^{d}\times\mathbb{R}$ and a diffeomorphism $\Phi$ between $\operatorname{epi}_\geq(f)$ and $\mathcal{F}_{\mathrm{cvx}}$;
    \item $\Phi$ further satisfies $(y,\gamma) = \Phi(x,\gamma), \; \forall (x,\gamma) \in \operatorname{epi}_\geq(f)$ (i.e., the mapping $\Phi$ directly outputs $\gamma$). 
\end{itemize}
Then the following statements hold.
\begin{enumerate}
    \item The minimization of $f(x)$ over $x\in\mathcal{D}$ is equivalent to a convex problem in the sense that
    \[
        \inf_{x\in\mathcal{D}} f(x) = \inf_{(y,\gamma)\in\mathcal{F}_{\mathrm{cvx}}} \gamma.
    \]
    \item {Let $\mathcal{D}$ be an open domain and $f(x)$ be subdifferentially regular.\footnote{This is a very large class of continuous functions, covering all optimal and robust control problems in \Cref{section:applications}. Also, see Part I of this paper \cite[Appendix B]{zheng2023benign} for the background of subdifferential regularity and Clarke stationarity.}} If $x^\ast$ is a Clarke stationary point (i.e., $0 \in \partial f(x^\ast)$), then $x^\ast$ is globally optimal to $f(x)$ over $x\in\mathcal{D}$.
\end{enumerate}
\end{fact}

The first statement in \Cref{fact:epigraph-lift} is easy to see from the following facts
$$
\begin{aligned}
    \inf_{x\in\mathcal{D}} f(x) = \inf_{(x,\gamma)\in \operatorname{epi}_\geq(f)} \gamma 
    = \inf_{(y,\gamma)\in\mathcal{F}_{\mathrm{cvx}}} \gamma,
\end{aligned}
$$
where the first equality is due to the definition of epigraphs and the second equality is directly from the diffeomorphism $\Phi$. The second statement in \Cref{fact:epigraph-lift} is closely related to the fact that a locally optimal point of a convex optimization problem is also globally optimal. We postpone the rigorous proof of \Cref{fact:epigraph-lift} to \Cref{section:ECL}, as a special case of our general \ECL{} framework.

Note that the epigraph form in \Cref{fact:epigraph-lift} is a standard trick and has been widely used in convex optimization \cite{boyd2004convex}. Also, \Cref{fact:epigraph-lift} works for nonsmooth and nonconvex functions. It is clear that both \cref{example:academic,example:LQR_ex} are special cases of \Cref{fact:epigraph-lift}. For illustration, in \Cref{example:academic}, we have 
    $$
    \Phi(x,\gamma) = (x_2/x_1,x_2,\gamma), \qquad \forall (x,\gamma) \in \operatorname{epi}_\geq(f), 
    $$
while in \Cref{example:LQR_ex}, we have (recall $g(k)$ in \cref{eq:LQR-ex-1-mapping})
    $$
    \Phi(k,\gamma) = (g(k),\gamma), \qquad \forall (k,\gamma) \in \operatorname{epi}_\geq(J).
    $$
As for \Cref{example:Hinf}, the change of variables $y=kp^{-1}$ is also related to a diffeomorphism. However, this case is not covered by \Cref{fact:epigraph-lift} due to the lifting variable $p$ in \Cref{example:Hinf}. Inspired by \Cref{example:Hinf}, we will introduce a much more generalized version of \Cref{fact:epigraph-lift} in \Cref{section:ECL}.

\begin{remark}[LMI/semidefinite representable sets]
In \Cref{example:LQR_ex}, the set $\mathcal{F}_{\mathrm{cvx}}$ shown in \cref{eq:LQR-ex-1-cvx} is not only convex but also represented by an LMI. This class of convex sets is called LMI representable or semidefinite representable sets \cite{helton2010semidefinite}. Thus, the minimization of $\gamma$ over $(y,\gamma)\in\mathcal{F}_{\mathrm{cvx}}$ is an LMI (or semidefinite programs), which is ready to be solved using many existing conic solvers (even in a large scale when special structures appear~\cite{zheng2021chordal}).
\hfill \qed
\end{remark}

%% file: sec_ECL.tex
\section{Extended Convex Lifting (\ECL) for Benign Non-convexity} \label{section:ECL}

In this section, we present a generic template for extended convex reparameterization of epigraphs of nonconvex functions, which significantly generalizes the setup in \Cref{fact:epigraph-lift}. We call this generic template \texttt{Extended Convex Lifting (ECL)}: We first lift its (non)-strict epigraph to a higher dimensional set by appending auxiliary variables, and then assume a diffeomorphism $\Phi$ that maps this lifted nonconvex set to a ``partially'' convex set (see \Cref{fig:ECL-process} for illustration). 

Despite non-convexity and non-smoothness, the existence of an \ECL{} not only reveals that minimizing the original function is equivalent to a convex problem (\Cref{theorem:convex-equivalency}) but also identifies a class of first-order stationary points to be globally optimal (\Cref{theorem:ECL-guarantee}).

\subsection{A Framework of Extended Convex Lifting (\texttt{ECL})} \label{subsection:ECL}

For a real-valued function $f$ with domain $\mathcal{D} \subseteq \mathbb{R}^d$, we define its strict and non-strict epigraphs by
\begin{align*}
    \operatorname{epi}_>(f) &\coloneqq \{(x,\gamma)\in\mathcal{D}\times\mathbb{R}\mid
    \gamma>f(x)
    \}, \\  
    \operatorname{epi}_\geq(f) &\coloneqq \{(x,\gamma)\in\mathcal{D}\times\mathbb{R}\mid
    \gamma\geq f(x)
    \}. 
\end{align*}

\begin{definition}[Extended Convex Lifting] \label{definition:ECL}
Let $f:\mathcal{D}\rightarrow \mathbb{R}$ be a continuous function, where $\mathcal{D}\subseteq\mathbb{R}^{d}$. We say that the tuple $(\mathcal{L}_{\mathrm{lft}},\mathcal{F}_{\mathrm{cvx}},\mathcal{G}_{\mathrm{aux}},\Phi)$ is an \ECL{} of $f$, if the following conditions hold:
\begin{subequations}
\begin{enumerate}
\item $\mathcal{L}_{\mathrm{lft}} \subseteq \mathbb{R}^{d}\times\mathbb{R}\times\mathbb{R}^{d_\xi}$ is a lifted set with an extra variable $\xi \in \mathbb{R}^{d_\xi}$, such that the canonical projection of $\mathcal{L}_{\mathrm{lft}}$ onto the first $d+1$ coordinates, given by
$ 
    \pi_{x,\gamma}(\mathcal{L}_{\mathrm{lft}})
= \{(x,\gamma):\exists \xi \in\mathbb{R}^{d_{\xi}} \; \text{s.t.}\; (x,\gamma,\xi)\in \mathcal{L}_{\mathrm{lft}}\},
$ 
satisfies
\begin{equation} \label{eq:epi-graph-lifting}
\operatorname{epi}_>(f)
\subseteq \pi_{x,\gamma}(\mathcal{L}_{\mathrm{lft}})\subseteq 
\operatorname{cl}\operatorname{epi}_{\geq} (f).
\end{equation}

\item $\mathcal{F}_{\mathrm{cvx}} \subseteq \mathbb{R}\times \mathbb{R}^{d_1} $ is a convex set, $\mathcal{G}_{\mathrm{aux}} \subseteq \mathbb{R}^{d_2} $ is an auxiliary set, and $\Phi$ is a $C^2$ diffeomorphism from $\mathcal{L}_{\mathrm{lft}}$ to $\mathcal{F}_{\mathrm{cvx}}\times\mathcal{G}_{\mathrm{aux}}$.%
\footnote{Precisely, this means that $\Phi$ can be extended to a $C^2$ function $\tilde{\Phi}$ defined on some open subset $U\subseteq\mathbb{R}^d\times\mathbb{R}\times\mathbb{R}^{d_\xi}$ with $U\supseteq \mathcal{L}_{\mathrm{lft}}$, such that $\tilde{\Phi}$ is a bijection from $U$ to $\tilde{\Phi}(U)$ and its inverse $\tilde{\Phi}^{-1}$ is $C^2$.

The auxiliary set $\mathcal{G}_{\mathrm{aux}}$ can be nonconvex or even disconnected. We also allow $d_2=0$, in which case we adopt the convention that $\mathcal{G}_{\mathrm{aux}} = \mathbb{R}^0 = \{0\}$, and identify the set $\mathcal{F}_{\mathrm{cvx}}\times \{0\}$ with $\mathcal{F}_{\mathrm{cvx}}$.}

\item For any $(x,\gamma,\xi)\in \mathcal{L}_{\mathrm{lft}}$, we have
\begin{equation} \label{eq:mapping-invariant}
\Phi(x,\gamma,\xi) = (\gamma, \zeta_1,\zeta_2)
\quad \text{and} \quad
(\gamma,\zeta_1)\in \mathcal{F}_{\mathrm{cvx}}
\end{equation}
for some $\zeta_1\in\mathbb{R}^{d_1}$ and $\zeta_2\in\mathcal{G}_{\mathrm{aux}}$ (i.e., the mapping $\Phi$ directly outputs $\gamma$ in the first component).
\end{enumerate}
\end{subequations}
\end{definition}

Compared to the basic \Cref{fact:epigraph-lift}, the~notion of \texttt{ECL} gives more flexibility by 1) introducing an extra variable to \emph{lift} the epigraph to a higher dimension, 2) relaxing the projection of the lifted set to sit between the strict epigraph and the closure of the non-strict epigraph (\emph{chain of inclusion}), and 3) adding an auxiliary set $\mathcal{G}_{\mathrm{aux}}$ to extend the convex image under a diffeomorphism (\textit{extended convex parameterization}). The setup in \Cref{fact:epigraph-lift} is a special case of \texttt{ECL} by letting $d_\xi = 0$ (no lifting), requiring $\pi_{x,\gamma}(\mathcal{L}_{\mathrm{lft}}) = \operatorname{epi}_\geq (f)$ (transforming exactly the non-strict epigraph), and letting $\mathcal{G}_{\mathrm{aux}} = \{0\}$ (no auxiliary set). 

The interested reader may wonder why we need such a peculiar chain of inclusion in \cref{eq:epi-graph-lifting}. A simpler and more straightforward requirement for the lifting process might be 
\begin{equation} \label{eq:projection-ECL-alternative}
\pi_{x,\gamma}(\mathcal{L}_{\mathrm{lft}}) = \operatorname{epi}_\geq (f).
\end{equation}
Evidently, \cref{eq:epi-graph-lifting} includes \cref{eq:projection-ECL-alternative} as a special case. Note that \cref{eq:epi-graph-lifting} further allows for the case $\operatorname{epi}_\geq (f) \subseteq \pi_{x,\gamma}(\mathcal{L}_{\mathrm{lft}}) \subseteq \mathrm{cl}\,\operatorname{epi}_\geq (f) $ when the non-strict epigraph $\operatorname{epi}_\geq (f)$ is not closed; see \cref{appendix:closure-strict-epigraph} for conditions when $\operatorname{epi}_\geq (f)$ is closed. 
We will see in \Cref{section:applications} that one indeed needs to employ the more general condition \cref{eq:epi-graph-lifting} for many policy optimization problems in control (e.g., LQG and $\mathcal{H}_\infty$ output-feedback control \cite[Facts 4.1 \& 5.2]{zheng2023benign}).
In fact, one technical difficulty of establishing the main results in Part I of this paper \cite[Theorems 4.2 \& 5.2]{zheng2023benign} lies in the proof of \cref{eq:epi-graph-lifting} for a diffeomorphism arising in the classical change of variables (cf. \Cref{lemma:inclusion-projection-LQG,lemma:inclusion-projection-Hinf}).

An almost immediate benefit from the \ECL{}  $(\mathcal{L}_{\mathrm{lft}},\mathcal{F}_{\mathrm{cvx}},\mathcal{G}_{\mathrm{aux}},\Phi)$ is that we can reformulate the minimization of $f(x)$ over $x\in\mathcal{D}$ as a convex optimization problem.
\begin{theorem} \label{theorem:convex-equivalency}
Let $f:\mathcal{D}\rightarrow\mathbb{R}$ be a continuous function equipped with an \ECL{} $(\mathcal{L}_{\mathrm{lft}},\mathcal{F}_{\mathrm{cvx}},\mathcal{G}_{\mathrm{aux}},\Phi)$. Then, we have 
\begin{equation*}
\inf_{x\in\mathcal{D}} f(x) = \inf_{(\gamma,\zeta_1)\in\mathcal{F}_{\mathrm{cvx}}} \gamma.
\end{equation*}
\end{theorem}
\begin{proof}
Let $x\in\mathcal{D}$ and $\epsilon>0$ be arbitrary. We have $(x,f(x)+\epsilon)\in \operatorname{epi}_>(f)$ by definition. By~\eqref{eq:epi-graph-lifting}, there exists $\xi\in\mathbb{R}^{d_\xi}$ such that $(x,f(x)+\epsilon,\xi)\in \mathcal{L}_{\mathrm{lft}}$. Then by~\eqref{eq:mapping-invariant}, we can find $\zeta_1\in\mathbb{R}^{d_1}$ such that $(f(x)+\epsilon,\zeta_1)\in\mathcal{F}_{\mathrm{cvx}}$, which implies $f(x)+\epsilon\geq \inf_{(\gamma,\zeta_1)\in\mathcal{F}_{\mathrm{cvx}}} \gamma $. By the arbitrariness of $x\in\mathcal{D}$ and $\epsilon>0$, we get
\[
\inf_{x\in\mathcal{D}} f(x) \geq \inf_{(\gamma,\zeta_1)\in\mathcal{F}_{\mathrm{cvx}}} \gamma.
\]
To show the other direction, let $(\gamma,\zeta_1)\in\mathcal{F}_{\mathrm{cvx}}$ be arbitrary, and pick any $\zeta_2\in\mathcal{G}_{\mathrm{aux}}$. Noting that $\Phi$ is a diffeomorphism from $\mathcal{L}_{\mathrm{lft}}$ to $\mathcal{F}_{\mathrm{cvx}}\times\mathcal{G}_{\mathrm{aux}}$. Using~\eqref{eq:mapping-invariant},  we can get $$
\Phi^{-1}(\gamma,\zeta_1,\zeta_2)=(x,\gamma,\xi)\in\mathcal{L}_{\mathrm{lft}}
$$
for some $x\in\mathbb{R}^d$ and $\xi\in\mathbb{R}^{d_{\xi}}$. Then~\eqref{eq:epi-graph-lifting} implies $(x,\gamma)\in \operatorname{cl}\operatorname{epi}_{\geq}(f)$, which means that for any $\epsilon>0$ we can find $x_\epsilon\in\mathcal{D}$ and $\gamma_\epsilon\geq f(x_\epsilon)$ such that $\|x_\epsilon-x\|<\epsilon$ and $|\gamma_\epsilon-\gamma|<\epsilon$. Consequently,
\[
f(x_\epsilon)\leq \gamma_\epsilon \leq |\gamma_\epsilon-\gamma| + \gamma
< \gamma + \epsilon,
\]
which further implies $\gamma+\epsilon\geq \inf_{x\in\mathcal{D}}f(x)$. By the arbitrariness of $\epsilon>0$ and $(\gamma,\zeta_1)\in\mathcal{F}_{\mathrm{cvx}}$, we see that
\[
\inf_{(\gamma,\zeta_1)\in\mathcal{F}_{\mathrm{cvx}}} \gamma\geq \inf_{x\in\mathcal{D}} f(x),
\]
and the proof is complete.
\end{proof}

In \cref{theorem:convex-equivalency}, the function $f$ can be nonsmooth and nonconvex (see \cref{example:academic,example:LQR_ex,example:Hinf}), but the existence of an \ECL{} reveals its hidden convexity, and consequently,  optimizing $f(x)$ over $x\in\mathcal{D}$ is equivalent to a convex problem. In many control applications, the convex set $\mathcal{F}_{\mathrm{cvx}}$ is represented by certain LMIs, and  \Cref{theorem:convex-equivalency} can indeed be considered as the rationale behind many LMI formulations in control \cite{scherer2000linear,boyd1994linear}. In~\Cref{section:applications}, we will provide specific examples of the \ECL{}s corresponding to various control problems.

\begin{remark}[Lifting and auxiliary set in \ECL{}] 
We shall see in \Cref{section:applications} that the flexibility of \ECL{} is crucial for the existence of such a  $C^2$-diffeomorphism in many control problems, especially dynamic output control such as LQG and $\mathcal{H}_\infty$ control. Indeed, we need to use the closure of non-strict epigraphs to account for non-strict  LMIs for $\mathcal{H}_2$ and $\mathcal{H}_\infty$ norms. The additional lifting variable $\xi$ often corresponds to Lyapunov variables, and the auxiliary set $\mathcal{G}_{\mathrm{aux}}$ will represent the set of similarity transformations in dynamic policies. For policy optimization with static state feedback, no auxiliary set is needed (since there exist no similarity transformations in static feedback policies), and in this case, $\mathcal{G}_{\mathrm{aux}}$ will be $\{0\}$. Still, the lifting variable $\xi$ is useful for static policy optimization since Lyapunov variables are naturally involved, as already previewed in \Cref{example:Hinf}.    
\hfill $\square$
\end{remark}

\begin{remark}[When the infimum is achieved]
Note that in an \ECL{}, the sets $\mathcal{D}$ and $\mathcal{F}_{\mathrm{cvx}}$ are not necessarily closed or bounded. No guarantees are provided in \Cref{theorem:convex-equivalency} when their infimums are achieved. In fact,~almost~all existing LMI formulations in controller~synthesis from \cite{scherer2000linear} rely on strict LMIs where the corresponding $\mathcal{F}_{\mathrm{cvx}}$ is not closed and the infimum $\inf_{(\gamma,\zeta_1)\in\mathcal{F}_{\mathrm{cvx}}} \gamma$ typically cannot be achieved. Our control applications in \Cref{section:applications} will focus on non-strict LMIs, as advocated in Part I of this paper \cite{zheng2023benign}. Even in the case of non-strict LMIs, the infimum may not be attained and can only be approached as the variable tends to infinity; we will provide an $\mathcal{H}_\infty$ control instance in \Cref{subsection:non-coercivity} as one such example.
\hfill $\square$
\end{remark}

\subsection{Non-degenerate Stationarity Implies Global Optimality}

In addition to convex re-parameterization in terms of $\mathcal{F}_{\mathrm{cvx}}$ as shown in \Cref{theorem:convex-equivalency}, the existence of an \ECL{} can further reveal \textit{global optimality of certain first-order stationary points} for the original function $f$ that can potentially be nonconvex or nonsmooth. This allows us to optimize $f(x)$ over $x\in\mathcal{D}$ by direct local search without relying on the \ECL{}, since we may only know its existence but not its particular form. This feature is particularly important for learning-based model-free control applications \cite{hu2023toward,Talebi2024geometry}.   

Before proceeding, we here re-emphasize that the chain of inclusion \cref{eq:epi-graph-lifting} is critical to the existence of the diffeomorphism $\Phi$ in many control problems. This might be considered reminiscent of the subtleties 
between strict and non-strict LMIs (see \Cref{lemma:H2norm,lemma:bounded_real} in \cref{section:applications}).
Particularly, this definition allows existence of points $(x,f(x))$ that are not covered by $\pi_{x,\gamma}(\mathcal{L}_{\mathrm{lft}})$, as well as members of $\pi_{x,\gamma}(\mathcal{L}_{\mathrm{lft}})$ that are only accumulation points of $\operatorname{epi}_\geq(f)$. We introduce the notion of \textit{(non-)degeneracy} to characterize the former type of points.

\begin{definition}[Non-degenerate points] \label{definition:degenerate-points}
Let $f:\mathcal{D}\rightarrow \mathbb{R}$ be a continuous function equipped with an \ECL{} $(\mathcal{L}_{\mathrm{lft}},\mathcal{F}_{\mathrm{cvx}},\mathcal{G}_{\mathrm{aux}},\Phi)$. A point $x\in\mathcal{D}$ is called \textit{non-degenerate} if
$
(x,f(x))\in\pi_{x,\gamma}(\mathcal{L}_{\mathrm{lft}}),
$
otherwise \textit{degenerate}. The set of non-degenerate points in $\mathcal{D}$ will be denoted by $\mathcal{D}_{\mathrm{nd}}$.    
\end{definition}

The notion of non-degenerate points depends on a particular choice of \ECL{} $(\mathcal{L}_{\mathrm{lft}},\mathcal{F}_{\mathrm{cvx}},\mathcal{G}_{\mathrm{aux}},\Phi)$. From \Cref{definition:degenerate-points}, we have a quick fact below.
\begin{fact} \label{fact:non-degenerate-point}
    If \cref{eq:projection-ECL-alternative} holds for the \ECL{} $(\mathcal{L}_{\mathrm{lft}},\mathcal{F}_{\mathrm{cvx}},\mathcal{G}_{\mathrm{aux}},\Phi)$, then all points $x \in \mathcal{D}$ are non-degenerate.
\end{fact}
Thus, it is clear that all points $x \in \mathcal{D}$ for \cref{example:academic,example:LQR_ex} are non-degenerate. Indeed, we will prove in \Cref{subsection:static-policies} that all stabilizing LQR or state-feedback $\mathcal{H}_\infty$ policies are non-degenerate using a standard \ECL{} (cf. \Cref{proposition:LQR-non-degenearte,fact:Hinf-non-degenearte}). 
Moreover, we will see that for LQG and $\mathcal{H}_\infty$ control, the set of non-degenerate dynamic policies  defined in Part I \cite[Defintions 3.1 \& 3.2]{zheng2023benign} indeed corresponds to the set of non-degenerate points in the sense of \Cref{definition:degenerate-points}. 

Despite the generality of \texttt{ECL}, we have almost the same guarantees as \Cref{fact:epigraph-lift}, as summarized in the theorem below, which is one main technical result in this paper. 
The proof idea has a strong geometrical intuition, but the details are technically involved and are postponed to \cref{subsection:proof-ECL}.

\begin{theorem}\label{theorem:ECL-guarantee}
Let $f:\mathcal{D}\rightarrow\mathbb{R}$ be a subdifferentially regular function defined on an open domain $\mathcal{D}\subseteq\mathbb{R}^d$, and let $(\mathcal{L}_{\mathrm{lft}}, \mathcal{F}_{\mathrm{cvx}},\mathcal{G}_{\mathrm{aux}},\Phi)$ be an \ECL{} of $f$. If $x^\ast\in\mathcal{D}_{\mathrm{nd}}$ is a Clarke stationary point, i.e., $0\in\partial f(x^\ast)$, then $x^\ast$ is a global minimizer of $f(x)$ over $\mathcal{D}$.
\end{theorem}

This theorem guarantees that \textit{non-degenerate stationarity implies global optimality} for any subdifferentially regular function with an \ECL{}. We remark that subdifferentially regular functions are a very large class of functions, covering all optimal and robust control problems discussed in \Cref{section:applications} (see \cite[Appendix B]{zheng2023benign} for a review of subdifferential regularity and Clarke stationarity). 

Now it is clear that \Cref{fact:epigraph-lift} becomes a special case of \Cref{theorem:convex-equivalency} and \Cref{theorem:ECL-guarantee} (by choosing $d_\xi = 0$ and $d_2 = 0$, i.e., with no extra variable $\xi$ and no auxiliary set $\mathcal{G}_{\mathrm{aux}}$).
Considering \Cref{fact:non-degenerate-point}, we also have a quick corollary.
\begin{corollary}\label{corollary:ECL-guarantee}
Let $f:\mathcal{D}\rightarrow\mathbb{R}$ be a subdifferentially regular function defined on an open domain $\mathcal{D}\subseteq\mathbb{R}^d$, and let $(\mathcal{L}_{\mathrm{lft}}, \mathcal{F}_{\mathrm{cvx}},\mathcal{G}_{\mathrm{aux}},\Phi)$ be an \ECL{} of $f$. If \cref{eq:projection-ECL-alternative} holds, then any Clarke stationary point is a global minimizer of $f(x)$ over $\mathcal{D}$.
\end{corollary}

\begin{remark}[Degenerate points and saddles] \label{remark:saddle-points}
By definition \cref{eq:epi-graph-lifting}, $\pi_{x,\gamma}(\mathcal{L}_{\mathrm{lft}})$ may not cover the whole non-strict epigraph (i.e., we may have feasible points that $(x, f(x))  \notin \pi_{x,\gamma}(\mathcal{L}_{\mathrm{lft}})$), and thus the convex re-parameterization over $\mathcal{F}_{\mathrm{cvx}}$ may not be able to search over all the feasible points $\mathcal{D}$ (but their infimum are the same as shown in \cref{theorem:convex-equivalency}). Consequently, \cref{theorem:ECL-guarantee} only guarantees the global optimality for \textit{non-degenerate stationary} points. We do not have global optimality guarantees for degenerate stationary points $x \in \mathcal{D}\backslash \mathcal{D}_{\mathrm{nd}}$ since they cannot be covered by convex parameterization.\footnote{In this sense, some classical LMI formulations are not ``equivalent'' convex parameterizations for original control problems, especially in dynamic output feedback cases. This subtle point has been less emphasized in classical literature since most of them focus on suboptimal control policies  \cite{scherer1997multiobjective,scherer2000linear,boyd1994linear}.}
There might exist sub-optimal saddle points for $f$ even when it is equipped with an \ECL{}. Indeed, it has been revealed that there exist strictly sub-optimal saddle points in LQG control \cite[Thereom 5]{zheng2021analysis} \cite[Theorem 2]{zheng2022escaping}, which are all degenerate per \Cref{definition:degenerate-points} (see Part I  \cite{zheng2023benign} for more discussions). These saddle points may even be high-order in the sense that their Hessian are zero and characterizing their local behavior requires high-order (beyond second-order) approximations \cite[Theorem 2]{zheng2022escaping}. We suspect that for policy optimization of many control problems, locally optimal solutions that are not globally optimal \textit{do not exist} even in the set of degenerate points $\mathcal{D}\backslash \mathcal{D}_{\mathrm{nd}}$, but neither a proof nor a counterexample is known yet.
\hfill \qed  
\end{remark}

\begin{remark}[Applications in control] \label{remark:convex-reformulation}
    It is well-known that many state feedback and full-order dynamic feedback synthesis problems are nonconvex in their natural forms, but admit ``convex reparameterization'' in terms of LMIs using a suitable change of variables \cite{gahinet1994linear,scherer2000linear,scherer1997multiobjective,dullerud2013course,boyd1994linear}. We argue that our notion of \texttt{ECL} presents a unified treatment for many of these convex reparameterizations, including LQR, LQG, state-feedback $\mathcal{H}_\infty$ control, and dynamic output-feedback $\mathcal{H}_\infty$ control; the equivalence in \cref{theorem:convex-equivalency} provides the rational behind all of their convex re-parameterizations. Thus, many recent results on global optimality of (non-degenerate) stationary points (such as LQR in \cite{fazel2018global,sun2021learning,mohammadi2021convergence}, LQG in \cite{zheng2021analysis}, state-feedback $\mathcal{H}_\infty$ control in \cite{guo2022global}, Kalman filter in \cite{umenberger2022globally}) are special cases of \cref{theorem:ECL-guarantee} once the corresponding \texttt{ECL} is constructed. However, despite the wide use of convex reparameterization in terms of LMIs in control, exact constructions of \texttt{ECL} require special care, especially in dynamic output feedback. We will present some \texttt{ECL} construction details in \Cref{section:applications}.  \hfill \qed 
\end{remark}

\subsection{Proof of \Cref{theorem:ECL-guarantee}} \label{subsection:proof-ECL}

Basically, we prove \Cref{theorem:ECL-guarantee} by showing its contrapositive: Given $x\in\mathcal{D}_{\mathrm{nd}}$, if there exists $x'\in\mathcal{D}$ such that $f(x')<f(x)$, then $x$ cannot be a Clarke stationary point of $f$ due to the existence of the \ECL{} $(\mathcal{L}_{\mathrm{lft}}, \mathcal{F}_{\mathrm{cvx}},\mathcal{G}_{\mathrm{aux}},\Phi)$. 

On a high level, the proof consists of four steps (see \Cref{fig:proof-illustration} for a graphical illustration):
\begin{itemize}[itemsep=3pt]
\item \textbf{Step 1: Lifting and mapping $(x,f(x))$ and $(x',f(x')+\epsilon)$ into $\mathcal{F}_{\mathrm{cvx}}\times \mathcal{G}_{\mathrm{aux}}$. }
By~the~definition of non-degenerate points, we have $(x,f(x))\in\pi_{x,\gamma} (\mathcal{L}_{\mathrm{lft}})$, while for $\epsilon>0$, we have $(x',f(x')+\epsilon)\in\operatorname{epi}_>(f)\subseteq \pi_{x,\gamma} (\mathcal{L}_{\mathrm{lft}})$. Thus we may lift both of them and then map them into $\mathcal{F}_{\mathrm{cvx}} \times \mathcal{G}_{\mathrm{aux}}$ via the diffeomorphism $\Phi$, to obtain
$(f(x),\zeta_1,\zeta_2)$ and $(f(x')+\epsilon,\zeta_1',\zeta_2')$ that are in $\mathcal{F}_{\mathrm{cvx}} \times \mathcal{G}_{\mathrm{aux}}$.

\item \textbf{Step 2: Constructing a $C^2$ curve in $\mathcal{L}_{\mathrm{lft}}$}. The line segment from $(f(x),\zeta_1,\zeta_2)$ to $(f(x')+\epsilon,\zeta_1',\zeta_2)$ belongs to $\mathcal{F}_{\mathrm{cvx}}\times\mathcal{G}_{\mathrm{aux}}$ thanks to convexity, and also the value from $f(x)$ to $f(x')+\epsilon$ is strictly decreasing as long as $\epsilon$ is sufficiently small (since $f(x)>f(x')$ by assumption). The pre-image of this line segment under $\Phi$ is then a $C^2$ curve in $\mathcal{L}_{\mathrm{lft}}$.

\item \textbf{Step 3: Negative derivative along the projected curve in $\mathcal{D}$.} We project the curve constructed in Step 2 onto $\mathcal{D}$ to obtain a projected curve $\varphi$. Then, thanks to the property \cref{eq:mapping-invariant}, it is not difficult to show that $f(x)$ has a negative derivative along $\varphi$ since the value from $f(x)$ to $f(x')+\epsilon$ is strictly decreasing.

\item \textbf{Step 4: Constructing a decreasing direction.} Assuming the domain is open, we can let $v\in\mathbb{R}^d$ to be the tangent vector of $\varphi$ at $x$. The associated directional derivative $\lim_{t\downarrow 0} (f(x+tv)-f(x))/t$ will then be strictly negative. This would indicate that $x$ cannot be a Clarke stationary point, provided that $f$ is subdifferentially regular. 

\end{itemize}

\begin{figure}
    \centering
    \includegraphics[width = 0.7\textwidth]{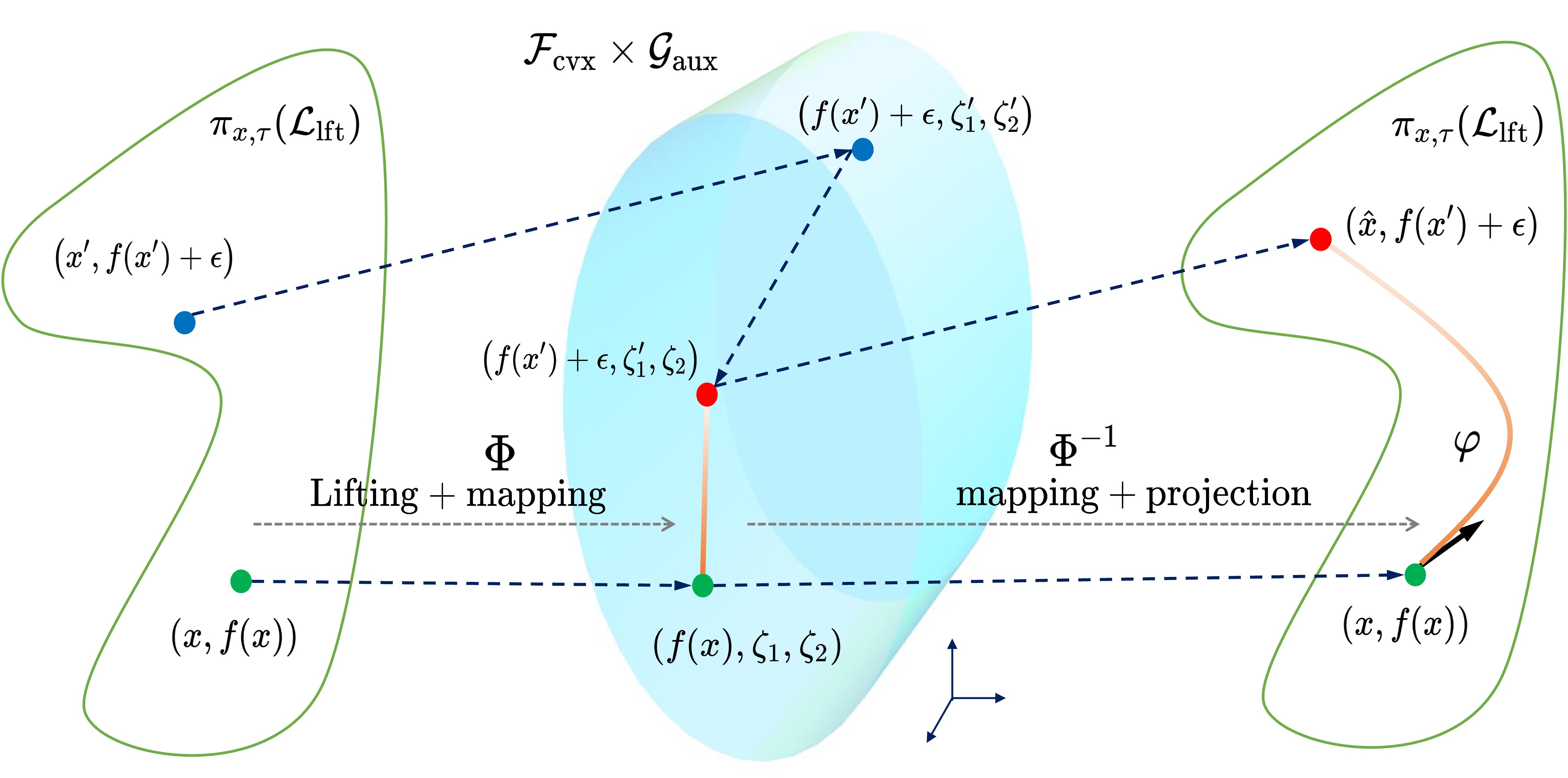}
    \caption{Illustration of the proof of \cref{theorem:ECL-guarantee}. The main proof is via contrapositive: Given $x\in\mathcal{D}_{\mathrm{nd}}$, if there exists $x'\in\mathcal{D}$ such that $f(x')<f(x)$, then we must have $0 \notin \partial f(x)$ due to the \ECL{} $(\mathcal{L}_{\mathrm{lft}}, \mathcal{F}_{\mathrm{cvx}},\mathcal{G}_{\mathrm{aux}},\Phi)$. \textit{Left figure}: we pick two points $(x,f(x))$ and $(x',f(x')+\epsilon)$ in $\pi_{x,\gamma} (\mathcal{L}_{\mathrm{lft}})$, where $\epsilon > 0$ is chosen such that $f(x')+\epsilon < f(x)$. \textit{Middle figure}: We lift and map the two points to $\mathcal{F}_{\mathrm{cvx}}\times\mathcal{G}_{\mathrm{aux}}$ via $\Phi$ and get a line segment from $(f(x),\zeta_1,\zeta_2)$ to $(f(x')+\epsilon,\zeta_1',\zeta_2)$ inside  $\mathcal{F}_{\mathrm{cvx}}\times\mathcal{G}_{\mathrm{aux}}$. \textit{Right figure}: The pre-image of this line segment under $\Phi$ is then a $C^2$ curve in $\mathcal{L}_{\mathrm{lft}}$, and we project this $C^2$ curve onto $\mathcal{D}$; then $f(x)$ has a negative derivative along $\varphi$ since $f(x)> f(x') + \epsilon$, confirming that $x$ cannot be a Clarke stationary point.}
    \label{fig:proof-illustration}
\end{figure}

We now formalize the idea above, and prove the following lemma that provides stronger and more general results. It is not hard to see that \Cref{theorem:ECL-guarantee} is just a direct corollary of the first part of \Cref{lemma:ECL-guarantee}.

\begin{lemma}\label{lemma:ECL-guarantee}
Let $f:\mathcal{D}\rightarrow \mathbb{R}$ be a locally Lipschitz continuous function equipped with the \ECL\ \ 
$(\mathcal{L}_{\mathrm{lft}}, \mathcal{F}_{\mathrm{cvx}},\mathcal{G}_{\mathrm{aux}},\Phi)$.
Suppose $x \in\mathcal{D}_{\mathrm{nd}}$ is not a global minimizer of $f$ on $\mathcal{D}$.
\begin{subequations}
\begin{enumerate}
\item If $\mathcal{D}$ is open, then there exists a non-zero direction $v\in\mathbb{R}^d\backslash\{0\}$ such that
\begin{equation} \label{eq:ECL-guarantee-1}
\limsup_{t\downarrow 0}\frac{f(x+tv)-f(x)}{t}<0,
\end{equation}
and if $f$ is also subdifferentially regular, we have $0\notin \partial f(x)$.

\item If $\mathcal{D}$ is closed, then there exists a $C^2$ curve $\varphi:[0,\delta)\rightarrow\mathcal{D}$ such that $\varphi(0)=x$, $\varphi'(0)\neq 0$ and
\begin{equation}  \label{eq:ECL-guarantee-2}
\limsup_{t\downarrow 0}\frac{f(\varphi(t))-f(x)}{t}<0.
\end{equation}
\end{enumerate}
\end{subequations}
\end{lemma}

One subtle technical difference between open and closed domains is that we allow a straight direction $v\in\mathbb{R}^d\backslash\{0\}$ in \cref{eq:ECL-guarantee-1}, while \cref{eq:ECL-guarantee-2} only allows the derivative along a curve $\varphi(t)$. This is because the point $x$ might be on the boundary of $\mathcal{D}$ when it is closed, and in this case, $x+tv$ might be outside $\mathcal{D}$ for any $t>0$ (thus $f(x+tv)$ is not defined). When $\mathcal{D}$ is open, any $x\in \mathcal{D}$ will be an interior point, and there exists $\delta > 0$ such that $x+tv \in \mathcal{D}, \forall t \in [0, \delta)$.

\begin{proof}
For technical convenience, we introduce the set
\[
\tilde{\mathcal{D}} = \setv*{y\in\operatorname{cl}\mathcal{D}}
{\liminf_{y'\rightarrow y,y'\in\mathcal{D}}f(y')<+\infty}
\]
and the auxiliary function
\[
\tilde{f}(y) = \liminf_{y'\rightarrow y, y'\in\mathcal{D}} f(y'),
\qquad y\in\tilde{\mathcal{D}}.
\]
It can be shown that $\operatorname{epi}_\geq(\tilde{f})=\operatorname{cl}\operatorname{epi}_{\geq} (f)$ (see, e.g., \cite[Section 1.D]{rockafellar2009variational}), and that $\tilde{f}(x)=f(x)$ whenever $x\in\mathcal{D}$ since $f$ is continuous on $\mathcal{D}$.

\vspace{3pt}
\noindent\textbf{Step 1.} Since $x\in\mathcal{D}_{\mathrm{nd}}$, we have $(x,f(x))\in\pi_{x,\gamma}(\mathcal{L}_{\mathrm{lft}})$. Thus, there exists $\xi\in\mathbb{R}^{d_\xi}$ such that $(x,f(x),\xi)\in\mathcal{L}_{\mathrm{lft}}$. The relation~\cref{eq:mapping-invariant} then establishes that 
\begin{equation} \label{eq:construction-point-1}
\Phi(x,f(x),\xi) = (f(x),\zeta_1,\zeta_2) \in \mathcal{F}_{\mathrm{cvx}} \times \mathcal{G}_{\mathrm{aux}}   
\end{equation}
for some $\zeta_1\in \mathbb{R}^{d_1}$ and $\zeta_2\in\mathcal{G}_{\mathrm{aux}}$. 

Next, since $x\in\mathcal{D}_{\mathrm{nd}}$ is not a global minimizer of $f$ on $\mathcal{D}$, we can find another point $x'\in\mathcal{D}$ such that $f(x')<f(x)$. We let $\epsilon>0$ be sufficiently small such that $f(x')+\epsilon<f(x)$. Then, we have $(x',f(x')+\epsilon) \in \operatorname{epi}_>(f)$, which, by~\cref{eq:epi-graph-lifting}, further leads to the existence of some $\xi'\in\mathbb{R}^{d_\xi}$ such that $(x',\gamma',\xi')\in \mathcal{L}_{\mathrm{lft}}$. By using~\cref{eq:mapping-invariant}, we find $\zeta_1'\in \mathbb{R}^{d_1}$ and $\zeta_2'\in\mathcal{G}_{\mathrm{aux}}$ such that
\begin{equation*} 
\Phi(x',f(x')+\epsilon,\xi') = (f(x')+\epsilon,\zeta_1',\zeta_2') \in \mathcal{F}_{\mathrm{cvx}} \times \mathcal{G}_{\mathrm{aux}} 
\end{equation*}

\vspace{3pt}
\noindent\textbf{Step 2.} Now we use the two points~constructed above to define a line segment $\psi:[0,1)\rightarrow \mathcal{F}_{\mathrm{cvx}}\times\mathcal{G}_{\mathrm{aux}}$ by
\[
\begin{aligned}
\psi(t) &=(1-t)\times (f(x),\zeta_1,\zeta_2) + t\times (f(x')+\epsilon, \zeta'_1, \zeta_2) \\
&= ((1-t)f(x)+t(f(x')+\epsilon), (1-t)\zeta_1+t\zeta'_1,\zeta_2), \qquad t\in[0,1),
\end{aligned}
\]
(note that in constructing $\psi$, we let the endpoint of $\psi$ to be $(f(x')+\epsilon, \zeta'_1, \zeta_2)$ rather than $(f(x')+\epsilon, \zeta'_1, \zeta_2')$). It is obvious that $\psi$ is a $C^2$ curve, and by the convexity of $\mathcal{F}_{\mathrm{cvx}}$, we have 
$$
\psi(t)\in \mathcal{F}_{\mathrm{cvx}}\times\mathcal{G}_{\mathrm{aux}},\qquad \forall t\in[0,1).
$$ 
Since $\Phi^{-1}$ is a $C^2$ diffeomorphism from $\mathcal{F}_{\mathrm{cvx}}\times\mathcal{G}_{\mathrm{aux}}$ to $\mathcal{L}_{\mathrm{lft}}$, we see that $\Phi^{-1}\circ\psi$ is a $C^2$ curve defined over $[0,1)$ with image in $\mathcal{L}_{\mathrm{lft}}$.

\vspace{3pt}
\noindent\textbf{Step 3.} We then let
\begin{align*}
\varphi =\ & \pi_{x}\circ\Phi^{-1}\circ\psi, \\
\varrho =\ & \pi_\gamma\circ\Phi^{-1}\circ\psi,
\end{align*}
where $\pi_x$ (resp. $\pi_\gamma$) denotes the canonical projection operator onto the first $d$ coordinates (resp. the $(d+1)$'th coordinate). It is obvious that we have
\[
\varphi(0)
=\pi_{x}(\Phi^{-1}(\psi(0)))
=\pi_{x}(\Phi^{-1}(f(x),\zeta_1,\zeta_2))
\overset{\cref{eq:construction-point-1}}{=}\pi_{x}(x,f(x),\xi) = x,
\]
and since $\Phi^{-1}\circ\psi$ is a curve in $\mathcal{L}_{\mathrm{lft}}$,
\begin{equation*}
(\varphi(t),\varrho(t))
=\pi_{x,\gamma}(\Phi^{-1}(\psi(t)))
\in \pi_{x,\gamma}(\mathcal{L}_{\mathrm{lft}})\subseteq
\operatorname{cl}\operatorname{epi}_{\geq}(f),
\qquad\forall t\in[0,1),   
\end{equation*}
where the last inclusion relationship is due to the definition~\cref{eq:epi-graph-lifting}.  
Moreover, we have
\begin{equation} \label{eq:varrho-mapping}
\begin{aligned}
\varrho(t)
=\ &
\pi_{\gamma}(\Phi^{-1}(\psi(t))) \\
=\ &
\pi_{\gamma}(\Phi^{-1}((1-t)f(x)+t(f(x')+\epsilon)
,(1-t)\zeta_1+t\zeta'_1,\zeta_2)) \\
=\ &
(1-t)f(x)+t(f(x')+\epsilon),
\end{aligned}
\end{equation}
thanks to the property \cref{eq:mapping-invariant}. Since $\operatorname{cl}\operatorname{epi}_{\geq}(f)=\operatorname{epi}_\geq(\tilde{f})$, we have $\tilde{f}(\varphi(t))\leq\varrho(t),\,\forall t\in[0,1)$. Consequently, noting that $\varrho(0)=f(x)=\tilde{f}(x)$, we get
\begin{equation} \label{eq:negative-derivative}
\begin{aligned}
\limsup_{t\downarrow 0}
\frac{\tilde{f}(\varphi(t))-\tilde{f}(x)}{t}
\leq\ &
\limsup_{t\downarrow 0}
\frac{\varrho(t)-\varrho(0)}{t} \\
\overset{\cref{eq:varrho-mapping}}{=}\ &
\limsup_{t\downarrow 0}
\frac{(1-t)f(x)+t(f(x')+\epsilon)-f(x)}{t} \\
=\ &
f(x')+\epsilon-f(x)<0.
\end{aligned}
\end{equation}

\vspace{3pt}
\noindent\textbf{Step 4.} We let $v=\varphi'(0)$, and we will argue below that $v\neq 0$ for both open and closed domains. First, the subdifferential regularity of $f$ implies its local Lipschitz continuity, and so there exist $L>0$ and a bounded open set $U_x\subseteq\mathbb{R}^d$ containing $x$ such that $|f(x_1)-f(x_2)|\leq L\|x_1-x_2\|$ for any $x_1,x_2\in U_x\cap\mathcal{D}$.

Next, we consider the following two cases:
\begin{enumerate}
\item $\mathcal{D}$ is open: In this case, $U_x\cap\mathcal{D}$ is a bounded open neighborhood of $x$. Since $\varphi$ is a $C^2$ curve in $\mathbb{R}^d$, we can find some $\delta\in(0,1]$ such that $\varphi(t)\in U_x\cap \mathcal{D}$ for all $t\in[0,\delta)$, and also find $M>0$ such that
\[
\left\|\varphi(t)-x-tv\right\| \leq \frac{M}{2}t^2,\qquad\forall t\in[0,\delta).
\]
As a result, if $v=0$, we would get
\[
\frac{|f(\varphi(t)) - f(x)|}{t}
\leq \frac{L \|\varphi(t)-x\|}{t}
\leq \frac{LM}{2}t
\rightarrow 0,\quad\text{as}\ t\downarrow 0,
\]
which contradicts \cref{eq:negative-derivative} as $\tilde{f}(\varphi(t))=f(\varphi(t))$ for all $t\in[0,\delta)$. Thus $v\neq 0$, and
\begin{align*}
\frac{f(x+tv)-f(x)}{t}
\leq\ &
\frac{f(\varphi(t))-f(x)}{t}
+
\frac{\left|f(x+tv)-f(\varphi(t))\right|}{t} \\
\leq\ &
\frac{f(\varphi(t))-f(x)}{t}
+\frac{L\|x+tv-\varphi(t)\|}{t} \\
\leq\ &
\frac{f(\varphi(t))-f(x)}{t}
+\frac{LM}{2}t.
\end{align*}
By taking the limit superior as $t\downarrow 0$, and combining it with \cref{eq:negative-derivative} and $\tilde{f}(\varphi(t))=f(\varphi(t))$ for all $t\in[0,\delta)$, we get
\[
\limsup_{t\downarrow 0}
\frac{f(x+tv)-f(x)}{t} \leq
f(x')+\epsilon-f(x)<0.
\]

If $f$ is also subdifferetially regular, then we have
\[
f^\circ(x;v)
=\lim_{t\downarrow 0}
\frac{f(x+tv)-f(x)}{t} = \limsup_{t\downarrow 0}
\frac{f(x+tv)-f(x)}{t} <0,
\]
confirming that $0\notin\partial f(x)$ (see \cite[Lemma B.2]{zheng2023benign}).

\item $\mathcal{D}$ is closed: In this case, since $f$ is continuous on the closed set $\mathcal{D}$, we see that the set $\tilde{\mathcal{D}}$ is just the original domain $\mathcal{D}$, and $\tilde{f}=f$.
Consequently, the curve $\varphi$ lies in $\mathcal{D}$, and \cref{eq:negative-derivative} holds with $\tilde{f}$ replaced by $f$.

Since $\varphi$ is a $C^2$ curve, we can find $\delta>0$ such that $\varphi(t)\in U_x\cap\mathcal{D}$ for all $t\in[0,\delta)$. We can then mimic the proof for the case with $\mathcal{D}$ open and show that $\varphi'(0)\neq 0$.
\end{enumerate}
The proof is now complete.
\end{proof}

\subsection{Conditions for the Non-strict Epigraph to be Closed} \label{appendix:closure-strict-epigraph}

We finish this section by providing some conditions for the non-strict epigraph to be a closed set.
First, we note the following necessary and sufficient condition:
\begin{proposition}
Given $f:\mathcal{D}\rightarrow\mathbb{R}$ where $\mathcal{D}\subseteq\mathbb{R}^n$, we have
$
\operatorname{epi}_\geq (f) =
\operatorname{cl}\operatorname{epi}_\geq(f)$
if and only if the extended real-valued function $\bar{f}:\mathbb{R}^n\rightarrow\mathbb{R}\cup\{+\infty\}$ defined as
\[
\bar{f}(x) = \begin{cases}
f(x), & x\in\mathcal{D}, \\
+\infty, & x\notin\mathcal{D}
\end{cases}
\]
is lower semicontinuous.
\end{proposition}
\begin{proof}

This is a directly consequence of the equivalence between the lower semicontinuity of $\bar{f}$ and the closedness of $\operatorname{epi}_\geq (f)$ as shown by \cite[Theorem 1.6]{rockafellar2009variational}.
\end{proof}

By applying sufficient conditions for lower semicontinuity, we directly get the following corollary.
\begin{corollary} \label{corollary:non-strict-epigraph}
Let $f:\mathcal{D}\rightarrow\mathbb{R}$ be a continuous function defined on the domain $\mathcal{D}\subseteq\mathbb{R}^n$. We have
$
\operatorname{epi}_\geq (f) = \operatorname{cl}\operatorname{epi}_\geq(f) 
$ 
if any one of the following conditions holds:
\begin{enumerate}
\item $\mathcal{D}$ is closed.
\item $\mathcal{D}$ is open, and for any sequence $\{x_t\}_{t=1}^\infty$ that converges to some point on the boundary of $\mathcal{D}$, we have $f(x_t)\rightarrow+\infty$ as $t\rightarrow\infty$.
\end{enumerate}
\end{corollary}

As we will see next in \Cref{section:applications}, the objective function in some policy optimization problems in control may not have a closed non-strict epigraph. 

%% file: sec_Applications.tex
\section{Applications in Optimal and Robust Control} \label{section:applications}

In this section, we reveal hidden convexity in several benchmark optimal and robust control problems, including LQR \cite{fazel2018global}, state feedback $\mathcal{H}_\infty$ control \cite{guo2022global}, LQG \cite{zheng2021analysis}, output feedback $\mathcal{H}_\infty$ control \cite{tang2023global}, and a class of distributed control problems  \cite{furieri2020learning}. In particular, we will construct appropriate \ECL{}s for nonconvex (and potentially nonsmooth) policy optimization problems in control. Then, convex re-parameterization in \Cref{theorem:convex-equivalency} and global optimality in \cref{theorem:ECL-guarantee} directly apply to these nonconvex policy optimization problems. The main global optimality characterizations in Part~I of this paper \cite[Theorems 4.2 \& 5.2]{zheng2023benign} are direct corollaries of the results in this section. We remark that \ECL{} constructions for LQG and output feedback $\mathcal{H}_\infty$ control require resolving non-trivial technical challenges. 

We first review standard LMI characterizations for $\mathcal{H}_2/\mathcal{H}_\infty$ norms in \Cref{subsection:LMIs}. Static policy optimization with state feedback and dynamic policy optimization with output feedback are discussed in \cref{subsection:static-policies} and \cref{subsection:dynamic-policies}, respectively. Finally, we will discuss a class of distributed policy optimization problems \cite{furieri2020learning} in \cref{subsection:distributed-policies}. 

\subsection{Linear Matrix Inequalities in Control} \label{subsection:LMIs}

We first review a state-space characterization for the $\mathcal{H}_2$ norm of a stable transfer function using Lyapunov equations and LMIs. 

\begin{lemma} \label{lemma:H2norm} 
    Consider a transfer function $\mathbf{G}(s) = C(sI - A)^{-1}B$, where $A \in \mathbb{R}^{n\times n}$ and $B \in \mathbb{R}^{n \times m}, C \in \mathbb{R}^{p \times n}$. The following statements hold.
    \begin{enumerate}
        \item (Lyapunov equations). Suppose $A$ is stable. We have $\|\mathbf{G}\|_{\mathcal{H}_2}^2 = \mathrm{tr}(B^\tr L_{\mathrm{o}} B) = \mathrm{tr}(C L_{\mathrm{c}} C^\tr) $, where $L_{\mathrm{o}}$ and $L_{\mathrm{c}}$ are observability and controllability Gramians which can be obtained from the Lyapunov equations
        \begin{subequations} \label{eq:Lyapunov-equations-H2norm}
            \begin{align}
                AL_{\mathrm{c}} + L_{\mathrm{c}}A^\tr + BB^\tr &= 0, \label{eq:Lyapunov-equations-H2norm-a}\\
                A^\tr L_{\mathrm{o}} + L_{\mathrm{o}}A + C^\tr C & = 0.  \label{eq:Lyapunov-equations-H2norm-b}
            \end{align}
        \end{subequations}
        \item (Strict LMI). $\|\mathbf{G}\|_{\mathcal{H}_2} < \gamma$ if and only if there exist $P\in \mathbb{S}^n_{++}$ and $\Gamma \in \mathbb{S}^p_{++}$ such that the following strict LMI is feasible:
            \begin{align} \label{eq:strict-LMI-H2norm}
                \begin{bmatrix} A^\tr P+PA & PB \\ B^\tr P & -\gamma I \end{bmatrix}&\prec 0, \;\; \begin{bmatrix} P & C^\tr \\ C & \Gamma \end{bmatrix}\succ 0,\;\; \mathrm{tr}(\Gamma)<\gamma. 
            \end{align}
        \item (Non-strict LMI). Suppose $A$ is stable. We have $\|\mathbf{G}\|_{\mathcal{H}_2} \leq \gamma$ if there exist $P \in \mathbb{S}^n_{++}$ and $\Gamma \in \mathbb{S}^p_{+}$ such that the following non-strict LMI is feasible:
            \begin{align} \label{eq:nonstrict-lmi-h2}
                \begin{bmatrix} A^\tr P+PA & PB \\ B^\tr P & -\gamma I \end{bmatrix}&\preceq 0, \;\; \begin{bmatrix} P & C^\tr \\ C & \Gamma \end{bmatrix}\succeq 0,\;\; \mathrm{tr}(\Gamma)\leq\gamma.
            \end{align}
            The converse holds if $(A, B)$ is controllable.
    \end{enumerate}
\end{lemma}

We next present the celebrated \textit{bounded real lemma} that gives a state-space characterization for the $\mathcal{H}_\infty$ norm of a stable transfer function using LMIs.

\begin{lemma}[Bounded real lemma]
\label{lemma:bounded_real}
 Consider $\mathbf{G}(s) = C(sI - A)^{-1}B+D$, where $A \in \mathbb{R}^{n\times n}$, $B \in \mathbb{R}^{n \times m}, C \in \mathbb{R}^{p \times n}, D \in \mathbb{R}^{p \times m}$. 
Let $\gamma>0$ be arbitrary. The following statements hold. 

\begin{enumerate}
\item (Strict version) $\|\mathbf{G}\|_{\mathcal{H}_\infty}< \gamma$ if and only if there exists $P\in\mathbb{S}^n_{++}$ such that
\begin{equation} \label{eq:strict-hinf}
\begin{bmatrix}
A^\tr P + P A & PB & C^\tr \\
B^\tr P & -\gamma I & D^\tr \\
C & D & -\gamma I
\end{bmatrix}\prec 0.
\end{equation} 

\item (Non-strict version) Suppose $A$ is stable. We have $\|\mathbf{G}\|_{\mathcal{H}_\infty}\leq \gamma$ if there exists $P\in\mathbb{S}^n$ such that
\begin{equation} \label{eq:non-strict-hinf}
    \begin{bmatrix}
A^\tr P + P A & PB & C^\tr \\
B^\tr P & -\gamma I & D^\tr \\
C & D & -\gamma I
\end{bmatrix}\preceq 0.
\end{equation}
{The converse holds if $(A,B)$ is controllable.}
\end{enumerate}
\end{lemma}

Both \cref{lemma:H2norm} and \cref{lemma:bounded_real} are classical results in control, but the subtleties between strict and non-strict LMI characterizations (\cref{eq:strict-LMI-H2norm} vs. \cref{eq:nonstrict-lmi-h2}; \cref{eq:strict-hinf} vs. \cref{eq:non-strict-hinf}) have been less emphasized before. We refer the interested reader to Part I of this paper \cite[Section 3.1 \& Appendix A.3]{zheng2023benign} for more discussions. We here only remark that the non-strict LMIs in both \cref{lemma:H2norm} and \cref{lemma:bounded_real} \textit{assume that the system is stable; but when applying them to controller synthesis, stability is a design constraint}; converse results for the non-strict LMIs also require controllability of the (closed-loop) system. These subtleties motivate our inclusion in \cref{eq:epi-graph-lifting} and the notion of non-degenerate points (\Cref{definition:degenerate-points}).

\subsection{State Feedback Policy Optimization} \label{subsection:static-policies}

We consider a continuous-time\footnote{Within the scope of this paper, we only consider the continuous-time case. Unless explicitly stated otherwise, analogous results for discrete-time systems are also available.
} 
linear time-invariant (LTI) dynamical system
\begin{equation}\label{eq:Dynamic}
\begin{aligned}
\dot{x}(t) &= Ax(t)+Bu(t)+ B_ww(t), 
\end{aligned}
\end{equation}
where $x(t) \in \mathbb{R}^n$ is the vector of state variables, $u(t)\in \mathbb{R}^m$ the vector of control inputs, and $w(t) \in \mathbb{R}^n$ is the disturbance on the system process. We introduce the matrix $B_w\in\mathbb{R}^{n\times n}$ for a unified treatment of LQR and $\mathcal{H}_\infty$ control problems in this section. 
We consider the following performance signal 
\begin{equation} \label{eq:performance-signal}
    z(t) = \begin{bmatrix}
        Q^{1/2} \\ 0
    \end{bmatrix}x(t) + \begin{bmatrix}
        0 \\ R^{1/2}
    \end{bmatrix}u(t),
\end{equation}
where $Q \succeq 0$ and $R\succ 0$ are performance weight matrices. A standard assumption is: 
\begin{assumption} \label{assumption:controllability}
    $(A, B)$ is controllable and $(Q^{1/2}, A)$ is observable. 
\end{assumption}
Further assumptions will be imposed as needed in specific setups. For both LQR and $\mathcal{H}_\infty$ control, their cost values depend on $B_w$ only via $B_wB_w^\tr$. We thus define the weight~matrix
$
W \coloneqq B_w B_w^\tr,,
$
and assume $B_w=W^{1/2}$ without loss of generality. We also assume that $W \succ 0$, or equivalently that $B_w$ is of full row rank.

For policy optimization regarding state feedback control problems, we restrict ourselves to the class of static state feedback policies of the form $u(t) = Kx(t)$, with $K\in \mathbb{R}^{m \times n}$, to regulate the behavior of \cref{eq:performance-signal} under the influence of $w(t)$. This restriction has no loss of generality \cite[Theorem 14.2]{zhou1996robust}. The set of stabilizing state feedback policies is then parameterized by $K$ and can be represented by
$$
\mathcal{K}\coloneqq
\left\{K \in \mathbb{R}^{m \times n} \mid 
\max\nolimits_i \operatorname{Re}\lambda_i(A+BK)
<0\right\},
$$ 
and the closed-loop transfer function from $w$ to $z$ becomes
\begin{equation} \label{eq:transfer-function-Tzw}
\mathbf{T}_{zw}(K,s) = \begin{bmatrix}
    Q^{1/2} \\ R^{1/2}K
\end{bmatrix} (sI - A - BK)^{-1}B_w. 
\end{equation}

\subsubsection{Linear Quadratic Regulator (LQR)} \label{subsection:LQR}

For the linear quadratic regulator (LQR) problem, we consider $w(t)$ as white Gaussian noise with an identity intensity matrix, i.e., $\mathbb{E}\left[w(t)w(\tau )\right] = \delta(t - \tau )I_n$. The policy optimization formulation for LQR then reads as
\begin{equation} \label{eq:LQR-stochastic-noise}
\begin{aligned}
\min_{K \in \mathbb{R}^{m \times n}}\ \ &
\lim_{T \to \infty}\mathbb{E}\!\left[\frac{1}{T}\int_0^{T}
x^\tr(t) Q x(t) + u^\tr(t) R u(t)\,dt
\right]
\\
\text{subject to}\ \ & \dot{x}(t) = Ax(t)+Bu(t)+ B_ww(t), \\
& u(t) = Kx(t).
\end{aligned}
\end{equation}
Thanks to $W \succ 0$, the objective value is finite if and only if $A+BK$ is stable, i.e., $K \in \mathcal{K}$. Moreover, it can be shown that the objective value equals to $\|\mathbf{T}_{zw}(K,s)\|_{\mathcal{H}_2}^2$ whenever $K\in\mathcal{K}$. Thus we can equivalently reformulate \eqref{eq:LQR-stochastic-noise} as the following $\mathcal{H}_2$ optimization problem:
\begin{equation} \label{eq:LQR-H2}
\begin{aligned}
\min_{K}\ \ & J_{\mathtt{LQR}}(K) \coloneqq \|\mathbf{T}_{zw}(K,s)\|_{\mathcal{H}_2}^2  \\
\text{subject to}\ \ &
K\in\mathcal{K}.
\end{aligned}
\end{equation}

It is known that the policy optimization for LQR \cref{eq:LQR-H2} is smooth and nonconvex,  but has nice landscape properties that $J_{\mathtt{LQR}}(K)$ has a unique stationary point which is globally optimal and is gradient dominated on any sublevel set when $W \succ 0$ \cite{fazel2018global,mohammadi2019global}. These nice properties are closely related to the hidden convexity of $J_{\mathtt{LQR}}(K)$, and the aim of this subsection is to construct an \ECL{} for \cref{eq:LQR-H2}. Then \Cref{theorem:convex-equivalency,theorem:ECL-guarantee} will directly apply.

Our \ECL{} construction for LQR consists of the following three steps. 

\vspace{3pt}
\noindent \textbf{Step 1: Lifting.} We first define a lifted set with an extra Lyapunov variable $X$ as
\[
\mathcal{L}_{\mathtt{LQR}}
= \setv*{(K,\gamma,X)}{
X\succ 0,
(A+BK)X + X(A+BK)^\tr + W = 0,
\gamma \geq \operatorname{tr}\!\left(
(Q+K^\tr RK)X
\right)}.
\]
By \cref{eq:Lyapunov-equations-H2norm-a} in \cref{lemma:H2norm}, it is not difficult to verify that\footnote{This equivalence requires the assumption that $W := B_w B_w^\tr \succ 0$. In this case, we know that $A+BK$ is stable if and only if the Lyapunov equation
 $
 (A+BK) X + X(A+BK)^\tr + W=0 
 $ 
 has a unique positive definite solution $X \succ 0$.} $\gamma\geq J_{\mathtt{LQR}}(K)$ with $K \in \mathcal{K}$ if and only if there exists $X$ such that $(K,\gamma,X)\in\mathcal{L}_{\mathtt{LQR}}$. This further implies that $\pi_{K,\gamma}(\mathcal{L}_{\mathtt{LQR}})
 =\operatorname{epi}_\geq(J_{\mathtt{LQR}})$.
In fact, we further have $\operatorname{epi}_\geq(J_{\mathtt{LQR}}) = \operatorname{cl}\operatorname{epi}_\geq(J_{\mathtt{LQR}})$ thanks to the coercivity of $J_{\mathtt{LQR}}$ (see \Cref{corollary:non-strict-epigraph}).

\vspace{3pt}
\noindent \textbf{Step 2: Convex set.} We define the convex set 
\begin{align*}
\mathcal{F}_{\mathtt{LQR}}
=
\setv*{
(\gamma, Y, X)}{X \succ 0, 
AX+BY + XA^\tr + Y^\tr B^\tr + W = 0,
\gamma\geq \operatorname{tr}
\!\left(QX + X^{-1}Y^\tr RY\right)
},
\end{align*}
and let the auxiliary set be $\mathcal{G}_{\mathtt{LQR}} = \{0\}$. It is obvious that the first two constraints in $\mathcal{F}_{\mathtt{LQR}}$ are convex, and the convexity of the last inequality is due to the fact that the matrix fractional function $f(X,Y) = \operatorname{Tr}
\!\left(QX + X^{-1}Y^\tr RY\right)$ is jointly convex over $X \succ 0, Y \in \mathbb{R}^{m \times n}$ (this is a known fact, but we provide some details in \Cref{subsection:convexity-LQR} for completeness). 

\vspace{3pt}
\noindent \textbf{Step 3: Diffeomorphism.} To construct a full \ECL{}, we employ a classical change of variables $Y = KX$ \cite{boyd1994linear,bernussou1989linear} and introduce the mapping
\begin{equation*} \label{eq:Diffeomorphism-LQR}
\Phi_{\mathtt{LQR}}(K,\gamma, X) = (\gamma, \underbrace{KX, X}_{\zeta_1}),\qquad\forall (K,\gamma,X)\in\mathcal{L}_{\mathtt{LQR}}.
\end{equation*}
This mapping naturally satisfies \cref{eq:mapping-invariant}; we do not need an auxiliary variable $\zeta_2$ as no similarity transformation exists for static state feedback policies. It is straightforward to check that $\mathcal{F}_{\mathtt{LQR}} =
\Phi_{\mathtt{LQR}}(\mathcal{L}_{\mathtt{LQR}})$, and that $\Phi_{\mathtt{LQR}}$ admits an inverse on $\mathcal{F}_{\mathtt{LQR}}$ given by $\Phi_{\mathtt{LQR}}^{-1}(\gamma,Y,X)
=(YX^{-1},\gamma,X) \in \mathcal{L}_{\mathtt{LQR}}$ for any $(\gamma,Y,X)\in \mathcal{F}_{\mathtt{LQR}}$. Furthermore, $\Phi_{\mathtt{LQR}}$ is a $C^2$ (and in fact $C^\infty$) diffeomorphism between the lifted set $\mathcal{L}_{\mathtt{LQR}}$ and the convex set $\mathcal{F}_{\mathtt{LQR}}$. 

Consequently, $(\mathcal{L}_{\mathtt{LQR}}, \mathcal{F}_{\mathtt{LQR}},\{0\},\Phi_{\mathtt{LQR}})$ is an \ECL{} of $J_{\mathtt{LQR}}(K)$ in \cref{eq:LQR-H2}. One key step in the construction above is the utilization of the classical change of variables $Y = K X$ or equivalently $K = YX^{-1}$.
For this \ECL{} constructed above, we have the following nice result, which is due to  $\pi_{K,\gamma}(\mathcal{L}_{\mathtt{LQR}}) =\operatorname{epi}_\geq(J_{\mathtt{LQR}})$ by our construction  (see \Cref{fact:non-degenerate-point}).

\begin{proposition} \label{proposition:LQR-non-degenearte}
    Under \cref{assumption:controllability}, consider the LQR \cref{eq:LQR-H2} with $W \succ 0, Q \succeq 0, R \succ 0$. Then all stabilizing policies $K \in \mathcal{K}$ are non-degenerate with respect to the \ECL{} $(\mathcal{L}_{\mathtt{LQR}}, \mathcal{F}_{\mathtt{LQR}},\{0\},\Phi_{\mathtt{LQR}})$ constructed above.
\end{proposition}

Thus, our \ECL{} framework immediately implies the following two established results: 
\begin{enumerate}
    \item The policy optimization of LQR is a convex problem in disguise, confirmed by \Cref{theorem:convex-equivalency}, i.e., we have\footnote{
    In this case, the infima can be achieved for both formulations due to the coerciveness of $J_{\mathtt{LQR}}(K)$ and the compactness of $\setv*{(Y,X)}{(\gamma,Y,X)\in\mathcal{F}_{\mathtt{LQR}}}$ for any given $\gamma>0$, but we will not delve into the details here.
    }
    $$
    \min_{K \in \mathcal{K}}\ \  J_{\mathtt{LQR}}(K) = \min_{(\gamma, Y, X)\in \mathcal{F}_{\mathtt{LQR}}} \, \gamma,
    $$
    where the right-hand-side problem is convex and equivalent to an LMI, and their optimal solutions $K^\ast$ and $(\gamma^\ast,Y^\ast,X^\ast)$ are related by $K^\ast= Y^\ast {X^\ast}^{-1}$ and $\gamma^\ast=J_{\mathtt{LQR}}(K^\ast)$.  
    \item Any stationary point $K^\star$ of $J_{\mathtt{LQR}}$ (i.e., $\nabla J_{\mathtt{LQR}}(K^\star) = 0$) is globally optimal, confirmed by \Cref{theorem:ECL-guarantee}.
\end{enumerate}

In the global optimality, we have used the fact that $J_{\mathtt{LQR}}(K)$ is continuously differentiable (and indeed infinitely differentiable) over $\mathcal{K}$ \cite{levine1970determination}. Thus, Clarke stationarity is reduced to the normal stationarity with zero gradient. For completeness, we review the calculation of stationary points in \cref{appendix:stationary-point}. Furthermore, it is known that $J_{\mathtt{LQR}}(K)$ has a unique stationary point, is \textit{coercive}, is \textit{$L$-smooth} and \textit{gradient dominated} over any sublevel set~\cite{fazel2018global}. These properties are fundamental to establishing global convergence of direct policy search and their model-free extensions for solving LQR~\cite{malik2019derivative,mohammadi2021convergence}. We point out that some of these results cannot be obtained by our current \ECL{} framework without imposing further geometrical properties on $\mathcal{F}_{\mathtt{LQR}}$, and we left such refinements of \ECL{} to our future work.

\begin{remark}[Elimination of the lifting variable $X$] \label{remark:eliminating-lifting-variables}
 Thanks to the affine constraint $AX+BY+XA^\tr + Y^\tr B^\tr +W=0$, in the set $\mathcal{F}_{\mathrm{LQR}}$, we can express the variable $X$ as a function of $Y$ when the linear mapping $X  \to  \mathcal{A}(X)\coloneqq AX + XA^\tr$ is invertible; see \cite[Lemma 2.7]{zhou1996robust}. By this observation, with $\mathcal{A}$ being invertible, we may define
 \[
 \begin{aligned}
 & \tilde{\mathcal{L}}_{\mathtt{LQR}}
 =\setv*{(K,\gamma)}{
 \begin{aligned}
 \exists X\succ 0\ \text{s.t.}\ &
(A+BK)X + X(A+BK)^\tr + W = 0
\\
& \text{and}\ \gamma \geq \operatorname{tr}\!\left(
(Q+K^\tr RK)X\right) 
\end{aligned}}
=\operatorname{epi}_{\geq}\!\left(J_{\mathtt{LQR}}\right), \\
& \tilde{\mathcal{F}}_{\mathtt{LQR}}
 =\setv*{(\gamma,Y)}{
 \begin{aligned}
& -\mathcal{A}^{-1}(BY\!+\!Y^\tr B^\tr \!+\!W)\succ 0, \\
& \gamma\geq \operatorname{tr}
 \!\left(-Q\mathcal{A}^{-1}(BY\!+\!Y^\tr B^\tr \!+\!W)
 -\left(\mathcal{A}^{-1}(BY\!+\!Y^\tr B^\tr \!+\!W)\right)^{-1}Y^\tr R Y\right)
 \end{aligned}
 }, \\
 & \tilde{\Phi}_{\mathtt{LQR}}(K,\gamma)
  = (\gamma, KX),
 \end{aligned}
 \]
 and it can be shown that $(\tilde{\mathcal{L}}_{\mathtt{LQR}},\tilde{\mathcal{F}}_{\mathtt{LQR}},\{0\},\tilde{\Phi}_{\mathtt{LQR}})$ also gives an \ECL{} for LQR. In this \ECL{} formulation, we essentially eliminate the lifting variable $X$.
 In fact, we have employed this elimination technique for the simple LQR instance in \cref{example:LQR_ex}.
 Similar techniques have been used in the recent work \cite{mohammadi2021convergence} to derive the gradient dominance property.
 \hfill $\square$
\end{remark}

\begin{remark}[\ECL{} for deterministic LQR]
Our formulation of the LQR problem considers stochastic noise and infinite-horizon averaged cost. Another commonly studied form of LQR is the deterministic infinite-horizon LQR with a random initial state, i.e.,
\begin{equation} \label{eq:LQR-deterministic}
\begin{aligned}
\min_{K \in \mathbb{R}^{m \times n}}\ \ & \mathbb{E}_{x(0)\sim D}\!\left[\int_0^{\infty}
x^\tr(t) Q x(t) + u^\tr(t) R u(t)\,dt
\right]
\\
\text{subject to}\ \ & \dot{x}(t) = Ax(t) + Bu(t), \\
& u(t) = Kx(t),
\end{aligned}
\end{equation}
where ${D}$ is a probability distribution over $\mathbb{R}^n$ with zero mean.
It is not difficult to check that this form of LQR admits almost the same \ECL: The objective value of~\cref{eq:LQR-deterministic} has the same dependence on $K$ except that the covariance matrix $\Omega=\mathbb{E}_{x(0)\sim D}\!\left[x(0)x(0)^\tr\right]$ now plays the role of $W$, which~we~assume to be positive definite, similar to the setup in \cite{fazel2018global,mohammadi2021convergence}. Then, the same \ECL{} construction applies~here.

We note that in all existing policy optimization formulations for LQR \cite{fazel2018global,mohammadi2021convergence}, it is assumed $W \succ 0$ or $\Omega \succ 0$. However, this assumption is not needed for Riccati-based solutions.
\hfill $\square$
\end{remark}

\subsubsection{$\mathcal{H}_\infty$ State Feedback Control} \label{subsection:hinf-state-feedback}
In $\mathcal{H}_\infty$ state feedback control, we consider $w(t)$ as adversarial disturbance with bounded energy. Let $\mathcal{L}_2^{k}[0,\infty)$ be the set of square-integrable (bounded energy) signals of dimension $k$, i.e.,
$$
\mathcal{L}_2^{k}[0,\infty): =\left\{w:[0,+\infty)\rightarrow\mathbb{R}^{k}
\left|\,\|w\|_{\ell_2}^2 \coloneqq \int_0^\infty w(t)^\tr w(t)\,dt < \infty \right.\right\}.
$$
A standard form of policy optimization for $\mathcal{H}_\infty$ robust control with state feedback reads as
\begin{equation} \label{eq:Hinf-state-feedback}
\begin{aligned}
\inf_{K \in \mathbb{R}^{m \times n}}\ \ & \sup_{\|w\|_{\ell_2} \leq 1}\int_0^{\infty}
x^\tr(t) Q x(t) + u^\tr(t) R u(t)\,dt
\\
\text{subject to}\ \ & \dot{x}(t) = Ax(t)+Bu(t)+ B_ww(t), \ x(0) =0, \\
& u(t) = Kx(t).
\end{aligned}
\end{equation}
Again, when $W \succ 0$, the objective value is finite if and only if $A+BK$ is stable. 
Standard techniques in control allow us to equivalently reformulate \cref{eq:Hinf-state-feedback} as the following $\mathcal{H}_\infty$ optimization problem:
\begin{equation} \label{eq:Hinf}
\begin{aligned}
\inf_{K\in \mathbb{R}^{m \times n}}\ \ & J_{\infty}(K):= \|\mathbf{T}_{zw}(K,s)\|_{\mathcal{H}_\infty}  \\
\text{subject to}\ \ &
K\in\mathcal{K}. 
\end{aligned}
\end{equation}
Note that \cref{eq:Hinf-state-feedback,eq:Hinf} have the same feasible region but the objective value of \cref{eq:Hinf-state-feedback} is the square of that of \cref{eq:Hinf}. The infimum of \cref{eq:Hinf} might not be attainable, and there exist simple examples where the infimum can be approached only by letting the policy $K$ go to infinity (see \cref{subsection:non-coercivity}).

It is known that the policy optimization for the $\mathcal{H}_\infty$ control \cref{eq:Hinf} is nonconvex and nonsmooth, but it admits a convex reformulation \cite{scherer2000linear}. 
Very recently, it has been revealed in \cite{guo2022global} that the discrete-time version of \cref{eq:Hinf} enjoys a nice property that any Clarke stationary point is globally optimal. However, unlike the LQR, it is currently unclear when the stationary point of \cref{eq:Hinf} is unique.
Moreover, we have the following interesting fact; see \cref{example:Hinf} and another explicit example in \cref{subsection:non-coercivity}.  

\begin{fact} \label{fact:Hinf-non-coercive}
    The objective function $J_\infty:\mathcal{K}\rightarrow\mathbb{R}$ defined in~\cref{eq:Hinf} is in general not coercive, even when $(A,B)$ is controllable and $W,Q,R$ are all positive definite. 
\end{fact}

This fact is different from the discrete-time version of $\mathcal{H}_\infty$ control, whose corresponding objective function is shown to be coercive \cite{guo2022global}.  
Our aim here is to construct an \ECL{} for \cref{eq:Hinf}, which then allows us to draw conclusions from \Cref{theorem:convex-equivalency,theorem:ECL-guarantee} directly. The construction process is very similar to the LQR case. The only key difference is that the $\mathcal{H}_\infty$ case relies on the non-strict version of the bounded real lemma~in~\cref{eq:non-strict-hinf}. 

We have the following three steps. 

\vspace{3pt}
\noindent \textbf{Step 1: Lifting.} Based on \cref{lemma:bounded_real}, we define a lifted set with an extra Lyapunov variable $P$ as
\begin{equation} \label{eq:lifted-set-Hinf-state-feedback}
\mathcal{L}_{\infty}
= \setv*{(K,\gamma,P)}{
P\succ 0,
\begin{bmatrix}
(A+BK)^\tr P \!+\! P (A+BK) & PB_w & Q^{1/2} & K^\tr R^{1/2} \\
B_w^\tr P & -\gamma I & 0 & 0 \\
Q^{1/2} & 0 & -\gamma I & 0 \\
R^{1/2}K & 0 & 0 & -\gamma I
\end{bmatrix}\preceq 0}.
\end{equation}
We will show below that this lifted set satisfies  $
\pi_{K,\gamma}(\mathcal{L}_{\infty})
=\operatorname{epi}_\geq(J_{\infty}).$

\vspace{3pt}
\noindent \textbf{Step 2: Convex set.} We define a convex set 
\begingroup
    \setlength\arraycolsep{2pt}
\def\arraystretch{0.8} 
\begin{equation} \label{eq:convex-set-hinf-state-feedback}
\mathcal{F}_{\infty}
=
\setv*{
(\gamma, Y, X)}{  \begin{aligned} &X \succ 0, \\
     &Y \in \mathbb{R}^{m \times n},
\end{aligned} \; 
\begin{bmatrix}
    AX\!+\! BY \!+\! XA^\tr \!+\!Y^\tr B^\tr & B_w & XQ^{1/2} & Y^\tr R^{1/2} \\
    B_w^\tr & -\gamma I & 0 & 0 \\
    Q^{1/2}X & 0 & -\gamma I & 0 \\
    R^{1/2}Y & 0 & 0 & -\gamma I
\end{bmatrix} \preceq 0 },
\end{equation}
\endgroup 
and let the auxiliary set be $\mathcal{G}_{\infty} = \{0\}$. 

\vspace{3pt}
\noindent \textbf{Step 3: Diffeomorphism.} To construct a full \ECL{}, we employ the classical change of variables $Y = KP^{-1}, X = P^{-1}$ and introduce the mapping
\begin{equation} \label{eq:mapping-hinf-state-feedback}
\Phi_\infty(K,\gamma, P) = (\gamma, \underbrace{KP^{-1}, P^{-1}}_{\zeta_1}),\qquad\forall (K,\gamma,P)\in\mathcal{L}_{\infty}.
\end{equation}
Similar to the LQR case, no auxiliary variable $\zeta_2$ is needed. 

For the construction above, we have the following results, whose proof is given in \cref{subsection:proofs-state-feedback-Hinf}. 
\begin{proposition} \label{proposition:hinf-lifting}
    Under \cref{assumption:controllability}, consider the state feedback $\mathcal{H}_\infty$ policy optimization problem \cref{eq:Hinf}, where $B_w$ is full row rank and {$Q \succ 0, R \succ 0$}. The following statements hold. 
    \begin{enumerate}
        \item We have $\gamma\geq J_{\infty}(K)$ with $K \in \mathcal{K}$ if and only if there exists $P$ such that $(K,\gamma,P)\in\mathcal{L}_{\infty}$. This further implies $ \pi_{K,\gamma}(\mathcal{L}_{\infty})
=\operatorname{epi}_\geq(J_{\infty}). $
\item The mapping $\Phi_\infty$ in \cref{eq:mapping-hinf-state-feedback} is a $C^2$ (and in fact $C^\infty$) diffeomorphism between the lifted set $\mathcal{L}_{\infty}$ in \cref{eq:lifted-set-Hinf-state-feedback} and the convex set $\mathcal{F}_{\infty}$ in \cref{eq:convex-set-hinf-state-feedback}. 
    \end{enumerate}
    
\end{proposition}

The proof is not very difficult, but one needs to be careful about one technical subtlety:  \cref{lemma:bounded_real} already assumes that the system is stable, while the stability is a design constraint in \cref{eq:Hinf}. Therefore, the proof needs to explicitly ensure the stability of $A+BK$ for any $(K,\gamma,P)\in \mathcal{L}_\infty$, and also argue that $P$ obtained from \cref{eq:non-strict-hinf} in \cref{lemma:bounded_real} is positive definite, to show that $ \pi_{K,\gamma}(\mathcal{L}_{\infty})$ and 
$\operatorname{epi}_\geq(J_{\infty})$ are equal to each other. The details are provided in \cref{subsection:proofs-state-feedback-Hinf}.

\Cref{proposition:hinf-lifting} guarantees that $(\mathcal{L}_{\infty}, \mathcal{F}_{\infty},\{0\},\Phi_\infty)$ is an \ECL{} of $J_{\infty}(K)$ in \cref{eq:Hinf}. For this \ECL{}, we further have the following nice result, which is directly from the fact $\pi_{K,\gamma}(\mathcal{L}_{\infty})
=\operatorname{epi}_\geq(J_{\infty})$. 
\begin{proposition}\label{fact:Hinf-non-degenearte}
    Suppose \cref{assumption:controllability} holds, $B_w$ is full row rank, and {$Q \succ 0, R \succ 0$}. Then all stabilizing policies $K \in \mathcal{K}$ are non-degenerate with respect to the \ECL{} $(\mathcal{L}_{\infty}, \mathcal{F}_{\infty},\{0\},\Phi_\infty)$ constructed above. 
\end{proposition}

Considering that $J_\infty(K)$ is sudifferentially regular,\footnote{See Part I of this paper \cite{zheng2023benign} (Lemma 5.1 and Appendix D.5) for relevant discussions.} the following result is a direct corollary of \Cref{theorem:convex-equivalency,theorem:ECL-guarantee}.
\begin{corollary}
    Suppose \cref{assumption:controllability} holds, $B_w$ is full row rank, and {$Q \succ 0, R \succ 0$}. Then, the state feedback $\mathcal{H}_\infty$ policy optimization \cref{eq:Hinf} is equivalent to a convex problem in the sense that 
    \begin{equation} \label{eq:state-feedback-Hinf-LMI}
    \inf_{K \in \mathcal{K}}\ \  J_{\infty}(K) = \inf_{(\gamma, Y, X)\in \mathcal{F}_{\infty}} \, \gamma,
    \end{equation}
    Furthermore, any Clarke stationary point of \cref{eq:Hinf} is globally optimal.
\end{corollary}

Note that the convex problem in \cref{eq:state-feedback-Hinf-LMI} is indeed an LMI, which has been widely used in state feedback $\mathcal{H}_\infty$ control \cite{scherer2000linear}. The global optimality of Clarke stationary points is the continuous-time counterpart of the discrete-time result established in \cite{guo2022global}. Either of the formulations in \cref{eq:state-feedback-Hinf-LMI} might not be solvable, i.e., the infimum may not be achieved. In this case, the Clarke stationary point does not exist. An explicit single-input and single-output example is provided in \Cref{subsection:non-coercivity}.    

{
\begin{remark}[Lifting variable]
    Unlike the LQR case, the lifting variable $P$ in \cref{eq:lifted-set-Hinf-state-feedback} comes from an LMI due to the bounded real lemma (\Cref{lemma:bounded_real}). The solution $P$ in \cref{eq:lifted-set-Hinf-state-feedback} is in general not unique.
    Thus, the lifting procedure is necessary for $\mathcal{H}_\infty$ control. A similar strategy has been used in \cite{guo2022global} for discrete-time $\mathcal{H}_\infty$ control. This is also why we used a lifting variable in \Cref{example:Hinf}. \hfill $\square$
\end{remark}
}

\begin{remark}[The diffeomorphisms in state feedback policy optimization] \label{remark:state-feedback-policies}
The diffeomorphisms \cref{eq:Diffeomorphism-LQR,eq:mapping-hinf-state-feedback} are essentially in the same form. We have used~the classical change of variable $K = YX^{-1}$ \cite{khargonekar1991mixed,boyd1994linear,bernussou1989linear}, which linearizes many bilinear matrix inequalities in centralized state feedback controller synthesis (see Notes and References in \cite[Chapter 7]{boyd1994linear} for a historical perspective). A central block in all these bilinear matrix inequalities is the standard Lyapunov inequality $(A+BK)X + X(A+BK)^\tr \prec 0, X \succ 0$ (or $P(A+BK) + (A+BK)^\tr P \prec 0, P \succ 0$). The simple change of variables $Y =KX$ (or $K = YX^{-1}$ with $P = X^{-1}$) leads to the same LMI $AX +BY + (AX+BY)^\tr \prec 0, X \succ 0$. As we will see in the next section, the linearization process for dynamic output feedback policies is much more complicated. It is noted in \cite[Chapter 7]{boyd1994linear} that ``\textit{We suspect, however, that output feedback synthesis problems have high complexity and are therefore unlikely to be recast as LMI problems}''. Fortunately, the developments in \cite{scherer1997multiobjective,scherer2000linear} proposed a linearization procedure for dynamic output feedback policies. We will utilize those techniques in  \cite{scherer1997multiobjective,scherer2000linear} to construct \ECL{}s for LQG and $\mathcal{H}_\infty$ control.   
    \hfill $\square$
\end{remark}

\subsection{Output Feedback Policy Optimization} \label{subsection:dynamic-policies}
In many cases, we may not have full observation of the system state, and we need to use partial output observation for feedback control. In this section, we show that \ECL{} is also applicable to dynamic output feedback policies, including LQG control and robust $\mathcal{H}_\infty$ control.

Consider an LTI system with partial output observations
\begin{equation}\label{eq:Dynamic-output}
\begin{aligned}
\dot{x}(t) &= Ax(t)+Bu(t)+ B_w w(t), \\
y(t) &= Cx(t)+ D_v v(t),
\end{aligned}
\end{equation}
where $x(t) \in \mathbb{R}^n$ is the vector of state variables, $u(t)\in \mathbb{R}^m$ the vector of control inputs, $y(t) \in \mathbb{R}^p$ the vector of measured outputs available for feedback, and $w(t) \in \mathbb{R}^n, v(t)  \in \mathbb{R}^p$ are the disturbances on the system process and measurement at time $t$. Here we introduce the matrices $B_w\in\mathbb{R}^{n\times n}$ and $D_v\in\mathbb{R}^{p\times p}$ for a unified treatment of LQG and $\mathcal{H}_\infty$ control problems. For both~problems, their cost values depend on $B_w$ and $D_v$ only via $B_wB_w^\tr$ and $D_v D_v^\tr$. We thus define the weight~matrices
$
W \coloneqq B_w B_w^\tr,\, V\coloneqq D_vD_v^\tr,
$
and assume $B_w=W^{1/2}$ and $D_v=V^{1/2}$.
The following assumption~is~standard. 
\begin{assumption} \label{assumption:performance-weights}
The weight matrices satisfy $Q \succeq 0, R \succ 0, W \succeq 0, V \succ 0$. Furthermore, $(A,W^{1/2})$ is controllable, and $(Q^{1/2}, A)$ is observable. Also, $(A,B)$ is controllable and $(C,A)$ is observable.
\end{assumption}

We consider the same performance signal $z(t)$ in \cref{eq:performance-signal}. To properly regulate the signal $z(t)$, we design a linear full-order dynamic feedback policy of the form
\begin{equation}\label{eq:Dynamic_Controller}
    \begin{aligned}
        \dot \xi(t) &= A_{\mK}\xi(t) + B_{\mK}y(t), \\
        u(t) &= C_{\mK}\xi(t) + D_{\mK}y(t),
    \end{aligned}
\end{equation}
where $\xi(t) \in \mathbb{R}^n$ is the internal state, and $A_{\mK},B_{\mK},C_{\mK}$ and $D_{\mK}$ are matrices of proper dimensions that specify the policy dynamics. We parameterize dynamic feedback policies of the form~\cref{eq:Dynamic_Controller} by
\begin{equation} \label{eq:dynamic-policies-Dk}
\mK =
\begin{bmatrix}
D_\mK & C_\mK \\ B_\mK & A_\mK
\end{bmatrix} \in \mathbb{R}^{(m+n)\times (p+n)}.
\end{equation}
Combining~\cref{eq:Dynamic_Controller} with~\cref{eq:Dynamic-output} via simple algebra leads to the closed-loop system
\begin{subequations}\label{eq:closed-loop}
\begin{equation}
\begin{aligned}
\frac{d}{dt} \begin{bmatrix} x \\ \xi \end{bmatrix} &=
A_{\mathrm{cl}}(\mK)
\begin{bmatrix} x \\ \xi \end{bmatrix}
+ B_{\mathrm{cl}}(\mK)
\begin{bmatrix} w \\ v \end{bmatrix}, \\
z &= C_{\mathrm{cl}}(\mK)
\begin{bmatrix} x \\ \xi \end{bmatrix}
+ D_{\mathrm{cl}}(\mK)
\begin{bmatrix} w \\ v \end{bmatrix},
\end{aligned}
\end{equation}
where we denote the closed-loop system matrices by
\begin{equation}\label{eq:closed-loop-matrices}
\begin{aligned}
A_{\mathrm{cl}}(\mK)
\coloneqq\  &
\begin{bmatrix}
A + BD_\mK C & BC_\mK \\ B_\mK C & A_\mK
\end{bmatrix}, 
& B_{\mathrm{cl}}(\mK)
\coloneqq\  & \begin{bmatrix} W^{1/2} & BD_{\mK}V^{1/2} \\ 0 & B_{\mK}V^{1/2}  \end{bmatrix}, \\
C_{\mathrm{cl}}(\mK)
\coloneqq\  &
 \begin{bmatrix}
        Q^{1/2} & 0 \\
        R^{1/2}D_{\mK}C & R^{1/2}C_{\mK}
    \end{bmatrix}, &
D_\mathrm{cl}(\mK) \coloneqq\ &
\begin{bmatrix} 0 & 0\\ 0 & R^{1/2}D_{\mK}V^{1/2}\end{bmatrix}.
\end{aligned}
\end{equation}
\end{subequations}

A dynamic feedback policy~\cref{eq:Dynamic_Controller} is said to \emph{internally stabilize} the plant~\cref{eq:Dynamic-output} if the 
the closed-loop matrix $A_{\mathrm{cl}}(\mK)$ is (Hurwitz) stable. We denote%
\begin{subequations} \label{eq:internallystabilizing}
\begin{align}
    \mathcal{C}_{n} & \coloneqq \left\{
    \left.\mK=\begin{bmatrix}
    D_{\mK} & C_{\mK} \\
    B_{\mK} & A_{\mK}
    \end{bmatrix}
    \in \mathbb{R}^{(m+n) \times (p+n)} \,\right|
    A_{\mathrm{cl}}(\mK)\text{ is stable}\right\}, \label{eq:internallystabilizing-a} \\
    \mathcal{C}_{n,0} & \coloneqq 
    \mathcal{C}_n\cap\mathcal{V}_{n,0},
    \quad\text{where}\ \ 
    \mathcal{V}_{n,0}
    \coloneqq
    \left\{\left.
    \begin{bmatrix}
    D_{\mK} & C_{\mK} \\
    B_{\mK} & A_{\mK}
    \end{bmatrix}
    \in \mathbb{R}^{(m+n)\times(p+n)}\,\right| D_{\mK} = 0_{m\times p}\right\}. \label{eq:internallystabilizing-b}
\end{align}
\end{subequations}
The policies in $\mathcal{C}_n$ are \textit{proper} and those in $\mathcal{C}_{n,0}$ are \textit{strictly proper}. It can be shown that $\mathcal{C}_n$ is an open subset of $\mathbb{R}^{(m+n)\times(p+n)}$, and $\mathcal{C}_{n,0}$ is an open subset of the linear space $\mathcal{V}_{n,0}$. The transfer matrix from the disturbance $d(t)=\begin{bmatrix}
w(t) \\ v(t)
\end{bmatrix}$ to the performance signal $z(t)$ becomes 
\begin{equation} \label{eq:transfer-function-zd}
    \mathbf{T}_{{zd}}(\mK, s) = C_{\mathrm{cl}}(\mK)
\left(sI - A_{\mathrm{cl}}(\mK)
\right)^{-1}
B_{\mathrm{cl}}(\mK)
+D_{\mathrm{cl}}(\mK).
\end{equation}
For notational simplicity, in later construction of ECLs, we define the following matrices corresponding to the general dynamics (see \cite[Appendix A.6]{zheng2023benign} for more details):
\begin{equation}
\label{eq:notations_general_H2}
    B_1=\begin{bmatrix} W^{\frac{1}{2}} & 0_{n\times p} \end{bmatrix}, \; B_2=B,  \; C_1=\begin{bmatrix} Q^{\frac{1}{2}} \\ 0_{m\times n} \end{bmatrix},\;  C_2=C,  \; 
 D_{12}=\begin{bmatrix} 0_{n\times m} \\ R^{\frac{1}{2}} \end{bmatrix}, \;D_{21}=\begin{bmatrix} 0_{p\times n} & V^{\frac{1}{2}} \end{bmatrix}.
\end{equation}

\subsubsection{Linear Quadratic Gaussian (LQG)} \label{subsection:ECL-LQG}

When $w(t)$ and $v(t)$ are white Gaussian noises with identity intensity matrices, i.e., $\mathbb{E}[w(t)w(\tau)]=\delta(t-\tau)I_n$ and $\mathbb{E}[v(t)v(\tau)]=\delta(t-\tau)I_p$, we consider an averaged mean performance
\begin{equation*}
\mathfrak{J}_{\LQG}
\coloneqq \lim_{T \rightarrow +\infty }\mathbb{E}\!\left[ \frac{1}{T} \int_{0}^T  z(t)^\tr z(t)\, dt\right]
=\lim_{T \rightarrow +\infty }\mathbb{E} \!\left[\frac{1}{T}\int_{0}^T \left(x^\tr Q x + u^\tr R u\right)dt\right].
\end{equation*}
The classical LQG control problem can be formulated as
\begin{equation*}
    \begin{aligned}
        \min_{\mK \in \mathbb{R}^{(m+n)\times (p+n)}} \quad & \mathfrak{J}_{\LQG} \\
        \text{subject to} \quad & ~\cref{eq:Dynamic-output},\cref{eq:Dynamic_Controller}, \mK \in \mathcal{C}_{n,0}.
    \end{aligned}
\end{equation*}
This problem is equivalent to an $\mathcal{H}_2$ optimization problem below (see Part I \cite{zheng2023benign} for further details)
\begin{equation}\label{eq:LQG_policy_optimization}
\begin{aligned}
\min_{\mK}\ \ & J_{\LQG,n}(\mK):= \|\mathbf{T}_{zd}(\mK,s)\|_{\mathcal{H}_2} \\
\text{subject to}\ \ &  \mK \in \mathcal{C}_{n,0}.
\end{aligned}
\end{equation}

It is known that the policy optimization for LQG \cref{eq:LQG_policy_optimization} is smooth and nonconvex, but all stationary points corresponding to minimal controllers are globally optimal \cite[Theorem 6]{zheng2021analysis}. Part I of this paper in \cite[Theorem 4.2]{zheng2023benign} has revealed that any \textit{non-degenerate stationary points are globally optimal}; see \cite{zheng2023benign} for further landscape discussions. 
This subsection aims to prove \cite[Theorem 4.2]{zheng2023benign} using our \ECL{} framework. Our strategy is to construct an \ECL{} for \cref{eq:LQG_policy_optimization} from the classical change of variables in \cite{scherer2000linear}, then \Cref{theorem:convex-equivalency,theorem:ECL-guarantee} directly apply.

Similar to \Cref{subsection:static-policies}, our construction has three main steps.

\vspace{3pt}
\noindent \textbf{Step 1: Lifting}. Motivated by \cref{eq:nonstrict-lmi-h2} in \Cref{lemma:H2norm}, we first introduce a lifted set $\mathcal{L}_{\LQG}$ as 
\begin{equation} \label{eq:LQG-lift}
\mathcal{L}_{\LQG}=  \left\{ (\mK, \gamma, \underbrace{P, \Gamma}_{\xi}) \left|
       \begin{array}{c} 
       \mK\in\mathcal{V}_{n,0},\; \gamma\in\mathbb{R},\;
         P \in \mathbb{S}^{2n}_{++}, \; P_{12} \in \mathrm{GL}_n, \;\Gamma\in\mathbb{S}_+^{n+m}, \\
         \begin{aligned}
        &\begin{bmatrix} A_{\mathrm{cl}}(\mK)^\tr P+PA_{\mathrm{cl}}(\mK) & PB_{\mathrm{cl}}(\mK) \\ B_{\mathrm{cl}}(\mK)^\tr P & -\gamma I \end{bmatrix} \preceq 0, \\
        &\begin{bmatrix} P & C_{\mathrm{cl}}(\mK)^\tr \\ C_{\mathrm{cl}}(\mK) & \Gamma \end{bmatrix} \succeq 0,\;  \mathrm{tr}(\Gamma) \leq \gamma
        \end{aligned}
        \end{array}
        \right.\right\},
\end{equation}
where $A_{\mathrm{cl}}(\mK), B_{\mathrm{cl}}(\mK), C_{\mathrm{cl}}(\mK)$ are the closed-loop matrices of the policy parameter $\mK$ defined in \Cref{eq:closed-loop-matrices}, and $P_{12}$ denotes the off-diagonal block of $P$ formed by elements in the first $n$ rows and last $n$ columns of $P$. The extra variables $(P, \Gamma)$ play the role of $\xi$ as lifting variables in \ECL.

\vspace{3pt}
\noindent \textbf{Step 2: Convex and auxiliary sets}. We let the auxiliary set $\mathcal{G}_{\LQG}$ be the set of $n \times n$ invertible matrices
$
\mathrm{GL}_n
$, which captures the effect of similarity transformations. We then let the convex set be
\begin{align*}
    \mathcal{F}_{\LQG} = \left\{  (\gamma, \underbrace{\Lambda,X, Y, \Gamma}_{\zeta_1} ) \left|
       \begin{array}{c}
       \gamma\in\mathbb{R},\;\Lambda\in\mathcal{V}_{n,0},\;
       \begin{bmatrix}
        X & I \\
        I & Y
\end{bmatrix}\in\mathbb{S}_{++}^{2n},\;\Gamma\in\mathbb{S}^{n+m}_{+}, \\
       \begin{aligned}
       & {\mathcal{A}}(\gamma, \Lambda,X,Y,\Gamma) \preceq 0, \\
       & {\mathcal{B}}(\gamma, \Lambda,X,Y,\Gamma) \succeq 0,\;
       \mathrm{tr} (\Gamma)\leq\gamma
       \end{aligned}
       \end{array}
        \right.\right\},
\end{align*}
where ${\mathcal{A}}$ and ${\mathcal{B}}$ are two  affine operators
defined as
\begingroup
    \setlength\arraycolsep{2pt}
\def\arraystretch{0.85} 
\begin{align*}
    {\mathcal{A}}\!\left(\gamma, \!\begin{bmatrix}
    0 & F \\
    H & M
    \end{bmatrix}\!,X,Y,\Gamma\right)& \!\coloneqq\!
    \begin{bmatrix} AX \!+\! B_2F \!+\! (AX \!+\! B_2F)^\tr  & M^\tr \!+\! A & B_1 \\
            *  & YA \!+\! HC_2 \!+\! (YA \!+\! HC_2)^\tr  & YB_1\!+\!HD_{21}  \\       
            *  & *  & -\gamma I\end{bmatrix}, \\
    {\mathcal{B}}\!\left(\gamma, \!\begin{bmatrix}
    0 & F \\
    H & M
    \end{bmatrix}\!,X,Y,\Gamma\right)&\!\coloneqq\!
    \begin{bmatrix} \ \ \ X\ \ \  & \ \ \ I\ \ \  & (C_1X\!+\!D_{12}F)^\tr  \\ * & \ \ \ Y\ \ \  & C_1^\tr \\ * & * &\Gamma \end{bmatrix},
\end{align*}
where $*$ denotes the symmetric part. 
These two linear operators ${\mathcal{A}}$ and ${\mathcal{B}}$ resemble the matrix inequalities in \Cref{eq:LQG-lift}. 
\endgroup 

\vspace{3pt}
\noindent \textbf{Step 3: Diffeomorphism}. We next build a smooth bijection from $\mathcal{L}_{\LQG}$ to $\mathcal{F}_{\LQG}\times\mathrm{GL}_n$. Define the mapping $\Phi_{\LQG}$ by
\begin{equation} \label{eq:diffeomorphism-LQG}
\Phi_{\LQG}(\mK,\gamma,P,\Gamma)=\Bigg(\gamma, \underbrace{\begin{bmatrix}
        0 & C_{\mK}(P^{-1})_{21} \\ P_{12}B_{\mK} & \Phi_M
    \end{bmatrix},(P^{-1})_{11},P_{11},\Gamma}_{\zeta_1},\; \underbrace{P_{12}}_{\zeta_2}\Bigg),
    \quad \forall (\mK, \gamma, P, \Gamma)\in {\mathcal{L}}_{\LQG},
\end{equation} 
where
\begin{equation*}
\begin{aligned}        
\Phi_M&= P_{12}B_{\mK}C_2(P^{-1})_{11}+P_{11}B_2C_{\mK}(P^{-1})_{21}+P_{11}A(P^{-1})_{11}+P_{12}A_{\mK}(P^{-1})_{21}.
\end{aligned}
\end{equation*}
Furthermore, we can construct its inverse explicitly (see \Cref{appendix:LQG-control} for details): Define the mapping $\Psi_{\LQG}$ by
\begin{equation} \label{eq:Psi-LQG}
    \Psi_{\LQG}(\mZ,\Xi)=\left(\Psi_{\mK},\gamma, \Psi_P,\Gamma\right),
    \qquad\forall \, \mZ = \left(\gamma, \begin{bmatrix}
        0 & F\\H & M
    \end{bmatrix}, X,Y,\Gamma \right) \in \mathcal{F}_{\LQG},\,\Xi \in \mathrm{GL}_n,
\end{equation}
where
\begin{equation*}
\begin{aligned}
    \Psi_{\mK}&= \begin{bmatrix} I & 0 \\ YB_2 & \Xi \end{bmatrix}^{-1}\begin{bmatrix} 0 & F \\ H & M-YAX \end{bmatrix}\begin{bmatrix} I & C_2X \\ 0 & -\Xi^{-1}(Y-X^{-1})X \end{bmatrix}^{-1},\\
    \Psi_P&=\begin{bmatrix} Y & \Xi \\ \Xi^\tr  & \Xi^\tr (Y-X^{-1})^{-1}\Xi \end{bmatrix}.
\end{aligned}
\end{equation*}

The constructions of $\Phi_{\LQG}$ and $\Psi_{\LQG}$ are highly nontrivial, but have become classical in control; see e.g., \cite{scherer2000linear}. Our next result shows that  $\Phi_{\LQG}$ is a diffeomorphism from $\mathcal{L}_{\LQG}$ to $\mathcal{F}_{\LQG}\times \mathrm{GL}_{n}$. 

\begin{proposition} \label{lemma:inclusion-projection-LQG}
    Under \cref{assumption:performance-weights}, the canonical projection of $\mathcal{L}_{\LQG}$ in \cref{eq:LQG-lift} onto $(\mK, \gamma)$ satisfies
    $$
    \operatorname{epi}_>(J_{\LQG,n})  \subseteq \pi_{\mK, \gamma} (\mathcal{L}_{\LQG}) \subseteq \operatorname{cl}\, \operatorname{epi}_{\geq}(J_{\LQG,n}). 
    $$
    Furthermore, the mapping $\Phi_{\LQG}$ given by \cref{eq:diffeomorphism-LQG} is a $C^2$ (and in fact $C^\infty$) diffeomorphism from $\mathcal{L}_{\LQG}$ to $\mathcal{F}_{\LQG}\times\mathrm{GL}_{n}$, and its inverse $\Phi_{\LQG}^{-1}$ is given by $\Psi_{\LQG}$.
\end{proposition}

Unlike the state feedback cases in \Cref{subsection:static-policies}, the proof of \cref{lemma:inclusion-projection-LQG} is much more involved for the reasons highlighted in the remark below. We present the proof details in \Cref{appendix:LQG-control}. 

\begin{remark}[Internal stability] \label{remark:internal-stability}
While the lifted set $\mathcal{L}_{\LQG}$ in \cref{eq:LQG-lift} is motivated by \cref{eq:nonstrict-lmi-h2} in \Cref{lemma:H2norm}, there is one main technical subtlety: \Cref{lemma:H2norm} already assume the system to be stable, but the internal stability $\mK \in \mathcal{C}_{n,0}$ is a domain constraint in \cref{eq:LQG_policy_optimization}. Indeed, the policy  $\mK$ from $\pi_{\mK, \gamma} (\mathcal{L}_{\LQG})$ might not internally stabilize the plant (\textit{the closed-loop system $A_{\mathrm{cl}}(\mK)$ may have eigenvalues with zero real parts}), and some point in $\pi_{\mK, \gamma} (\mathcal{L}_{\LQG})$ may be only in the closure $\operatorname{cl}\, \operatorname{epi}_{\geq}(J_{\LQG,n})$ (note that $J_{\LQG,n}$ is not coercive). Another difference is that we require $P_{12} \in \mathrm{GL}_n$ which does not appear in \cref{eq:nonstrict-lmi-h2}. As highlighted in Part I, this invertibility requirement on $P_{12}$ is very important for constructing a diffeomorphism; also see \Cref{eq:diffeomorphism-LQG}. These two differences complicate the proof of \Cref{lemma:inclusion-projection-LQG}. In \Cref{appendix:dynamic-policies}, we will present a key technical lemma (\Cref{lemma:ECL_second_inclusion}) and the proof details. \hfill $\square$
\end{remark}

\Cref{lemma:inclusion-projection-LQG} guarantees that $(\mathcal{L}_{\LQG}, \mathcal{F}_{\LQG},\mathrm{GL}_n,\Phi_{\LQG})$ constructed above is an \ECL{} of the LQG policy optimization $J_{\LQG,n}(\mK)$ in \Cref{eq:LQG_policy_optimization}. For convenience, we restate the definition~of~non-degenerate LQG policies from \cite[Definition 3.1]{zheng2023benign} below, which is the same as \Cref{definition:degenerate-points}.

\begin{definition} \label{definition:non-degenerate-controller-LQG}
    A full-order dynamic policy $\mK \in {\mathcal{C}}_{n,0}$ is called \textit{non-degenerate for LQG} if there exists a $P \in \mathbb{S}^{2n}_{++}$ with $P_{12} \in \mathrm{GL}_{n}$ and $\Gamma$ such that $(\mK, J_{\LQG,n}(\mK), P, \Gamma) \in \mathcal{L}_{\LQG}$. 
\end{definition}

We refer the interested reader to Part I \cite{zheng2023benign} for more discussions on non-degenerate LQG policies. Then, \cite[Theorem 4.2]{zheng2023benign} is a direct corollary of \Cref{theorem:convex-equivalency,theorem:ECL-guarantee}.
\begin{corollary}
    Under \cref{assumption:performance-weights}, the LQG policy optimization \cref{eq:LQG_policy_optimization} is equivalent to a convex problem in the sense that 
    \begin{equation} \label{eq:LQG-LMI}
    \inf_{\mK \in \mathcal{C}_{n,0}}\ \  J_{\LQG,n}(\mK) = \inf_{(\gamma, \Lambda,X, Y, \Gamma) \in \mathcal{F}_{\LQG}} \, \gamma,
    \end{equation}
    Furthermore, for a non-degenerate policy $\mK \in \mathcal{C}_{n,0}$, if it is a stationary point, i.e. $0 \in \nabla J_{\LQG,n}(\mK)$, then it is globally optimal for \cref{eq:LQG_policy_optimization}.
\end{corollary}

We remark that the equivalence in \cref{eq:LQG-LMI} is essentially the same as \cite[Corollary 4.7]{scherer2000linear}, but with one main difference that all the LMIs in \cite[Corollary 4.7]{scherer2000linear} are strict (which naturally avoids the issue of internal stability in \Cref{remark:internal-stability}) while we used non-strict LMIs in \cref{eq:LQG-lift}. This allows us to derive the global optimality of non-degenerate stationary points. A direct application of \cite[Corollary 4.7]{scherer2000linear} can only characterize the upper bound of $J_{\LQG,n}(\mK)$, but not the function $J_{\LQG,n}(\mK)$ itself. Thus, the global optimality of non-degenerate stationary points cannot be derived from \cite[Corollary 4.7]{scherer2000linear}.

\subsubsection{$\mathcal{H}_\infty$  Output Feedback Control} \label{subsection:ECL-Hinf}

When $w(t)$ and $v(t)$ are deterministic disturbances, we consider the worst-case performance in an adversarial setup. Assume that the system starts from a zero initial state $x(0)=0$. 
Denote $
d(t) = \begin{bmatrix}
w(t) \\ v(t)
\end{bmatrix} \in \mathbb{R}^{p + n}
$, and consider the worst-case performance when the disturbance signal $d(t)$ has bounded energy less than or equal to $1$:
\begin{equation*}
    \mathfrak{J}_{\infty} := \sup_{\|d\|_{\ell_2} \leq 1} \; \int_0^\infty z(t)^\tr z(t)\, dt = \sup_{\|d\|_{\ell_2} \leq 1} \int_{0}^\infty \left(x^\tr Q x + u^\tr R u\right)dt. 
\end{equation*}
The $\mathcal{H}_\infty$ output feedback control problem can be formulated as
\begin{equation*}
    \begin{aligned}
        \inf_{\mK} \quad & \mathfrak{J}_{\infty} \\
        \text{subject to} \quad & ~\cref{eq:Dynamic-output},\cref{eq:Dynamic_Controller}, \;
        x(0)=0, \mK \in \mathcal{C}_n
    \end{aligned}
\end{equation*}
This problem is equivalent to the following $\mathcal{H}_\infty$ policy optimization problem (see Part I \cite{zheng2023benign} for further details):
\begin{equation}\label{eq:Hinf_policy_optimization}
\begin{aligned}
\inf_{\mK}\ \ & J_{\infty,n}(\mK):= \|\mathbf{T}_{zd}(\mK,s)\|_{\mathcal{H}_\infty} \\
\text{subject to}\ \ &  \mK \in \mathcal{C}_{n}.
\end{aligned}
\end{equation}
It is known that the $\mathcal{H}_\infty$ policy optimization \cref{eq:Hinf_policy_optimization} is nonsmooth and nonconvex. Part I of this paper has revealed that all \textit{non-degenerate Clarke stationary points} are globally optimal to  \cref{eq:Hinf_policy_optimization} \cite[Theorem 5.2]{zheng2023benign}.  This subsection aims to prove \cite[Theorem 5.2]{zheng2023benign} using our \ECL{} framework.

Similar to the LQG case, our \ECL{} construction has three main steps.

\vspace{3pt}
\noindent \textbf{Step 1: Lifting}. Motivated by \Cref{eq:non-strict-hinf} in \cref{lemma:bounded_real}, we first introduce a lifted set 
\begin{equation} \label{eq:Hinf-lift}
\mathcal{L}_{\infty,\mathrm{d}}=  \left\{ (\mK, \gamma, \underbrace{P}_{\xi}) \left|\,
       \begin{aligned} 
       & \mK \in\mathbb{R}^{(m+n)\times(p+n)}, \gamma\in\mathbb{R}, P
\in\mathbb{S}_{++}^{2n},\ 
P_{12}\in\mathrm{GL}_{n}, \\
         & \begin{bmatrix}
A_{\mathrm{cl}}(\mK)^\tr P \!+\! P A_{\mathrm{cl}}(\mK) & PB_{\mathrm{cl}}(\mK) & C_{\mathrm{cl}}(\mK)^\tr \\
B_{\mathrm{cl}}(\mK)^\tr P & -\gamma I & D_{\mathrm{cl}}(\mK)^\tr \\
C_{\mathrm{cl}}(\mK) & D_{\mathrm{cl}}(\mK) & -\gamma I
\end{bmatrix} \preceq 0
        \end{aligned}
        \right.\right\}.
\end{equation}
where $A_{\mathrm{cl}}(\mK), B_{\mathrm{cl}}(\mK), C_{\mathrm{cl}}(\mK)$ are the closed-loop matrices of the policy parameter $\mK$ defined in \Cref{eq:closed-loop-matrices}, and $P_{12}$ denotes the off-diagonal block of $P$ formed by elements in the first $n$ rows and last $n$ columns of $P$. The extra variable $P$ plays the role of $\xi$ as lifting variables in \ECL.

\vspace{3pt}
\noindent \textbf{Step 2: Convex and auxiliary sets}. We let the auxiliary set be the set of $n \times n$ invertible matrices
$
\mathrm{GL}_n,
$
and let the convex set be given by
\begin{align*}
    \mathcal{F}_{\infty,\mathrm{d}}= \left\{  ( \gamma, \underbrace{\Lambda,X, Y}_{\zeta_1} ) \left| \,
       \begin{aligned}
       &\gamma\in\mathbb{R},\Lambda\in\mathbb{R}^{(m+n)\times(p+n)},\begin{bmatrix}
X & I \\
I & Y
\end{bmatrix}\in\mathbb{S}_{++}^{2n}, \\
& \mathscr{M}(\gamma,\Lambda,X,Y)\preceq 0
        \end{aligned}
        \right.\right\},
\end{align*}
where $\mathscr{M} (\gamma,\Lambda,X,Y)$ is an affine operator defined as (we have denoted $\Lambda = \begin{bmatrix}
G & F \\
H & M
\end{bmatrix}$)
\begingroup
    \setlength\arraycolsep{3pt}
\def\arraystretch{0.9} \begin{equation*}
\begin{aligned}
& \mathscr{M}\left(\gamma,\begin{bmatrix}
G & F \\
H & M
\end{bmatrix},X,Y\right) \\
\coloneqq &
\begin{bmatrix}
AX + B_2F + (AX + B_2F)^\tr & M^\tr + A + B_2GC_2 & B_1 + B_2GD_{21} & (C_1 X \!+\! D_{12}F)^\tr \\
* & YA \!+\! HC_2 \!+\! (YA \!+\! HC_2)^\tr & YB_1 \!+\! HD_{21} & (C_1 \!+\! D_{12}GC_2)^\tr \\
* & * & -\gamma I & (D_{11} \!+\! D_{12}GD_{21})^\tr \\
* & * & * & -\gamma I
\end{bmatrix}. 
\end{aligned}
  \end{equation*}  \endgroup 

\vspace{3pt}

\noindent \textbf{Step 3: Diffeomorphism}. 
We define the mapping $\Phi_{\infty,\mathrm{d}}$ by
\begin{subequations} \label{eq:diffeomorphism-Hinf}
\begin{equation} \label{eq:Hinf-Phi}
\Phi_{\infty,\mathrm{d}}(\mK,\gamma,P) = 
\Bigg(\gamma, \underbrace{\begin{bmatrix}
        D_\mK & \Phi_F\\\Phi_H & \Phi_M
    \end{bmatrix}, (P^{-1})_{11},P_{11}}_{\zeta_1},\; \underbrace{P_{12}}_{\zeta_2}
\Bigg),
\quad(\mK,\gamma,P)\in\mathcal{L}_{\infty,\mathrm{d}},
\end{equation}
where 
\begin{equation}
\begin{aligned}
\Phi_M ={} &
P_{12}B_\mK C_2 (P^{-1})_{11}
+ P_{11} B_2 C_\mK (P^{-1})_{21} 
 + P_{11}(A+B_2 D_\mK C_2)(P^{-1})_{11} 
 + P_{12}A_\mK (P^{-1})_{21},\\
 \Phi_{H} ={} &
P_{11}B_2 D_\mK + P_{12} B_\mK,\\
\Phi_{F} ={} &
D_\mK C_2 (P^{-1})_{11} + C_\mK (P^{-1})_{21}.
\end{aligned}
\end{equation}
Furthermore, we can construct its inverse explicitly (see \Cref{appendix:Hinf-control} for details): Define the mapping $\Psi_{\infty,\mathrm{d}}$ by
\begin{equation} \label{eq:Hinf-Psi}
    \Psi_{\infty,\mathrm{d}}(\mZ,\Xi):=\left(\Psi_{\mK},\gamma, \Psi_P\right)
\end{equation}
for any $\mZ = \left(\gamma, \begin{bmatrix}
        G & F\\H & M
    \end{bmatrix}, X,Y\right) \in \mathcal{F}_{\infty,\mathrm{d}}$ and $\Xi \in \mathrm{GL}_n$,
where the components are given by
\begin{equation}
\begin{aligned}
    \Psi_{\mK}&= \begin{bmatrix} I & 0 \\ YB & \Xi \end{bmatrix}^{-1}\begin{bmatrix} G & F \\ H & M-YAX \end{bmatrix}\begin{bmatrix} I & CX \\ 0 & -\Xi^{-1}(Y-X^{-1})X  \end{bmatrix}^{-1}, \\ 
    \Psi_P&=\begin{bmatrix} Y & \Xi \\ \Xi^\tr  & \Xi^\tr (Y-X^{-1})^{-1}\Xi \end{bmatrix}.
\end{aligned}
\end{equation}
\end{subequations}

The following result shows that  $(\mathcal{L}_{\infty,\mathrm{d}},\mathcal{F}_{\infty,\mathrm{d}},\mathrm{GL}_n,\Phi_{\infty,\mathrm{d}})$ constructed above is an \ECL{} for $J_{\infty,n}$.

\begin{proposition} \label{lemma:inclusion-projection-Hinf}
    Under \cref{assumption:performance-weights}, the canonical projection of $\mathcal{L}_{\infty,\mathrm{d}}$ in \cref{eq:Hinf-lift} onto $(\mK, \gamma)$ satisfies
    $$
    \operatorname{epi}_>(J_{\infty,n})  \subseteq \pi_{\mK, \gamma} (\mathcal{L}_{\infty,\mathrm{d}}) \subseteq \operatorname{cl}\, \operatorname{epi}_{\geq}(J_{\infty,n}). 
    $$
    Furthermore, the mapping $\Phi_{\infty,\mathrm{d}}$ \cref{eq:Hinf-Phi} is a $C^2$ (and in fact $C^\infty$) diffeomorphism from $\mathcal{L}_{\infty,\mathrm{d}}$ to $\mathcal{F}_{\infty,\mathrm{d}}\times\mathrm{GL}_{n}$, and its inverse $\Phi_{\infty,\mathrm{d}}^{-1}$ is given by $\Psi_{\infty,\mathrm{d}}$.
\end{proposition}

The proof of \cref{lemma:inclusion-projection-Hinf} is technically involved due to a similar issue of internal stability as \Cref{remark:internal-stability}: A policy $\mK$ in \cref{eq:Hinf-lift} may not internally stabilize the plant \cref{eq:Dynamic-output}, i.e., we may have $\mK \notin \mathcal{C}_n$. We present the proof details in \Cref{appendix:Hinf-control}.

\Cref{lemma:inclusion-projection-Hinf} guarantees that $(\mathcal{L}_{\infty,\mathrm{d}}, \mathcal{F}_{\infty,\mathrm{d}},\mathrm{GL}_n,\Phi)$ constructed above is an \ECL{} of  $J_{\infty,n}(\mK)$ in \Cref{eq:Hinf_policy_optimization}. For convenience, we restate the definition~of~non-degenerate $\mathcal{H}_\infty$ policies from \cite[Definition 3.2]{zheng2023benign} below, which is the same as \Cref{definition:degenerate-points}.

\begin{definition} \label{definition:non-degenerate-controller-Hinf}
    A full-order dynamic policy $\mK \in {\mathcal{C}}_{n}$ is called \textit{non-degenerate} for $\mathcal{H}_\infty$ control if there exists a $P \in \mathbb{S}^{2n}_{++}$ with $P_{12} \in \mathrm{GL}_{n}$ such that $(\mK, J_{\infty,n}(\mK), P) \in \mathcal{L}_{\infty,\mathrm{d}}$ in \Cref{eq:Hinf-lift}. 
\end{definition}

Then,  \cite[Theorem 5.2]{zheng2023benign} is a direct corollary of \Cref{theorem:convex-equivalency,theorem:ECL-guarantee}.
\begin{corollary}
    Under \cref{assumption:performance-weights}, the $\mathcal{H}_\infty$ policy optimization \cref{eq:Hinf_policy_optimization} is equivalent to a convex problem in the sense that 
    \begin{equation} \label{eq:Hinf-LMI}
    \inf_{\mK \in \mathcal{C}_{n}}\ \  J_{\infty,n}(\mK) = \inf_{(\gamma, \Lambda, X, Y) \in \mathcal{F}_{_{\infty,\mathrm{d}}}} \, \gamma,
    \end{equation}
    Furthermore, for a non-degenerate policy $\mK \in \mathcal{C}_{n}$, if it is a Clarke stationary point, i.e. $0 \in \partial J_{\infty,n}(\mK)$, then it is globally optimal for \cref{eq:Hinf_policy_optimization}.
\end{corollary}

We have used the fact that the $\mathcal{H}_\infty$ cost function $J_{\infty,n}$ is subdifferentially regular (see Part I \cite{zheng2023benign} for more discussions). The equivalency in \cref{eq:Hinf-LMI} is essentially the same as \cite[Chapter 4.2.3]{scherer2000linear}. However, the global optimality of non-degenerate $\mathcal{H}_\infty$ policies cannot be derived from \cite[Chapter 4.2.3]{scherer2000linear} due to its use of strict LMIs. 

\begin{remark}[The diffeomorphisms in output feedback policy optimization]
    The diffeomorphisms \cref{eq:diffeomorphism-LQG,eq:diffeomorphism-Hinf} are essentially the same, which are derived from \cite{scherer2000linear}. They are much more complicated than the state feedback cases in \cref{subsection:static-policies}. This complexity arises even when we deal with the internal stability constraint using Lyapunov inequality:
\begin{equation*}
    \begin{aligned}
    &\begin{bmatrix}
    A+BD_{\mK}C & BC_{\mK} \\ B_{\mK}C & A_{\mK}
    \end{bmatrix}\; \text{is stable }  \\
    \Longleftrightarrow \quad &\exists P \succ 0, \, P\begin{bmatrix}
    A+BD_{\mK}C & BC_{\mK} \\ B_{\mK}C & A_{\mK}
    \end{bmatrix} + \begin{bmatrix}
    A+BD_{\mK}C & BC_{\mK} \\ B_{\mK}C & A_{\mK}
    \end{bmatrix}^\tr P \prec 0, 
    \end{aligned}
    \end{equation*}
    where the coupling between the auxiliary variable $P$ and the controller parameters $A_{\mK}, B_{\mK}, C_{\mK}, D_{\mK}$ are much more involved. 
    Note that this is also a bilinear matrix inequality  (as the matrix depends on $P$ and $\mK$ bilinearly), but its linearization procedure for this bilinear matrix inequality requires the complicated change of variables in \cref{eq:diffeomorphism-LQG,eq:diffeomorphism-Hinf} (see \cite[Lemma 4.1]{scherer2000linear} for more details).
    \hfill $\square$    
\end{remark}

\subsection{Policy Optimization for Distributed Control with Quadratic Invariance} \label{subsection:distributed-policies}

  All discussions in \cref{subsection:static-policies,subsection:dynamic-policies} so far focus on centralized policies, which admit equivalent convex reformulations using the classical change of variables. Applications arising from the control of networked systems e.g., smart grid \cite{dorfler2014sparsity} or automated highways \cite{li2017dynamical}, may not have centralized information available for feedback, due to privacy concerns, geographic distance, or limited communication. The lack of centralized information can enormously complicate the design of optimal distributed policies. The seminal work \cite{rotkowitz2005characterization} has introduced the notion of \textit{Quadratic Invariance} (QI) to quantify an {algebraic} relation that allows for an equivalent convex reformulation in the frequency domain; see also \cite{furieri2020sparsity} for related discussions.   

Here, we show that our \ECL{} framework is also applicable to a class of distributed control problems when QI holds. This class of problems was first discussed in \cite{furieri2020learning}, where a model-free learning-based algorithm was also introduced. 
In particular, we focus on a finite-time horizon distributed optimal control problem. In this case, we may allow the linear system to be time-varying, and the stability constraint is not necessary. 
Consider a linear time-varying system in discrete-time
	\begin{equation}
	\begin{aligned}
	\label{eq:sys_disc}
	x_{t+1}&=A_tx_t+B_tu_t+w_t\,, \\
	\quad y_t&=C_tx_t+v_t \,,
	\end{aligned}
	\end{equation}
	where $x_t \in \mathbb{R}^n$ is the system state at time $t$ affected by process noise $w_t \sim \mathcal{D}_w$ with $x_0=\mu_0+\delta_0$, $\delta_0 \sim \mathcal{D}_{\delta_0}$,  $y_t\in \mathbb{R}^p$ is the  observed output at time $t$ affected by  measurement noise	$v_t\sim \mathcal{D}_v$, 
	and  $u_t \in \mathbb{R}^m$ is the control input at time $t$. We assume that the distributions $\mathcal{D}_w,\mathcal{D}_{\delta_0}$ $\mathcal{D}_v$ are bounded, have zero mean and variances  $\Sigma_w,\Sigma_{\delta_0},\Sigma_v \succ 0$ respectively.  
 
 We consider the following optimal control problem in a finite horizon of length $N$ 
 \begin{equation}
 \begin{aligned} \label{eq:optimal-control-discrete-time}
     \min_{u_0, \ldots, u_{N-1}} \quad &\mathbb{E}_{{w}_t,{v}_t} \left[\sum_{t=0}^{N-1} \left(y_t^\mathsf{T}M_ty_t +  u_t^\mathsf{T}R_tu_t\right)+y_N^\mathsf{T}M_N y_N \right]\, \\
     \text{subject to}\quad & \cref{eq:sys_disc},
 \end{aligned}
\end{equation}
where  $M_t \succeq 0$ and $R_t \succ 0, t = 0, \ldots, N$ are performance weight matrices.
In \cref{eq:optimal-control-discrete-time}, at each time $t$, the input $u_t$ can use the information available in $(y_0, y_1, \ldots, y_t)$ with an additional information constraint that will be imposed below. In particular, we consider linear feedback policies of the form
\begin{equation} \label{eq:distributed-policies}
u_t = K_{t,0}y_0 + K_{t,1}y_1, + \cdots + K_{t,t}y_t, \qquad t = 0, 1, \ldots, N-1,
\end{equation}
where $K_{t,i} \in \mathcal{S}_{t,i}, i = 1, \ldots, t$ and $ \mathcal{S}_{t,i} \subseteq \mathbb{R}^{m \times p}$ is a subspace constraint encoding communications. 

To derive a compact form of \cref{eq:optimal-control-discrete-time}, let us collect all signals in the finite horizon $N$ into a single vector as follows: 
 $$
 \begin{aligned}
 \mathbf{x}&=\begin{bmatrix}x_0 \\ \vdots \\ x_N \end{bmatrix},  \mathbf{y} =\begin{bmatrix}y_0 \\ \vdots \\ y_{N} \end{bmatrix}, \; \mathbf{u} =\begin{bmatrix}u_0 \\ \vdots \\ u_{N-1} \end{bmatrix},   \mathbf{w}=\begin{bmatrix}x_0 \\ w_0 \\ \vdots \\ w_{N-1}\end{bmatrix},
 \mathbf{v}=\begin{bmatrix}v_0 \\ \vdots \\ v_{N} \end{bmatrix}.
 \end{aligned}
 $$ 
We also define the matrices
\begin{equation*}
\mathbf{A}=\text{blkdg}(A_0,\ldots,A_{N}), \quad  \mathbf{B}\hspace{-0.1cm}=\hspace{-0.1cm}\begin{bmatrix}
 \text{blkdg}(B_0,\ldots,B_{N{-}1})\\
 0_{n \times mN}
 \end{bmatrix}, \quad \mathbf{C}=\text{blkdg}(C_0,\ldots,C_{N})
\,,
\end{equation*}
where $\text{blkdg}(\cdot)$ means a block-diagonal matrix, and defines a block down-shift matrix $$
\mathbf{Z}=\begin{bmatrix}
0_{1 \times N}&0\\
I_{N}&0_{N \times 1} 
\end{bmatrix}\otimes I_n\,.
$$
Then, we can write the evaluation of~\eqref{eq:sys_disc} over the horizon $N$ compactly as $\mathbf{x}= \mathbf{Z}\mathbf{A}\mathbf{x}+\mathbf{Z}\mathbf{B}\mathbf{u}+\mathbf{w}$,  $\mathbf{y}=\mathbf{Cx}+\mathbf{v},$ leading to 
$\mathbf{x}=\mathbf{P}_{11}\mathbf{w}+\mathbf{P}_{12}\mathbf{u}, \, \mathbf{y}=\mathbf{Cx}+\mathbf{v},$ 
where  $\mathbf{P}_{11}=(I-\mathbf{Z}\mathbf{A})^{-1}$ and $\mathbf{P}_{12}=(I-\mathbf{Z}\mathbf{A})^{-1}\mathbf{Z}\mathbf{B}$. 
The distributed policies in \cref{eq:distributed-policies} can be written compactly as
$ \mathbf{u}=\mathbf{Ky}, \, \mathbf{K} \in \mathcal{S}\,,$ 
where $\mathcal{S}$ is a subspace in $\mathbb{R}^{mN \times p(N+1)}$ which ensures causality of $\mathbf{K}$ by setting to $0$ those entries corresponding to future outputs, and also enforces the time-varying spatio-temporal information structure $\mathcal{S}_{t,i}$ for distributed policies. 
 
 Then, the problem \cref{eq:optimal-control-discrete-time} with a distributed linear policy \cref{eq:distributed-policies} can be written as 
 \begin{equation} \label{eq:LQfinite}
 \begin{aligned}
\min_{\mathbf{K} \in \mathbb{R}^{mN \times p(N+1)}} \quad & J(\mathbf{K})\,, \\
\text{subject to} \quad & \mathbf{K} \in \mathcal{S},  
 \end{aligned}
 \end{equation}
 where the cost $J(\mathbf{K})$ correspond to the cost in \cref{eq:optimal-control-discrete-time} for the policy $\mathbf{K}$. Note that \cref{eq:LQfinite} does not have a stability constraint since we focus on a finite horizon. Yet, due to the information constraint $\mathbf{K} \in \mathcal{S}$, this nonconvex distributed control problem \cref{eq:LQfinite} is in general hard to solve.

Thanks to the seminal work of \cite{rotkowitz2005characterization}, it is known that problem~\eqref{eq:LQfinite} can be equivalently transformed into a  convex program if and only if quadratic invariance (QI) holds, i.e., 
\begin{equation}
\label{eq:QI}
\mathbf{KCP}_{12}\mathbf{K} \in \mathcal{S},\quad \forall \mathbf{K} \in \mathcal{S}\,. 
\end{equation} 
One key observation from \cite[Lemma 5 \& Proposition 7]{furieri2020learning} is that the following smooth and invertible mapping $\mathcal{H}: \mathbb{R}^{mN \times p(N+1)} \to \mathbb{R}^{mN \times p(N+1)}$
\begin{equation*}
    \mathcal{H}(\mathbf{Q}) := (I + \mathbf{Q}\mathbf{C}\mathbf{P}_{12})^{-1}\mathbf{Q}, \qquad \mathcal{H}^{-1}(\mathbf{K}) = \mathbf{K}(I - \mathbf{C}\mathbf{P}_{12}\mathbf{K})^{-1}
\end{equation*}
can convexify \cref{eq:LQfinite}, and that the QI property \cref{eq:QI} ensures the information constraint on $\mathbf{K}$ can be transferred to $\mathbf{Q}$, i.e., $\mathbf{K} \in \mathcal{S} \Leftrightarrow \mathbf{Q} \in \mathcal{S}$. 
In particular, let us consider the epigraph of \cref{eq:LQfinite}
$$
\operatorname{epi}_\geq(J):= \{(\mathbf{K},\gamma) \in \mathcal{S} \times \mathbb{R}\mid J({\mathbf{K}})  \leq \gamma\},
$$
which is nonconvex. We can further show the following set is convex (see \cite{furieri2020learning}) 
$$
\mathcal{F}_{\mathrm{QI}} :=  \{(\gamma,\mathbf{Q}) \in \mathbb{R}\times\mathcal{S}\mid g(\mathbf{Q}):=J(\mathcal{H}({\mathbf{Q}}))  \leq \gamma\}.
$$
\begin{lemma} \label{lemma:distributed-control}
    Suppose the set $\mathcal{S}$ is QI with respect to $\mathbf{CP}_{12}$, i.e., \cref{eq:QI} holds. Then, the mapping $\Phi: (\mathbf{K},\gamma)\mapsto (\gamma,\mathcal{H}^{-1}(\mathbf{K}))$ is a $C^2$ (and in fact $C^\infty$)-diffeomorphism from $\operatorname{epi}_\geq(J)$ to $\mathcal{F}_{\mathrm{QI}}$.
\end{lemma}

The lemma is simply a restatement of the results in \cite[Lemma 5 \& Proposition 7]{furieri2020learning} using our \ECL{} terminologies. Thus, the distributed control \cref{eq:LQfinite} is equivalent to the following convex problem
    \begin{equation*}
        \begin{aligned}
            \min_{\mathbf{Q} \in \mathbb{R}^{mN \times p(N+1)}} \quad& g(\mathbf{Q}) \\
            \text{subject to}\quad & \mathbf{Q} \in \mathcal{S}, 
        \end{aligned}
    \end{equation*}
and any stationary point of \cref{eq:LQfinite} in the subspace $\mathcal{S}$ is globally optimal. These two results of hidden convexity and global optimality for \cref{eq:LQfinite} are indeed covered by \Cref{fact:epigraph-lift} (a special case of \ECL{}), once \Cref{lemma:distributed-control} is established. We note that no lifting procedure is required in this distributed control problem \cref{eq:LQfinite}, which is similar to \cref{example:academic,example:LQR_ex}. 

{
\begin{remark}[Hidden convexity in distributed policies] In many control instances (either with static polices or dynamic polices), a change of variables in the form of $K = YX^{-1}$ is one key step to convexify the problems and get a convex reformulation in terms of the new variables $X$ and $Y$. In static state feedback cases, the variables $X$ and $Y$ often arise from Lyapunov theory (see \Cref{subsection:static-policies} and \Cref{remark:state-feedback-policies}). In dynamic output feedback cases, the variables $X$ and $Y$ can be related to certain closed-loop responses in the frequency domain (such as Youla parameterization \cite{youla1976modern}, system-level synthesis \cite{wang2019system}, and input-output parameterization \cite{furieri2019input,zheng2021equivalence,zheng2022system}). In the case of distributed control, this seemingly innocent constraint $K \in \mathcal{S}$ becomes nonconvex in $X$ and $Y$ as it requires $YX^{-1} \in \mathcal{S}$. Thus, the change of variables $K = YX^{-1}$ cannot lead to a convex reformulation for distributed control. The notion of QI \cite{rotkowitz2005characterization} can translate the constraint $K \in \mathcal{S}$ into convex constraints on the new variables $X$ and $Y$; see \cite{furieri2020sparsity,zheng2021equivalence} for related discussions on sparsity invariance. Thus, together with classical changes of variables, the notion of QI reveals the hidden convexity in distributed control.      
\hfill $\square$
\end{remark}
}

%% file: app_sec1.tex
\section{Computational Details} \label{appendix:computational-details}

\subsection{Details of \Cref{fig:non-convexity-illustration}} \label{appendix:details-of-figure-1}

For \Cref{fig:non-convexity-illustration}(a), we consider $A=0$ and $B=I_2$. To illustrate the non-convexity of the set
\[
\mathcal{K} = \left\{K\mid A+BK\text{ is stable}\right\},
\]
let us define $\mathcal{K}'$ as the intersection of $\mathcal{K}$ and a linear subspace, i.e.,
\[
\mathcal{K}':=\mathcal{K}\cap\left\{K=\begin{bmatrix}
    -1 & k_1 \\ k_2 & -1
\end{bmatrix}: k_1,k_2\in\mathbb{R}\right\}.
\]
Using the Routh-Hurwitz criterion, it is straightforward to derive that 
\[
\mathcal{K}'=\left\{K=\begin{bmatrix}
    -1 & k_1 \\ k_2 & -1
\end{bmatrix}: k_1k_2<1\right\}
\]
and thus $\mathcal{K}'$ is a nonconvex set, which implies that $\mathcal{K}$ is also nonconvex. \Cref{fig:non-convexity-illustration}(a) plots the projection of $\mathcal{K}'$ onto $(k_1,k_2)$.

\Cref{fig:non-convexity-illustration}(b) illustrates the nonconvex landscape of an LQR example. We consider
\[
A = 0, \quad B=B_w=Q=R=I_2
\]
in the problem formulation \cref{eq:LQR-stochastic-noise}.
The feasible set has been demonstrated in \Cref{fig:non-convexity-illustration}(a). By fixing the diagonal entries of $K$ to be $-1$ (i.e., $K\in\mathcal{K}'$), the LQR cost around the optimal policy $K=-I_2$ is shown in \Cref{fig:non-convexity-illustration}(b), where the red dot corresponds to the optimal policy $K^\star$.

\Cref{fig:non-convexity-illustration}(c) illustrates the nonconvex and nonsmooth landscape of an $\mathcal{H}_\infty$ instance. We consider
\[A=B=C=Q=R=W=V=1\] in the problem formulation \cref{eq:Hinf_policy_optimization}. We consider the set $\mathcal{C}_1$ of full-order dynamic output feedback policies defined in \cref{eq:internallystabilizing-a}. Using the Routh-Hurwitz criterion, $\mathcal{C}_1$ is given by
\[
\mathcal{C}_1=\left\{\mK=\begin{bmatrix}
    D_\mK & C_\mK \\ B_\mK & A_\mK
\end{bmatrix}: A_\mK+D_\mK<-1, B_\mK C_\mK<A_\mK+A_\mK D_\mK\right\}
\]
By fixing $A_\mK=-1$ and $D_\mK=-1-\sqrt{3}$, the $\mathcal{H}_\infty$ cost around the dynamic policy $\mK=\begin{bmatrix}
    -1-\sqrt{3} & 0 \\ 0 & -1 
\end{bmatrix}$ is shown in \Cref{fig:non-convexity-illustration}(c), where a set of nonsmooth points are highlighted by red lines.
\vspace{20mm}

\subsection{Details of \Cref{example:LQR_ex}} \label{appendix:examples-of-ECL}

\Cref{example:LQR_ex} is a particular application of \Cref{remark:eliminating-lifting-variables}. Using the problem formulation \cref{eq:LQR-stochastic-noise}, this is an LQR example with matrices 
    $$A = \begin{bmatrix}
        -2 & 0 \\ 0 & 1
    \end{bmatrix},\quad B = \begin{bmatrix}
        0 \\ 1
    \end{bmatrix}, \quad B_w = 2I_2, \quad Q = I_2, \quad R = 1.
    $$ 
Let $K=\begin{bmatrix}
    k_1 & k_2
\end{bmatrix}\in\mathbb{R}^{1\times2}$. The set of stabilizing policies is given by 
\[
\mathcal{K}:=\{K
: A+BK \text{ is Hurwitz}\}=\{K
:k_1\in \mathbb{R}, \ k_2 < -1\}.
\]
For $K\in\mathcal{K}$,
the LQR cost function can be evaluated by
\begin{equation}\label{eq:LQR-cost-K}
    J_{\mathtt{LQR}}(k_1,k_2) =\operatorname{tr}\left((Q+K^\tr RK)X\right)=\operatorname{tr}\left((I+K^\tr K)X\right),
\end{equation}
where $X$ is the unique positive definite solution to the Lyapunov equation below, and can be explicitly expressed in terms of $k_1$ and $k_2$ as follows 
\begin{equation}\label{eq:LQR-Lyapunov-sol-X(K)}
    (A+BK)X+X(A+BK)^\tr+B_wB_w^\tr=0,\quad X = \frac{1}{k_2^2 - 1} \begin{bmatrix}
    k_2^2-1 & -k_1(k_2+1) \\ -k_1(k_2+1) & k_1^2-2k_2+2
    \end{bmatrix}.
\end{equation}
Substituting $X$ in \Cref{eq:LQR-Lyapunov-sol-X(K)} into \Cref{eq:LQR-cost-K} gives
\[
J_{\mathtt{LQR}}(k_1,k_2) = \frac{1-2k_2+3k_2^2 - 2k_2^3 - 2k_1^2k_2}{k_2^2 - 1}, \qquad \forall k_1\in \mathbb{R},\  k_2<-1. 
\]
Let $Y=\begin{bmatrix}y_1 & y_2 \end{bmatrix}$. Now define the change of variables $Y=g(K):=KX,$ i.e.,
\begin{equation}\label{eq:LQR-change-of-variables}
    \begin{aligned}
    y_1=\frac{k_1}{1-k_2}, \qquad 
   y_2 = \frac{2k_2-2k_2^2-k_1^2}{k_2^2-1} , \qquad \forall k_1\in \mathbb{R},\  k_2 < - 1, 
    \end{aligned}
\end{equation}
and note that $g$ is invertible with its inverse $K=g^{-1}(Y)=YX^{-1}$. We then define
\begin{equation} \label{eq:LQR-cost-Y}
    h(y_1,y_2):=J_{\mathtt{LQR}}\left(g^{-1}(y_1,y_2)\right)=\operatorname{tr}\left(QX + X^{-1}Y^\tr R Y \right)=\operatorname{tr}\left(X + X^{-1}Y^\tr  Y\right),
\end{equation}
where $X$ is now viewed as the unique positive definite solution to the Lyapunov equation below, and can be explicitly solved in terms of $y_1$ and $y_2$ as follows
\begin{equation} \label{eq:LQR-Lyapunov-sol-X(Y)}
AX+XA^\tr+(B_wB_w^\tr+BY+Y^\tr B^\tr)=0, \quad
    X=\begin{bmatrix}
    1 & y_1 \\
    y_1 & -y_2-2
\end{bmatrix}\succ0.
\end{equation}
Substituting $X$ in \Cref{eq:LQR-Lyapunov-sol-X(Y)} into \Cref{eq:LQR-cost-Y} gives
\[
h(y_1,y_2) = -y_2-1+\begin{bmatrix}
    y_1 & y_2
\end{bmatrix}\begin{bmatrix}
    1 & y_1 \\
    y_1 & -y_2-2
\end{bmatrix}^{-1}\begin{bmatrix}
    y_1 \\ y_2
\end{bmatrix}, \quad {\forall \begin{bmatrix}
    1 & y_1 \\
    y_1 & -y_2-2
\end{bmatrix}\succ0}. 
\]
Note that the inverse mapping $g^{-1}(Y)=YX^{-1}$ is given by
\begin{equation*} 
    \begin{aligned}
    k_1=\frac{2y_1(y_2+1)}{y_1^2+y_2+2}, \qquad 
   k_2=\frac{y_1^2-y_2}{y_1^2+y_2+2},
    \end{aligned}
\end{equation*}
and one can verify both $g \circ g^{-1}$ and $g^{-1} \circ g$ are the identity functions. 
We next show that the epigraph of $h$, $\operatorname{epi}_\geq(h) := \{(y_1,y_2,\gamma) \in \mathbb{R}^{3} \mid \gamma \geq h(y_1,y_2)\}$, is a convex set.
From the Schur complement, we have
\begin{equation*}
    \gamma \geq h(y_1,y_2)=\operatorname{tr}\left(X + X^{-1}Y^\tr  Y\right) \Leftrightarrow \begin{bmatrix}
        \gamma-\operatorname{tr}(X) & Y \\ Y^\tr & X
    \end{bmatrix}\succeq0
\end{equation*}
and thus
\[
\operatorname{epi}_\geq(h)=\left\{(y,\gamma) \in \mathbb{R}^3  \left|   \begin{bmatrix}
        \gamma +y_2 +1 & y^\tr \\
        y & \mathrm{aff}(y)
    \end{bmatrix} \succeq 0, \ 
    \mathrm{aff}(y) \!:=\! \begin{bmatrix}
         1 &  y_1 \\
y_1& -y_2 - 2
    \end{bmatrix} \!\succ\! 0\right.\right\}.
\]

%% file: app_sec2.tex
\section{Stationary Points in Optimal and Robust Control} \label{appendix:stationary-point}

In this section, we review (sub)-gradient computations of the benchmark control problems in \Cref{section:applications}. To further illustrate our \ECL{} framework, we also provide examples of stationary points in the benchmark control problems. 

\subsection{State Feedback Policy Optimization}

For convenience, we recall that the set of stabilizing static state feedback policies is 
$$
\mathcal{K}\coloneqq
\left\{K \in \mathbb{R}^{m \times n} \mid 
\max\nolimits_i \operatorname{Re}\lambda_i(A+BK)
<0\right\},
$$ 
and the closed-loop transfer function from $w$ to~$z$ is
\begin{equation} \label{eq:app-Tzw}
\mathbf{T}_{zw}(K,s) = \begin{bmatrix}
    Q^{1/2} \\ R^{1/2}K
\end{bmatrix} (sI - A - BK)^{-1}B_w. 
\end{equation}

\subsubsection{The LQR Problem \cref{eq:LQR-H2}}
It is well known that 
    for $K\in\mathcal{K}$, the LQR cost $J_\mathtt{LQR}(K)$ in \cref{eq:LQR-H2} can be evaluated by
    \begin{equation} \label{eq:LQR-cost}
    J_\mathtt{LQR}(K)=\operatorname{tr}\left((Q+K^\tr R K){X_K}\right)=\operatorname{tr}({P_K}W),
    \end{equation}
    where $X_K$ and $P_K$ are the unique positive semidefinite solutions to the Lyapunov equations 
    \begin{subequations} \label{eq:Lyapunov-LQR}
    \begin{align}
        (A+BK)X_K+X_K (A+BK)^\tr+W&=0, \label{eq:Lyapunov-LQR-a} \\
        (A+BK)^\tr P_K+P_K (A+BK)+Q+K^\tr R K&=0. \label{eq:Lyapunov-LQR-b}
    \end{align}
    \end{subequations}

The computation in \cref{eq:LQR-cost} can be seen from \cref{eq:Lyapunov-equations-H2norm} in \cref{lemma:H2norm}. Note that we can also see the equality of $\operatorname{tr}\left((Q+K^\tr R K){X_K}\right)=\operatorname{tr}({P_K}W)$ from the manipulation of the Lyapunov equations \cref{eq:Lyapunov-LQR} and the trace operator, as shown below  
\begin{align*}
    \operatorname{tr}\left((Q+K^\tr R K)X_K\right)&=-\operatorname{tr}\left(((A+BK)^\tr P_K+P_K (A+BK))X_K\right)\\
    &=-\operatorname{tr}\left(P_K(X_K(A+BK)^\tr+(A+BK)X_K)\right)=\operatorname{tr}(P_KW), 
\end{align*}
where the first equality uses the fact $-Q-K^\tr R K = (A+BK)^\tr P_K+P_K (A+BK)$ from \cref{eq:Lyapunov-LQR-b}, the second equality is from the trace property, and the last equality applies $-W = (A+BK)X_K+X_K (A+BK)^\tr$ from \cref{eq:Lyapunov-LQR-a}. 

From \cref{eq:LQR-cost,eq:Lyapunov-LQR}, it is clear that $J_\mathtt{LQR}(K)$ is a rational function (i.e., a ratio of two polynomials) with respect to the elements of $K$, thus it is infinitely differentiable over $\mathcal{K}$. Indeed, we have a closed-form formula to compute its gradient at any point $K \in \mathcal{K}$.  

\begin{lemma}[{\cite[Section IV]{levine1970determination}}] \label{lemma:LQR-cost-gradient}
    The LQR cost in \cref{eq:LQR-H2} is infinitely differentiable over $\mathcal{K}$. Its gradient is given by
    \[
    \nabla{J_\mathtt{LQR}(K)}=2(RK+B^\tr P_K)X_K, \quad \forall K \in \mathcal{K}. 
    \]
where $X_K$ and $P_K$ are the unique positive semidefinite solutions to Lyapunov equations \cref{eq:Lyapunov-LQR}. 
\end{lemma}

If $(A+BK,B_w)$ is controllable, then $X_K$ from \cref{eq:Lyapunov-LQR-a} is always positive definite. Then, the only case for the stationary condition $\nabla{J_\mathrm{LQR}(K)} =0 $ is $RK+B^\tr P_K = 0$, which implies 
$$
K = -R^{-1}B^\tr P_K. 
$$
Plugging this into \cref{eq:Lyapunov-LQR-b} leads to the famous Riccati equation $A^\tr P_K+P_K A+Q-P_K BR^{-1}B^\tr P_K = 0$. We summarize this fact into the following lemma. 

\begin{lemma} \label{lemma:LQR-stationary-optimality}
    Suppose $(A+BK,B_w)$ is controllable for all $K\in\mathcal{K}$ (which is implied by $B_w$ being full row rank), then the LQR cost \cref{eq:LQR-H2} has a unique stationary point, which is in the form of 
    $
    K^\star=-R^{-1}B^\tr {P^\star} \in \mathcal{K},
    $ 
    where ${P^\star}$ is the unique stabilizing solution to the following Riccati equation 
    \begin{equation} \label{eq:ARE-LQR}
    A^\tr {P^\star}+{P^\star}A+Q-{P^\star} BR^{-1}B^\tr {P^\star}=0.
    \end{equation}
\end{lemma}

It is well-known that the Riccati equation \cref{eq:ARE-LQR} has a unique positive semidefinite solution $P^\star$ when $(A,B)$ is stabilizable and $(Q^{1/2}, A)$ is detectable \cite[Corollary 13.8]{zhou1996robust}. This solution is also stabilizing (i.e., $A - BR^{-1}B^\tr {P^\star}$ is stable, implying $K^\star=-R^{-1}B^\tr {P^\star} \in \mathcal{K}$). From the classical result in \cite[Theorem 14.2]{zhou1996robust}, the unique stationary point in \cref{lemma:LQR-stationary-optimality} is indeed globally optimal, which is consistent with our \ECL{}  guarantee in \Cref{theorem:ECL-guarantee}.

\begin{example} \label{example:LQR-stationary-point-K-1by2}
    We here provide a simple example to illustrate \Cref{lemma:LQR-cost-gradient,lemma:LQR-stationary-optimality}. Consider LQR with problem data from \cref{example:LQR_ex}, i.e., 
    \[
    A=\begin{bmatrix}
        -2 & 0 \\ 0 & 1
    \end{bmatrix},\quad B=\begin{bmatrix}
        0 \\ 1
    \end{bmatrix},\quad  B_w=2I_2, \quad Q=I_2,\quad R=1.
    \]
    Consider for instance $K=\begin{bmatrix}
        1 & -2
    \end{bmatrix}\in\mathcal{K}$, i.e., $A+BK$ is Hurwitz. It is straightforward to solve the Lyapunov equations \cref{eq:Lyapunov-LQR} and get 
    $$
     X_{K}=\frac{1}{3}\begin{bmatrix}
        3 & 1 \\ 1 & 7
    \end{bmatrix}, \qquad P_{K}=\frac{1}{12}\begin{bmatrix}7 & 2\\2 & 30
    \end{bmatrix},
    $$
    which are both positive definite.  
    Then, its cost and policy gradient can be evaluated via \cref{eq:LQR-cost} and \Cref{lemma:LQR-cost-gradient} as follows 
    \begin{align*}
       &J_\mathrm{LQR}(K)=\operatorname{tr}\left((Q+K^\tr R K){X_K}\right)=\operatorname{tr}({P_K}W)=12\frac{1}{3},\\
    &\nabla{J_\mathrm{LQR}(K)}=2(RK+B^\tr P_K)X_{K}=\frac{1}{9}\begin{bmatrix}24 & 28
    \end{bmatrix}. 
    \end{align*}
    
    From \cref{lemma:LQR-stationary-optimality}, by solving the Riccati equation \cref{eq:ARE-LQR}, we obtain the unique stabilizing solution $P^\star$
    and the unique stationary point as follows
    \[
    P^\star=\begin{bmatrix}
        0.25 & 0 \\ 0 & 1\!+\!\sqrt{2}
    \end{bmatrix}, \qquad K^\star=-R^{-1}B^\tr P^{\star}=\begin{bmatrix}
        0 & -1\!-\!\sqrt{2}
    \end{bmatrix}.
    \]
    {It is easy to verify that $\nabla J_\mathrm{LQR}(K^\star)=0$ and $J_\mathrm{LQR}(K^\star)=5+4\sqrt{2}$.} This controller is the globally optimal policy, which is consistent with our \ECL{} guarantee in \Cref{theorem:ECL-guarantee}.
\end{example}

\subsubsection{The State Feedback $\mathcal{H}_\infty$ Control Problem \cref{eq:Hinf}}
The $\mathcal{H}_\infty$ cost $J_{\infty}$ in \cref{eq:Hinf} is in general nonconvex and also nonsmooth. A nice property is that $J_{\infty}$ is subdifferential regular \cite[Proposition 3.1]{apkarian2006nonsmooth}, \cite[Lemma 5.1]{zheng2023benign}. We here briefly discuss the computation of the Clarke subdifferential of $J_{\infty}(K)$. 

Let $j\mathbb{R}$ denote the imaginary axis in $\mathbb{C}$. Fix $K \in \mathcal{K}$, and define the set of frequencies achieving the $\mathcal{H}_\infty$ norm as  
\[
\mathcal{Z} := \{s\in j\mathbb{\mathbb{R}}\cup \{\infty\}
\mid \sigma_{\max}(\mathbf{T}_{zw}(K,s)) = J_{\infty}(K)\},
\]
where $\mathbf{T}_{zw}(K,s)$ is the closed-loop transfer function from $w$ to~$z$ in \Cref{eq:app-Tzw}. 
For each $s\in\mathcal{Z}$, let $Q_s$ be a complex matrix whose columns form an orthonormal basis of the eigenspace of $\mathbf{T}_{zw}(K,s)\mathbf{T}_{zw}(K,s)^\her$ associated with its maximal eigenvalue $J_{\infty}^2(K)$, where $(\cdot)^\her$ denotes the Hermitian transpose. 

The following lemma characterizes the subdifferential $\partial J_{\infty}(K)$, which is a special case in \cite[Lemma 5.2]{zheng2023benign}. 

\begin{lemma}[{\cite[Lemma 5.2]{zheng2023benign}}]\label{lemma:subdifferential-Hinf-SF}
Fix $K \in \mathcal{K}$. A matrix $\Phi \in \mathbb{R}^{m\times n}$ is a member of $\partial J_{\infty}(K)$ if and only if there exist finitely many $s_1,\ldots,s_q\in \mathcal{Z}$ and positive semidefinite Hermitian matrices $Y_1,\ldots,Y_q$ of compatible dimensions with $\sum_{\kappa=1}^q\operatorname{tr}(Y_\kappa)=1$ such that
\begin{align*}
\Phi = \frac{1}{J_{\infty}(K)}
\sum_{\kappa=1}^q
\operatorname{Re}\bigg\{\!
& 
(s_\kappa I-A-BK)^{-1}B_w\cdot
\mathbf{T}_{zw}(K,s_\kappa)^\her Q_{s_\kappa} Y_\kappa Q_{s_\kappa}^\her \\
& \qquad \qquad \cdot
\left(
\begin{bmatrix}
    Q^{1/2} \\ R^{1/2}K
\end{bmatrix}(s_\kappa I-A-BK)^{-1}B_w+\begin{bmatrix}
    0 \\ R^{1/2}
\end{bmatrix}
\right)
\!\bigg\}^{\!\tr}.
\end{align*}
\end{lemma}
This result is stronger than previous subdifferential calculations in \cite{apkarian2006nonsmooth}, as it presents sufficient and necessary conditions for the subdifferential even when the $\mathcal{H}_\infty$ norm is attained at infinitely many frequencies; see a technical comparison in \cite[Remark B.1]{zheng2023benign}. 
As shown in \cref{lemma:subdifferential-Hinf-SF}, the analytical computation of $\partial J_{\infty}(K)$ for the $\mathcal{H}_\infty$ cost in \cref{eq:Hinf} is much more complicated than the smooth LQR case in \Cref{lemma:LQR-cost-gradient}. 

We here present arguably the simplest $\mathcal{H}_\infty$ instance in \Cref{example:Hinf}, which involves an open-loop stable system with one scalar state. As we have seen, this $\mathcal{H}_\infty$ example is already nonconvex.  

\begin{example}[Nonconvex and smooth $\mathcal{H}_\infty$ instance] \label{example:Hinf-SF-stationary-point-K-1by1}
    Consider the $\mathcal{H}_\infty$ state feedback policy optimization \cref{eq:Hinf} with problem data
    \[
    A=-1,\quad B=1,\quad B_w=1,\quad Q=0.1,\quad R=1.
    \]
    It turns out that this simple example has a nonconvex differentiable cost function 
    \begin{align}\label{eq:Hinf-cost-1by1}
        J_\infty(K)=\sup_{\omega \in \mathbb{R}} \sqrt{\frac{0.1+K^2}{(1-K)^2 + \omega^2}} =\frac{\sqrt{0.1+K^2}}{1-K},\quad \forall K<1,
    \end{align}
    whose gradient is given by 
    \begin{align}\label{eq:Hinf-grad-1by1}
        \nabla J_\infty(K)=\frac{0.1+K}{(1-K)^2\sqrt{0.1+K^2}},\quad \forall K<1.
    \end{align}
    The optimal policy occurs at the stationary point $K^\star=-0.1$ with $J_\infty(K^\star)=0.3015$. 
    
    Note that for this differentiable case, the subgradient computation using \cref{lemma:subdifferential-Hinf-SF} yields the same value as the gradient.  
       Indeed from \cref{eq:Hinf-cost-1by1}, one can see that for all $K<1$, we have 
    \[
    \mathcal{Z}=\{0\},\quad \mathbf{T}_{zw}(K,0)^\tr=\frac{1}{1-K}\begin{bmatrix}
        {\sqrt{0.1}} & {K}
    \end{bmatrix},\quad Q_0^\tr=\frac{1}{\sqrt{0.1+K^2}}\begin{bmatrix}
        {\sqrt{0.1}}&  {K}
    \end{bmatrix},
    \]
    where $Q_0$ is an orthonormal eigenvector of $\mathbf{T}_{zw}(K,0)\mathbf{T}_{zw}(K,0)^\tr$ associated with the maximal eigenvalue $J_\infty^2(K)=(0.1\!+\!K^2)/(1\!-\!K)^2$.
    The only choice for $Y_1$ is $Y_1=1$ to satisfy $\mathrm{tr}(Y_1) = 1$. Therefore, the subgradient at $K$ only contains one element given by
    \begin{align*}
         \Phi &= \frac{1}{J_\infty(K)}
    \left(
    (-A-BK)^{-1}B_w\cdot
    \mathbf{T}_{zw}(K,0)^\tr Q_{0}Y_1Q_{0}^\tr \cdot
    \left(
    \begin{bmatrix}
        Q^{1/2} \\ R^{1/2}K
    \end{bmatrix}(-A-BK)^{-1}B+\begin{bmatrix}
        0 \\ R^{1/2}
    \end{bmatrix}
    \right)
    \right)^{\!\tr}\\
    &=\frac{1-K}{\sqrt{0.1+K^2}}\left(\frac{1}{(1-K)^3(0.1+K^2)}\begin{bmatrix}
        {\sqrt{0.1}} & {K}
    \end{bmatrix}\begin{bmatrix}
        {0.1} & {\sqrt{0.1}K} \\ {\sqrt{0.1}K} & K^2
    \end{bmatrix}\begin{bmatrix}
        {\sqrt{0.1}} \\ {1}
    \end{bmatrix}\right)\\
    &=\frac{0.1+K}{(1-K)^2\sqrt{0.1+K^2}}=\nabla J_\infty(K).
    \end{align*}
The fact that $\Phi=\nabla J_\infty$ also confirms the differentiability of $J_\infty$ in this instance.
    The nonconvex landscape is illustrated in \cref{fig:Hinf-SF-stationary-point-K-1by1}. \hfill $\qed$
\end{example}

\begin{figure}[t]
    \centering
    \setlength{\abovecaptionskip}{4pt}
    \begin{subfigure}{0.35\textwidth}
        \includegraphics[width=0.8\textwidth]{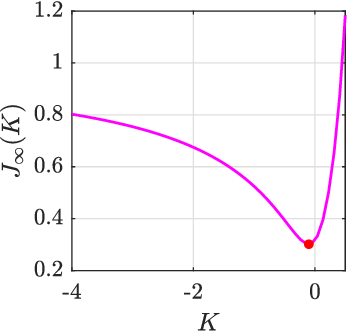}
        \caption{}     \label{fig:Hinf-SF-stationary-point-K-1by1}
    \end{subfigure}
    \hspace{10mm}
    \begin{subfigure}{0.35\textwidth}
        \includegraphics[width=1
\textwidth]{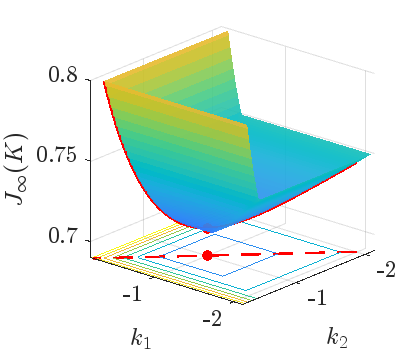}
        \caption{}     \label{fig:Hinf-SF-stationary-point-K-1by2}
    \end{subfigure}
    \caption{Optimization landscape of the state feedback $\mathcal{H}_\infty$ instances: (a) nonconvex and smooth $J_\infty$ in \cref{example:Hinf-SF-stationary-point-K-1by1}; (b) nonconvex and nonsmooth $J_\infty$ in \cref{example:Hinf-SF-stationary-point-K-1by2}. 
    }
    \label{fig:simple-example-Hinf_state}
\end{figure}

\begin{example}[Nonconvex and nonsmooth $\mathcal{H}_\infty$ instance] \label{example:Hinf-SF-stationary-point-K-1by2}
    Consider the $\mathcal{H}_\infty$ state feedback policy optimization \cref{eq:Hinf} with problem data
    \[
    A=-I_2,\quad B=B_w=Q=R=I_2.
    \]
   The $\mathcal{H}_\infty$ problem is nonconvex since its feasible set $\mathcal{K}$ is already nonconvex. Indeed, we can find $K_1=\begin{bmatrix}
        0 & 0.1 \\ 9.9 & 0
    \end{bmatrix}\!\in\!\mathcal{K}$ and $K_2=\begin{bmatrix}
        0 & 9.9 \\ 0.1 & 0
    \end{bmatrix}\!\in\!\mathcal{K}$ but $(K_1\!+\!K_2)/2\notin\mathcal{K}$.

   By solving the equivalent convex problem in \cref{eq:state-feedback-Hinf-LMI}, the optimal policy is given by $K^\star=-I_2$, which is a Clarke stationary point, i.e.,  $0\in\partial J_\infty(K^\star)$. In addition, we can check that there exist multiple subgradients for all $K_\ell=\begin{bmatrix}
        k & 0 \\ 0 & k
    \end{bmatrix}$ where $k<0$, justifying the nonsmoothness of $J_\infty$. Indeed, the cost at $K_\ell$ can be evaluated by
    \[
    J_\infty(K_\ell)=\sup_{\omega\in\mathbb{R}} \lambda^{1/2}_\mathrm{max}(\mathbf{T}_{zw}(K_\ell,i\omega)^\her \mathbf{T}_{zw}(K_\ell,i\omega))=\frac{\sqrt{k^2+1}}{1-k}, \quad \forall K_\ell=\begin{bmatrix}
        k & 0 \\ 0 & k
    \end{bmatrix},\ k<0,
    \]
    where the supremum occurs at $\omega=0$.
    Since the maximal eigenvalue $J_{\infty}^2(K_\ell)$ of $\mathbf{T}_{zw}(K_\ell,0)\mathbf{T}_{zw}(K_\ell,0)^\her$ has algebraic multiplicities equal to $2$, the corresponding eigenspace has dimension equal to 2. Therefore, in \Cref{lemma:subdifferential-Hinf-SF}, one has the freedom to choose different $Y_1\in\mathbb{S}_{+}^2$ to find multiple subgradients at $K_\ell$ as long as $\operatorname{tr}(Y_1)=1$.

    By fixing the off-diagonal entries of $K$ to be $0$, i.e., $K=\begin{bmatrix}
        k_1 & 0 \\ 0 & k_2
    \end{bmatrix}$, we illustrate the nonsmooth $\mathcal{H}_\infty$ cost $J_\infty(K)$ around the optimal policy $K^\star$ in \Cref{fig:Hinf-SF-stationary-point-K-1by2}, where the red dot corresponds to $K^\star$ and the red dash line corresponds to $K_\ell$. It is clear that $J_\infty(K)$ is nonsmooth on the red dash line. 
    \hfill $\qed$
\end{example}

\subsection{Output Feedback Policy Optimization} 

We briefly review the gradient computation for LQG cost and the Clarke subdifferential characterization for the $\mathcal{H}_\infty$ cost. More detailed discussions on the structure of stationary points can be found in \cite{zheng2021analysis} and Part I of this paper \cite{zheng2023benign}.

\subsubsection{The LQG Problem \cref{eq:LQG_policy_optimization}}
It is known that the LQG cost function $J_{\LQG,q}(\mK)$ can be computed by solving a Lyapunov equation, which is summarized in the lemma below.

\begin{lemma}\label{lemma:LQG_cost_formulation1}
Fix $q\in\mathbb{N}$ such that $\mathcal{C}_{q,0}\neq\varnothing$. Given any $\mK\in\mathcal{C}_{q,0}$, we have
\begin{equation}\label{eq:LQG_cost_formulation1}
\begin{aligned}
J_{\LQG,q}(\mK)
={} &
\sqrt{\operatorname{tr}\!
\left(C_{\mathrm{cl}}(\mK)X_\mK C_{\mathrm{cl}}(\mK)^\tr \right)}
=
\sqrt{
\operatorname{tr}\!
\left(B_{\mathrm{cl}}(\mK)^\tr Y_\mK B_{\mathrm{cl}}(\mK)\right)},
\end{aligned}
\end{equation}
where $X_{\mK}$ and $Y_{\mK}$ are the unique positive semidefinite 
solutions to the following Lyapunov equations
\begin{subequations} \label{eq:Lyapunov-LQG}
\begin{align}
\begin{bmatrix} A &  BC_{\mK} \\ B_{\mK} C & A_{\mK} \end{bmatrix}X_{\mK} + X_{\mK}\begin{bmatrix} A &  BC_{\mK} \\ B_{\mK} C & A_{\mK} \end{bmatrix}^\tr +  \begin{bmatrix} W & 0 \\ 0 & B_{\mK}VB_{\mK}^\tr  \end{bmatrix}
& = 0, \label{eq:LyapunovX}
\\
\begin{bmatrix} A &  BC_{\mK} \\ B_{\mK} C & A_{\mK} \end{bmatrix}^\tr Y_{\mK} +  Y_{\mK}\begin{bmatrix} A &  BC_{\mK} \\ B_{\mK} C & A_{\mK} \end{bmatrix} +   \begin{bmatrix} Q & 0 \\ 0 & C_{\mK}^\tr R C_{\mK} \end{bmatrix}
& = 0. \label{eq:LyapunovY}
\end{align}
\end{subequations}
\end{lemma}

Recall that the LQG cost $J_{\LQG,q}(\mK)$ is a real analytical function over its domain ${\mathcal{C}}_{q,0}$ {\cite[Lemma 2.3]{zheng2021analysis}}, which implies that $J_{\LQG,q}(\mK)$ is infinitely differentiable. 
The following lemma gives the closed-form expression to compute the gradient of $J_{\LQG,q}(\mK)$.

\begin{lemma}[{\cite[Lemma 4.2]{zheng2021analysis}}] \label{lemma:gradient_LQG_Jn}
    {Fix $q \geq 1$ such that ${\mathcal{C}}_{q,0} \neq \varnothing$.}
    For every $\mK = \begin{bmatrix} 0 & C_{\mK} \\B_{\mK} & A_{\mK} \end{bmatrix} \in {\mathcal{C}}_{q,0}$, the gradient of $J_{\LQG,q}(\mK)$ is given by
    \begin{subequations} \label{eq:gradient_Jn}
        \begin{align}
         \frac{\partial J_{\LQG,q}(\mK)}{\partial A_{\mK}} &= \frac{1}{J_{\LQG,q}(\mK)}\left(Y_{12}^\tr X_{12} + Y_{22}X_{22}\right), \label{eq:partial_Ak}\\
        \frac{\partial J_{\LQG,q}(\mK)}{\partial B_{\mK}} &= \frac{1}{J_{\LQG,q}(\mK)}\left(Y_{22}B_{\mK}V + Y_{22}X_{12}^\tr C^\tr + Y_{12}^\tr X_{11} C^\tr\right), \label{eq:partial_Bk}\\
        \frac{\partial J_{\LQG,q}(\mK)}{\partial C_{\mK}} &= \frac{1}{J_{\LQG,q}(\mK)}\left(RC_{\mK}X_{22} + B^\tr Y_{11}X_{12} + B^\tr Y_{12} X_{22}\right), \label{eq:partial_Ck}
        \end{align}
    \end{subequations}
    where $X_{\mK}$ and $Y_{\mK}$, partitioned as 
    \begin{equation} \label{eq:LyapunovXY_block}
        X_{\mK} = \begin{bmatrix}
        X_{11} & X_{12} \\ X_{12}^\tr & X_{22}
        \end{bmatrix},  \qquad Y_{\mK} = \begin{bmatrix}
        Y_{11} & Y_{12} \\ Y_{12}^\tr & Y_{22}
        \end{bmatrix}
    \end{equation}
    are the unique positive semidefinite 
    solutions to~\cref{eq:LyapunovX} and~\cref{eq:LyapunovY}, respectively.
\end{lemma}

\subsubsection{The Output Feedback $\mathcal{H}_\infty$ Control Problem \cref{eq:Hinf_policy_optimization}}
The $\mathcal{H}_\infty$ cost $J_{\infty,n}(\mK)$ in \cref{eq:Hinf_policy_optimization} is in general nonconvex and also nonsmooth. A nice property is that $J_{\infty,n}$ is subdifferential regular \cite[Proposition 3.1]{apkarian2006nonsmooth}, \cite[Lemma 5.1]{zheng2023benign}. We here briefly discuss the computation of the Clarke subdifferential of $J_{\infty,n}(\mK)$ at any feasible point $\mK \in \mathcal{C}_n$. 

Let $j\mathbb{R}$ denote the imaginary axis in $\mathbb{C}$. Fix $\mK\in\mathcal{C}_q$, and define
\[
\mathcal{Z} = \{s\in j\mathbb{\mathbb{R}}\cup \{\infty\}
\mid \sigma_{\max}(\mathbf{T}_{zd}(\mK,s)) = J_{\infty,n}(\mK)\}.
\]
For each $s\in\mathcal{Z}$, let $Q_s$ be a complex matrix whose columns form an orthonormal basis of the eigenspace of $\mathbf{T}_{zd}(\mK,s)\mathbf{T}_{zd}(\mK,s)^\her$ associated with its maximal eigenvalue $J_{\infty,n}^2(\mK)$. The following lemma characterizes the subdifferential $\partial J_{\infty,n}(\mK)$, whose proof is technically involved and presented in \cite[Lemma 5.2]{zheng2023benign}.

\begin{lemma}\label{lemma:subdifferential-Hinf}
A matrix $\Phi \in \mathbb{R}^{(m+n)\times(p+n)}$ is a member of $\partial J_{\infty,n}(\mK)$ if and only if there exist finitely many $s_1,\ldots,s_K\in \mathcal{Z}$ and associated positive semidefinite Hermitian matrices $Y_1,\ldots,Y_K$ with $\sum_{\kappa=1}^K\operatorname{tr}Y_\kappa=1$ such that
\begin{align*}
\Phi = \frac{1}{J_{\infty,n}(\mK)}
\sum_{\kappa=1}^K
\operatorname{Re}\bigg\{\!
& \left(
\begin{bmatrix}
0 & V^{1/2} \\ 0 & 0
\end{bmatrix}
+
\begin{bmatrix}
C & 0 \\
0 & I
\end{bmatrix}
(s_\kappa I-A_{\mathrm{cl}}(\mK))^{-1}B_{\mathrm{cl}}(\mK)
\right)
\mathbf{T}_{zd}(\mK,s_\kappa)^\her Q_{s_\kappa} Y_\kappa Q_{s_\kappa}^\her \\
&\quad \cdot
\left(
\begin{bmatrix}
0 & 0 \\ R^{1/2} & 0
\end{bmatrix}
+C_{\mathrm{cl}}(\mK)(s_\kappa I-A_{\mathrm{cl}}(\mK))^{-1}
\begin{bmatrix}
B & 0 \\ 0 & I
\end{bmatrix}
\right)
\!\bigg\}^{\!\tr}.
\end{align*}
\end{lemma}

We refer the interested reader to Part I of this paper \cite[Section 5]{zheng2023benign} for analytical examples of Clarke stationary points in $\mathcal{H}_\infty$ control. 

%% file: app_sec3.tex
\section{Proofs for State Feedback Policy Optimization} \label{appendix:state-feedback}
In this section, we provide some extra discussions and missing proofs for \ECL{} applied to the policy optimization problems for state feedback control in \cref{subsection:static-policies}. 

\subsection{Convexity of $\mathcal{F}_{\mathrm{LQR}}$} \label{subsection:convexity-LQR}

We provide some details to show that the following matrix fractional function arising in LQR
$$
f(X,Y) = \operatorname{tr} \!\left(QX + X^{-1}Y^\tr RY\right), \quad \mathrm{dom}(f):=\{(X, Y) \in \mathbb{S}^n_{++} \times \mathbb{R}^{m \times n}\}
$$
is convex. 
This is a known fact; see e.g., \cite[Lemma 4.4]{khargonekar1991mixed} and \cite[Section IV.A]{mohammadi2021convergence}.  

First, $\mathrm{dom}(f)$ is evidently a convex set. Second, note that
$$ 
f(X,Y) = \operatorname{tr}(QX) + \operatorname{tr}\!\left(
(Y^\tr R^{1/2})^\tr X^{-1} Y^\tr R^{1/2}\right)
= \operatorname{tr}(QX)
+\sum_{i=1}^n g\big(X,(Y^\tr R^{1/2})_i\big),
$$ 
where $(Y^\tr R^{1/2})_i$ denotes the $i$-th column of $Y^\tr R^{1/2}$, and
\[
g(X,w) = w^\tr X^{-1}w,
\qquad X \in \mathbb{S}^n_{++}, w \in \mathbb{R}^{n}.
\]
Since $X\mapsto \operatorname{tr}(QX)$ is linear and $Y\mapsto (Y^\tr R^{1/2})_i$ is a linear transformation, it suffices to show that $g(X,w)$ is jointly convex in $(X,w)$, or that the epigraph $\operatorname{epi}_{\geq}(g) = \setv*{(X,w,\gamma)}{X\succ 0,\gamma\geq g(X,w)}$ is convex. By using Schur complement, we have
\[
\begin{aligned}
\operatorname{epi}_{\geq}(g)
& =\setv*{(X,w,\gamma)}{X\succ 0, \gamma\geq w^\tr X^{-1} w} \\
& =\setv*{(X,w,\gamma)}
{
X\succ 0,
\begin{bmatrix}
\gamma & w^\tr \\
w & X
\end{bmatrix}\succeq 0}
\end{aligned}
\]
which is a convex set. The proof is now complete.

\subsection{Proofs for State Feedback $\mathcal{H}_\infty$ Control} 
\label{subsection:proofs-state-feedback-Hinf}

\subsubsection{A Non-coercive Example of State Feedback $\mathcal{H}_\infty$ Policy Optimization} \label{subsection:non-coercivity}

\Cref{example:Hinf} has already shown the $\mathcal{H}_\infty$ cost function is not coercive, but the globally optimal policy is achieved. In \Cref{example:Hinf}, the system is open-loop stable. 
Here, we show another simple instance where the  $\mathcal{H}_\infty$ cost function is not coercive, and the optimal $\mathcal{H}_\infty$ policy is not attained. 

Consider a state feedback $\mathcal{H}_\infty$ instance \cref{eq:Hinf-state-feedback} with the following problem data (which is open-loop unstable):
$$
A = 1, \; B = -1, \; B_w = 1,  \; Q = 1,\; R = 1. 
$$
This $\mathcal{H}_\infty$ instance satisfies all the assumptions in \Cref{fact:Hinf-non-coercive}, i.e. $(A,B)$ is controllable, $B_w$ has full row rank, and $Q\succ 0, R\succ 0$. Let us consider a linear state feedback policy $u(t) = k x(t)$ with stabilizing $k >1$. According to \cref{eq:transfer-function-Tzw} and after some simple calculations, we can compute the $\mathcal{H}_\infty$ cost function analytically below 
$$
\begin{aligned}
J_\infty(k) = \left\|\begin{bmatrix}
    1 \\ k
\end{bmatrix}(s - (1-k))^{-1}\right\|_\infty &= \sup_{\omega \in \mathbb{R}} \ \lambda_{\max}^{1/2}\left(\left(-i\omega - (1-k)\right)^{-1}\begin{bmatrix}
    1 & k
\end{bmatrix} \begin{bmatrix}
    1 \\ k
\end{bmatrix}\left(i\omega - (1-k)\right)^{-1}\right) \\
&=     \sup_{\omega \in \mathbb{R}} \sqrt{\frac{1+k^2}{(k-1)^2 + \omega^2}}  =  \frac{\sqrt{1+k^2}}{k-1},\qquad  \forall k>1.
\end{aligned}
$$
It is clear that we have $\lim_{k \to \infty } J_\infty(k) = 1$, and this function is not coercive. The infimum of $J_\infty(k)$ is not attained (we have $\inf_{k>1} J_\infty(k) = 1$, requiring $k \to \infty$). In this example, there exists no finite Clarke stationary point. 

Moreover, for this instance, we will see that the corresponding LMI in \cref{eq:state-feedback-Hinf-LMI} is not solvable. Indeed, the problem $\inf_{(\gamma, y, x) \in \mathcal{F}_\infty} \;\gamma$  reads as 
\begin{equation} \label{eq:LMI-hinf-nont-attained}
\begin{aligned}
\inf_{\gamma, y, x} & \; \gamma \\
\text{subject to}& \; x > 0, \ \begin{bmatrix}
2(x-y)& 1 & x &-y \\
1 & -\gamma  & 0&0 \\
x & 0 &  -\gamma &0\\
-y& 0 & 0 & -\gamma
\end{bmatrix}\preceq 0,   
\end{aligned}
\end{equation}
which, via the Shur complement, is equivalent to (note that it is easy to see that $\gamma >0$)
$$ 
\begin{aligned}
\inf_{\gamma, y, x} & \; \gamma \\
\text{subject to}& \; x > 0, \ 2(x-y) + \gamma^{-1} (1 + x^2 + y^2) \leq 0.  
\end{aligned}
$$
It is not difficult to argue that achieving the infimum requires $(\gamma,y,x)=(1,1,0)$. However, this point $(1, 1, 0)$ is on the boundary but outside $\mathcal{F}_\infty$, and we have $k = yx^{-1} = \infty$. The feasible region of $\mathcal{F}_\infty$ in \cref{eq:LMI-hinf-nont-attained} is shown in \Cref{fig:feasible-hinf-example}.

\begin{figure}[t]
    \centering
    \includegraphics[width=0.25 \textwidth]{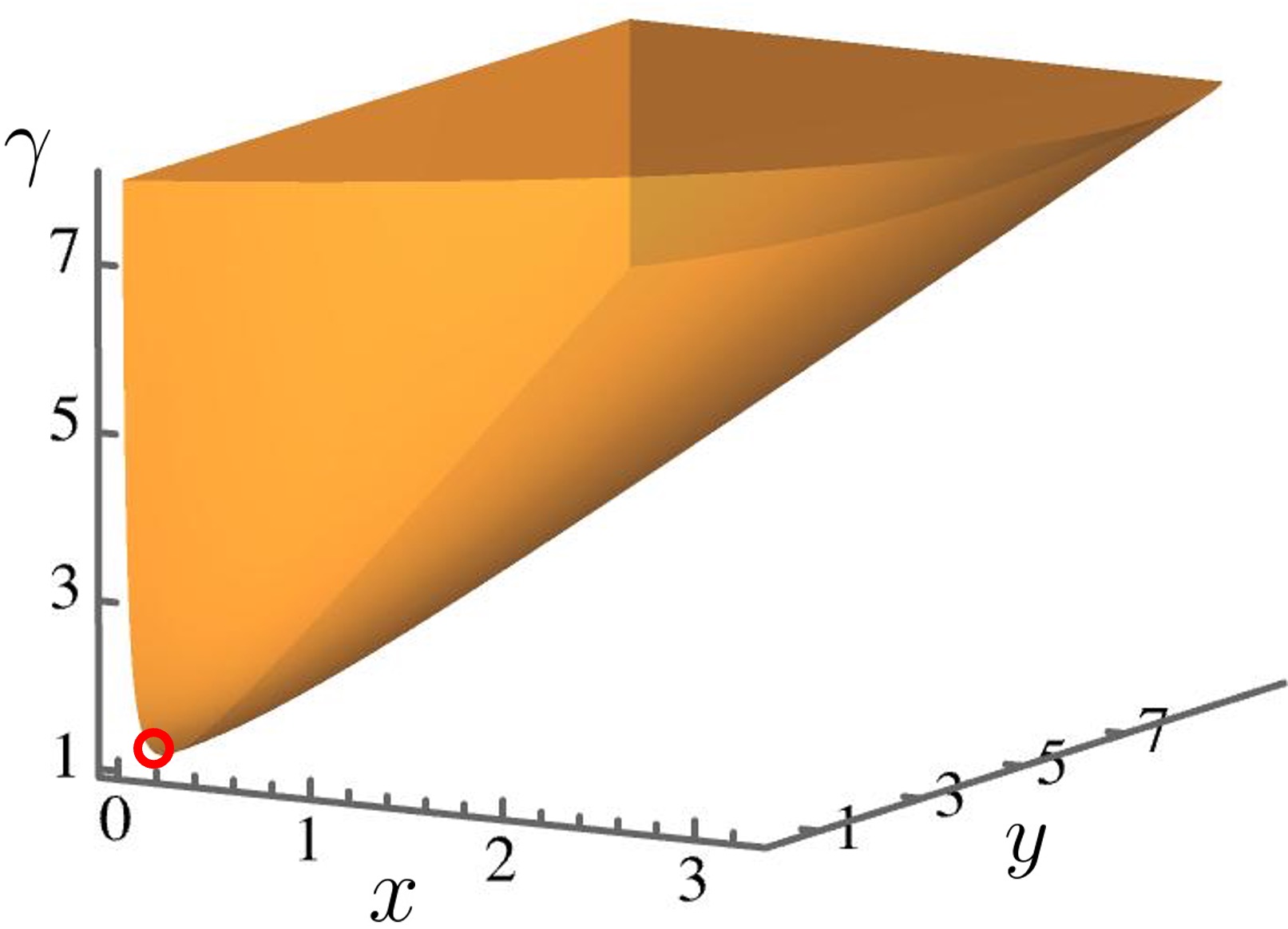}
    \caption{The LMI feasible region $\mathcal{F}_\infty$ in \cref{eq:LMI-hinf-nont-attained} of the 1-dim $\mathcal{H}_\infty$ example. The infimum of \cref{eq:LMI-hinf-nont-attained} is achieved on the boundary point $(\gamma,y,x)=(1,1,0) \notin \mathcal{F}_\infty$ (highlighted as the red circle). In this case, the corresponding controller $k = yx^{-1} \to \infty$.    
    }
    \label{fig:feasible-hinf-example}
\end{figure}

\subsubsection{Proof of \Cref{proposition:hinf-lifting}} \label{subsection:hinf-lifting}

We divide the proof of \Cref{proposition:hinf-lifting} into two parts. 

\vspace{2mm}
\noindent\textbf{Part I: Proving $\operatorname{epi}_{\geq} (J_{\infty})=\pi_{K,\gamma}(\mathcal{L}_{\infty})$.}

\vspace{2mm}

We first prove the inclusion $\operatorname{epi}_{\geq} (J_{\infty})\subseteq \pi_{K,\gamma}(\mathcal{L}_{\infty})$. Let $(K,\gamma)\in\operatorname{epi}_{\geq}(J_\infty)$ be arbitrary. We then have $K \in \mathcal{K}$ and $\gamma\geq J_{\infty}(K)$. 
With $A+BK$ stable and $(A+BK,B_w)$ controllable (since $B_w$ has full row rank), it follows from the non-strict version of \Cref{lemma:bounded_real} that there exists a symmetric matrix $P$ satisfying the LMI in \Cref{eq:lifted-set-Hinf-state-feedback}. 
We now argue that $P\succ0$. Indeed by Schur complement, \Cref{eq:lifted-set-Hinf-state-feedback} is equivalent to 
$$
(A+BK)^\tr P + P(A+BK) + \gamma^{-1} \left(Q+K^\tr RK + PB_wB_w^\tr P\right) \preceq 0, 
$$
which implies that $P$ is a solution to the Lyapunov equation
\begin{equation}
\label{eq:Riccati-Lyapunov}
    (A+BK)^\tr P + P(A+BK) + Q' = 0
\end{equation}
for some $Q'\succ0$ since $Q\succ0$.
With $A+BK$ stable and $Q'\succ0$, the standard Lyapunov argument \cite[Lemma 3.18 (ii)]{zhou1996robust} guarantees $P\succ0$. Thus, we have $(K,\gamma,P) \in \mathcal{L}_\infty$. By the arbitrariness of $(K,\gamma) \in \operatorname{epi}_{\geq} (J_{\infty})$, we know $\operatorname{epi}_{\geq} (J_{\infty})\subseteq \pi_{K,\gamma}(\mathcal{L}_{\infty})$.

Next we show the inclusion $\pi_{\mK,\gamma}(\mathcal{L}_\infty)\subseteq \operatorname{epi}_{\geq} (J_\infty)$. Let $(K,\gamma,P)\in\mathcal{L}_{\infty}$. Using the same argument as above, we know that $P\succ 0$ is a solution to \cref{eq:Riccati-Lyapunov} for some $Q'\succ0$, and thus the standard Lyapunov argument \cite[Lemma 3.19 (ii)]{zhou1996robust} guarantees that $A+BK$ is stable, i.e., $K \in \mathcal{K}$. With $A+BK$ stable, it follows from the non-strict version of \Cref{lemma:bounded_real} that $\gamma \geq J_{\infty}(K)$. Thus $(K,\gamma)\in\operatorname{epi}_{\geq}(J_\infty)$, and by the arbitrariness of $(K,\gamma,P)\in\mathcal{L}_\infty$ we have $\pi_{\mK,\gamma}(\mathcal{L}_\infty)\subseteq \operatorname{epi}_{\geq} (J_\infty)$. The proof of Part I is now complete.

\vspace{2mm}

\noindent\textbf{Part II: Proving $\Phi_\infty$ is a diffeomorphism between $\mathcal{L}_{\infty}$ and $\mathcal{F}_{\infty}$.}
\vspace{2mm}

Consider the mapping $\Phi_\infty$ given by $\Phi_\infty(K,\gamma,P)=(\gamma,KP^{-1},P^{-1})$ in \Cref{eq:mapping-hinf-state-feedback}. We will first show that $\Phi_\infty(\mathcal{L}_\infty)\subseteq\mathcal{F}_\infty$, i.e.,
$$
\Phi_\infty(K,\gamma, P) \in \mathcal{F}_{\infty}, \qquad \forall (K,\gamma,P)\in\mathcal{L}_{\infty}. 
$$
Let $(K,\gamma,P)\in\mathcal{L}_\infty$ be arbitrary, and let $X = P^{-1} \succ 0$ and $Y = K P^{-1}$ so that $\Phi_\infty(K,\gamma, P) = (\gamma, Y, X)$. We observe that 
$$
\begin{aligned}
& P\succ 0\ \text{and}\ 
\begin{bmatrix}
(A+BK)^\tr P \!+\! P (A+BK) & PB_w & Q^{1/2} & K^\tr R^{1/2} \\
B_w^\tr P & -\gamma I & 0 & 0 \\
Q^{1/2} & 0 & -\gamma I & 0 \\
R^{1/2}K & 0 & 0 & -\gamma I
\end{bmatrix}\preceq 0 \\
\Longleftrightarrow\quad
& X\succ 0\ \text{and}\ \begin{bmatrix}
X  & 0 & 0 \\
0 & I & 0 \\
0 & 0 & I
\end{bmatrix} \begin{bmatrix}
(A+BK)^\tr P \!+\! P (A+BK) & PB_w & C^\tr \\
B_w^\tr P & -\gamma I & 0 \\
C & 0 & -\gamma I
\end{bmatrix} \begin{bmatrix}
X  & 0 & 0 \\
0 & I & 0 \\
0 & 0 & I
\end{bmatrix} \preceq 0 \\
\Longleftrightarrow \quad &
X\succ 0\ \text{and}\ 
\begin{bmatrix}
    AX\!+\! BY \!+\! XA^\tr \!+\!Y^\tr B^\tr & B_w & XQ^{1/2} & Y^\tr R^{1/2} \\
    B_w^\tr & -\gamma I & 0 & 0 \\
    Q^{1/2}X & 0 & -\gamma I & 0 \\
    R^{1/2}Y & 0 & 0 & -\gamma I
\end{bmatrix} \preceq 0,
\end{aligned}
$$
which implies $(\gamma,Y,X)\in\mathcal{F}_\infty$. By the definitions of $\mathcal{L}_\infty$ and $\mathcal{F}_\infty$ we see that $\Phi_\infty(\mathcal{L}_\infty)\subseteq\mathcal{F}_\infty$.

Then, we define the mapping 
\begin{equation} \label{eq:inverse-Phi-Hinf}
\Psi_\infty(\gamma, Y, X) = (YX^{-1}, \gamma, X^{-1}), \qquad \forall (\gamma, Y, X) \in \mathcal{F}_\infty. 
\end{equation}
By reversing the derivation above, we can easily verify that $\Psi_\infty(\mathcal{F}_\infty) \subseteq \mathcal{L}_\infty$. 

Finally, it is straightforward to check that 
$$
\begin{aligned}
    \Phi_\infty \circ \Psi_\infty(\gamma, Y, X) &= (\gamma, Y, X), \quad \forall (\gamma, Y, X) \in \mathcal{L}_\infty \\
    \Psi_\infty \circ \Phi_\infty(K, \gamma, P) &= (K, \gamma, P), \quad  \forall (K, \gamma, P) \in \mathcal{F}_\infty,
\end{aligned}
$$
i.e., we have $\Phi_\infty^{-1} = \Psi_\infty$. It is clear that $\Phi_\infty$ in \Cref{eq:mapping-hinf-state-feedback} and its inverse $\Psi_\infty$ in \cref{eq:inverse-Phi-Hinf} are both in fact $C^\infty$ over their domains, since each element of $\Phi_\infty$ or $\Psi_\infty$ is a rational function.

%% file: app_sec4.tex
\section{Proofs for Output Feedback Policy Optimization} \label{appendix:dynamic-policies}
In this section, we provide some extra discussions and missing proofs for \ECL{} applied to output feedback control problems (\cref{subsection:dynamic-policies}).

\subsection{Auxiliary Notations and Results}

Here we present some auxiliary notations and results for subsequent analysis. Note that with the notations in~\eqref{eq:notations_general_H2}, the closed-loop matrices in \Cref{eq:closed-loop-matrices} can be written as 
\[
\begin{aligned}
A_{\mathrm{cl}}(\mK)
={} &
\begin{bmatrix}
A+B_2 D_\mK C_2 & B_2C_\mK \\ B_\mK C_2 & A_\mK
\end{bmatrix}, 
& B_{\mathrm{cl}}(\mK)
={} & \begin{bmatrix} B_1 + B_2D_\mK D_{21} \\ B_{\mK}D_{21}  \end{bmatrix}, \\
C_{\mathrm{cl}}(\mK)
={} &
 \begin{bmatrix}
        C_1 + D_{12}D_\mK C_2 & D_{12}C_\mK
    \end{bmatrix}, &
D_\mathrm{cl}(\mK) ={} & D_{12}D_\mK D_{21} .
\end{aligned}
\]

We will extensively use the inverse identity for a $2$ by $2$ block matrix 
\begin{equation}\label{eq:block_matrix_inversion_formula}
\begin{bmatrix}
M_{11} & M_{12} \\ M_{21} & M_{22}
\end{bmatrix}^{-1}
=\begin{bmatrix}
\left(M_{11}-M_{12}M_{22}^{-1}M_{21}\right)^{-1} & 
-\left(M_{11}-M_{12}M_{22}^{-1}M_{21}\right)^{-1}M_{12}M_{22}^{-1} \\
-M_{22}^{-1}M_{21}\left(M_{11}-M_{12}M_{22}^{-1}M_{21}\right)^{-1}
&
\left(M_{22}-M_{21}M_{11}^{-1}M_{12}\right)^{-1}
\end{bmatrix},
\end{equation}
where $M=\begin{bmatrix}
M_{11} & M_{12} \\ M_{21} & M_{22}
\end{bmatrix}$ is any invertible matrix. 
It turns out that it will be convenient to introduce the mappings $\Phi_M$ and $\Phi_\Lambda$ defined on $\mathbb{R}^{(m+n)\times(p+n)}\times\mathbb{S}^{2n}_{++}$ by
\begin{subequations}
\label{eq:Phi_M_V_def}
\begin{align}
\Phi_M(\mK,P)
\coloneqq{} & P_{12}B_\mK C_2 (P^{-1})_{11}
+ P_{11} B_2 C_\mK (P^{-1})_{21} 
 + P_{12}A_\mK (P^{-1})_{21} \nonumber \\
 & + P_{11}(A+B_2 D_\mK C_2)(P^{-1})_{11}, \\
\Phi_\Lambda(\mK,P)
\coloneqq{} &
\begin{bmatrix}
D_{\mK} & D_\mK C_2 (P^{-1})_{11} + C_\mK (P^{-1})_{21} \\
P_{11}B_2 D_\mK + P_{12} B_\mK &
\Phi_M(\mK, P)
\end{bmatrix},
\end{align}
\end{subequations}
where the parameter $\mK$ is partitioned as \cref{eq:dynamic-policies-Dk}. We also define the mappings $\Psi_{\mK}$ and $\Psi_P$ by
\begin{subequations}
\label{eq:Psi_K_P_def}
\begin{align}
\Psi_{\mK}\left(
\Lambda,X,Y,\Xi\right)
\coloneqq{} &
\begin{bmatrix} I & 0 \\ YB_2 & \Xi \end{bmatrix}^{-1}
\left(\Lambda-\begin{bmatrix} 0_{m\times p} & 0_{m\times n} \\ 0_{n\times p} & YAX \end{bmatrix}\right)
\begin{bmatrix} I & C_2X \\ 0 & -\Xi^{-1}(Y-X^{-1})X  \end{bmatrix}^{-1}, \\ 
\Psi_P(X,Y,\Xi)\coloneqq{} &
\begin{bmatrix} Y & \Xi \\ \Xi^\tr  & \Xi^\tr (Y-X^{-1})^{-1}\Xi \end{bmatrix}
\end{align}
\end{subequations}
for any $\Lambda\in\mathbb{R}^{(m+n)\times(p+n)}$, $\Xi\in\mathrm{GL}_n$ and $X,Y\in\mathbb{S}^n_{++}$ satisfying $\begin{bmatrix} X & I \\ I & Y \end{bmatrix}\succ 0$. Note that $\begin{bmatrix} X & I \\ I & Y \end{bmatrix}\succ 0$ implies $X\succ 0$ and $Y-X^{-1}\succ 0$ by Schur complement, which is why we can invert $Y-X^{-1}$ and $-\Xi^{-1}(Y-X^{-1})X$ in the definitions of $\Psi_P$ and $\Psi_\mK$. By straightforward algebraic calculation, it can be checked that
\begin{equation}
\label{eq:Psi_K_expanded}
\begin{aligned}
& \Psi_{\mK}\!\left(\begin{bmatrix}
G & F \\ H & M
\end{bmatrix},X,Y,\Xi\right) \\
={} &
\begin{bmatrix}
I & 0 \\
-\Xi^{-1}YB_2 & \Xi^{-1}
\end{bmatrix}
\begin{bmatrix} G & F \\ H & M-YAX \end{bmatrix}
\begin{bmatrix}
I & -C_2X\left(-\Xi^{-1}(Y\!-\!X^{-1})X\right)^{-1} \\
0 & \left(-\Xi^{-1}(Y\!-\!X^{-1})X\right)^{-1}
\end{bmatrix} \\
={} &
\begin{bmatrix}
G & (F-GC_2X)\left(-\Xi^{-1}(Y\!-\!X^{-1})X\right)^{-1} \\
\Xi^{-1}(H-YB_2G) &
\Xi^{-1}(M-Y(A\!-\!B_2GC_2)X-HC_2 X-YB_2 F)
\left(-\Xi^{-1}(Y\!-\!X^{-1})X\right)^{-1}
\end{bmatrix}.
\end{aligned}
\end{equation}
The mappings $\Phi_M,\Phi_\Lambda$ and $\Psi_{\mK},\Psi_P$ will be used for establishing the \ECL{}s for both LQG and $\mathcal{H}_\infty$ output feedback control. 

The following three lemmas will be useful for later analysis.
\begin{lemma}
\label{lemma:Psi_P_basic_properties}
\begin{enumerate}
\item Let $X,Y\in\mathbb{S}^{n}_{++}$ satisfy $\begin{bmatrix} X & I \\ I & Y \end{bmatrix}\succ 0$, and let $\Xi\in\mathrm{GL}_n$ be arbitrary. Then
\[
\Psi_P(X,Y,\Xi)
\succ 0.
\]
\item Let $P\in\mathbb{S}^{2n}_{++}$ be arbitrary with $P_{12}\in\mathrm{GL}_n$. Then
\[
\begin{bmatrix}
(P^{-1})_{11} & I \\
I & P_{11}
\end{bmatrix} \succ 0.
\]
\end{enumerate}
\end{lemma}
\begin{proof}[Proof of \Cref{lemma:Psi_P_basic_properties}]
\begin{enumerate}
\item
Note that $\begin{bmatrix} X & I \\ I & Y\end{bmatrix}\succ 0$ implies $X\succ 0$ and $Y\succ X^{-1}$. From the invertibility of $\Xi$, we see that $\Xi^\tr(Y-X^{-1})^{-1}\Xi\succ 0$ and
\[
Y - \Xi\left(
\Xi^\tr(Y-X^{-1})^{-1}\Xi\right)^{-1}\Xi^\tr
=X^{-1}\succ 0.
\]
By Schur complement, we get $\begin{bmatrix} Y & \Xi \\ \Xi^\tr  & \Xi^\tr (Y-X^{-1})^{-1}\Xi \end{bmatrix} \succ 0$.

\item Given an arbitrary $P\in\mathbb{S}^{2n}_{++}$ satisfying $P_{12}\in\mathrm{GL}_n$, by~\eqref{eq:block_matrix_inversion_formula}, we have
\[
\left(P^{-1}\right)_{11}
=\left(
P_{11} - P_{12}P_{22}^{-1}P_{12}^\tr
\right)^{-1},
\]
and thus
\[
P_{11} - \left(P^{-1}\right)_{11}^{-1}
=P_{12}P_{22}^{-1}P_{12}^\tr\succ 0,
\]
where the last inequality follows from $P_{22}\succ 0$ and $P_{12}\in\mathrm{GL}_n$. Together with $P_{11}\succ 0$, we can apply Schur complement to get
$
\begin{bmatrix}
\left(P^{-1}\right)_{11} & I \\ I & P_{11}
\end{bmatrix}\succ 0
$.\qedhere
\end{enumerate}
\end{proof}

\begin{lemma}
\label{lemma:Phi_Psi_identity}
Let $\Lambda\in\mathbb{R}^{(m+n)\times(p+n)}$, $\Xi\in \mathrm{GL}_n$ be arbitrary, and let $X,Y\in\mathbb{S}^{2n}_{++}$ satisfy $\begin{bmatrix} X & I \\ I & Y \end{bmatrix}\succ 0$. Then
\[
\left(\Psi_P(X,Y,\Xi)^{-1}\right)_{11} = X,
\]
and
\[
\Phi_\Lambda(\Psi_\mK(\Lambda,X,Y,\Xi),\Psi_P(X,Y,\Xi))
=\Lambda.
\]
\end{lemma}
\begin{proof}[Proof of \Cref{lemma:Phi_Psi_identity}]

By straightforward calculation, it is not hard to check that
\[
\begin{bmatrix}
X & -X(Y-X^{-1})\Xi^{-\tr} \\
-\Xi^{-1}(Y-X^{-1})X &
\Xi^{-1}(YXY-Y)\Xi^{-\tr}
\end{bmatrix}
\begin{bmatrix}
Y & \Xi \\ \Xi^\tr &
\Xi^\tr(Y-X^{-1})^{-1}\Xi
\end{bmatrix}
=I,
\]
which justifies that
\[
\Psi_{P}(X,Y,\Xi)^{-1}
=\begin{bmatrix}
X & -X(Y-X^{-1})\Xi^{-\tr} \\
-\Xi^{-1}(Y-X^{-1})X &
\Xi^{-1}(YXY-Y)\Xi^{-\tr}
\end{bmatrix}
=\begin{bmatrix}
X & \Pi^\tr \\ \Pi & \Xi^{-1}(YXY-Y)\Xi^{-\tr}
\end{bmatrix},
\]
where we denote $\Pi = -\Xi^{-1}(Y-X^{-1})X$. Consequently,
\[
\left(\Psi_P(X,Y,\Xi)^{-1}\right)_{11}=X,
\qquad \left(\Psi_P(X,Y,\Xi)^{-1}\right)_{12}
=\Pi^\tr,
\qquad
\left(\Psi_P(X,Y,\Xi)^{-1}\right)_{21}
=\Pi.
\]
Next, we write $\Lambda$ as
$
\Lambda=\begin{bmatrix}
G & F \\ H & M
\end{bmatrix}$
where $G\in\mathbb{R}^{m\times p}$ and $F, H, M$ are real matrices with proper dimensions. From~\eqref{eq:Psi_K_expanded} we have
\[
\begin{aligned}
\Psi_{\mK}(\Lambda,X,Y,\Xi)
={} &
\begin{bmatrix}
G & (F-GC_2X)\Pi^{-1} \\
\Xi^{-1}(H-YB_2G) &
\Xi^{-1}(M-Y(A-B_2GC_2)X-HC_2X-YB_2F)
\Pi^{-1}
\end{bmatrix}.
\end{aligned}
\]
By plugging these identities into the definitions of $\Phi_M$ and $\Phi_\Lambda$, we get
\[
\begin{aligned}
\Phi_M(\Psi_{\mK}(\Lambda,X,Y,\Xi),\Psi_P(X,Y,\Xi))
={} & (H-YB_2G)C_2X
+YB_2(F-GC_2X) + Y(A+B_2GC_2)X \\
& +(M-Y(A-B_2GC_2)X-HC_2X-YB_2F) \\
={} & M,
\end{aligned}
\]
and
\begin{align*}
\Phi_\Lambda(\Psi_{\mK}(\Lambda,X,Y,\Xi),\Psi_P(X,Y,\Xi))
={} &
\begin{bmatrix}
G & GC_2X + (F-GC_2X) \\
YB_2G + (H-YB_2G) & \Phi_M(\Psi_{\mK}(\Lambda,X,Y,\Xi),\Psi_P(X,Y,\Xi))
\end{bmatrix} \\
={} &
\begin{bmatrix}
G & F \\ H & M
\end{bmatrix} = \Lambda,
\end{align*}
which completes the proof.
\end{proof}

\begin{lemma}
\label{lemma:Psi_Phi_identity}
Let $\mK\in\mathbb{R}^{(m+n)\times(p+n)}$ and $P\in\mathbb{S}^{2n}_{++}$ be arbitrary. Then
\[
\Psi_P\!\left(
(P^{-1})_{11},P_{11},P_{12}
\right)=P,
\]
and
\[
\Psi_{\mK}\!\left(\Phi_\Lambda(\mK,P),
(P^{-1})_{11},
P_{11},P_{12}
\right)
=\mK.
\]
\end{lemma}
\begin{proof}[Proof of \Cref{lemma:Psi_Phi_identity}]

For the first identity, straightforward calculation leads to
\[
\Psi_P\!\left(
(P^{-1})_{11},P_{11},P_{12}
\right)
=
\begin{bmatrix}
P_{11} & P_{12} \\
P_{12}^\tr &
P_{12}^\tr \left(P_{11}-(P^{-1})_{11}^{-1}\right)^{-1}P_{12}
\end{bmatrix}.
\]
On the other hand, the identity~\eqref{eq:block_matrix_inversion_formula} gives
$
(P^{-1})_{11}
=\left(P_{11}-P_{12}P_{22}^{-1}P_{12}^\tr\right)^{-1}$, and so,
\begin{equation}\label{eq:LQG_P22_identity}
P_{12}^\tr\left(P_{11}-(P^{-1})_{11}^{-1}\right)^{-1}P_{12}
=P_{12}^\tr\left(P_{11}-\left(P_{11}-P_{12}P_{22}^{-1}P_{12}^\tr\right)\right)^{-1}P_{12}
=P_{22},
\end{equation}
which justifies $\Psi_P\!\left(
(P^{-1})_{11},P_{11},P_{12}
\right)=P$.

To show the second identity, we note that
\[
\begin{aligned}
(P^{-1})_{21}
={} &
-P_{22}^{-1}P_{12}^\tr (P^{-1})_{11} \\
={} &
-P_{12}^{-1}
\left(P_{11}-(P^{-1})_{11}^{-1}\right) P_{12}^{-\tr}
\cdot P_{12}^\tr (P^{-1})_{11}, \\
={} &
-P_{12}^{-1}
\left(P_{11}-(P^{-1})_{11}^{-1}\right) (P^{-1})_{11},
\end{aligned}
\]
where the first step follows from~\eqref{eq:block_matrix_inversion_formula}, and the second step uses the identity~\eqref{eq:LQG_P22_identity}.
As a result, by~\eqref{eq:Psi_K_expanded},
\[
\begin{aligned}
& \Psi_{\mK}\!\left(\Phi_\Lambda(\mK,P),
(P^{-1})_{11},
P_{11},P_{12}
\right) \\
={} &
\begin{bmatrix}
D_{\mK} & (D_{\mK}C_2(P^{-1})_{11} \!+\! C_\mK(P^{-1})_{21} \!-\! D_{\mK}C_2(P^{-1})_{11})(P^{-1})_{21}^{-1} \\
P_{12}^{-1}(P_{11}B_2D_\mK
\!+\! P_{12}B_\mK \!-\! P_{11}B_2D_\mK) &
(1)
\end{bmatrix} \\
={} &
\begin{bmatrix}
D_\mK & C_\mK \\
B_\mK & (1)
\end{bmatrix},
\end{aligned}
\]
in which
\[
\begin{aligned}
(1) ={} &
P_{12}^{-1}
\left(\Phi_M(\mK,\gamma,P,\Gamma)
-P_{11}(A-B_2D_\mK C_2)(P^{-1})_{11}\right)(P^{-1})_{21}^{-1} \\
&
-P_{12}^{-1}
\left(P_{11}B_2D_\mK C_2(P^{-1})_{11}+P_{12}B_\mK C_2(P^{-1})_{11}
\right) (P^{-1})_{21}^{-1} \\
& - P_{12}^{-1}\left(
P_{11}B_2D_\mK C_2(P^{-1})_{11}
+P_{11}B_2C_\mK(P^{-1})_{21}
\right)(P^{-1})_{21}^{-1}
\\
={} & A_\mK,
\end{aligned}
\]
justifying that $\Psi_{\mK}\!\left(\Phi_\Lambda(\mK,P),
(P^{-1})_{11},
P_{11},P_{12}
\right)=\mK$.
\end{proof}

\subsection{A Critical Lemma for Establishing \ECL{}} 

For LQG and $\mathcal{H}_\infty$ output feedback control, one of the main difficulties in establishing the corresponding \ECL{}s is showing the inclusion
\[
\operatorname{\pi}_{x,\lambda}(\mathcal{L}_{\mathrm{lft}})
\subseteq \operatorname{cl}\operatorname{epi}_{\geq}(f),
\]
which is mostly due to the intricacy between the strict and non-strict versions of the LMIs used in constructing the lifted set $\mathcal{L}_{\mathrm{lft}}$.
Here we provide a technical lemma that will be useful for establishing the second inclusion in the definition of \ECL{} \cref{eq:epi-graph-lifting} for LQG and $\mathcal{H}_\infty$ output feedback control.

\begin{lemma}\label{lemma:ECL_second_inclusion}
Suppose there exist sets $C$, $F_0$, $G$ and functions $h:F_0\rightarrow\mathbb{R}^{q}$, $\Psi:F_0\times G\rightarrow\mathbb{R}^k$ such that the following conditions hold:
\begin{enumerate}
\item $C$ is a convex set in $\mathbb{R}^q$ with a nonempty interior, and $F_0$ is a finite-dimensional convex set.
\item $h$ is an affine function, and $h^{-1}(\operatorname{int} C)$ is nonempty.
\item $\Psi(\cdot, z)$ is continuous on $F_0$ for any $z\in G$.
\end{enumerate}
Then
\[
\Psi(h^{-1}(C)\times G)\subseteq\operatorname{cl}
\Psi(h^{-1}(\operatorname{int}C)\times G).
\]
\end{lemma}
\begin{proof}
Let $x\in \Psi(h^{-1}(C)\times G)$ be arbitrary, and find $y\in h^{-1}(C)$, $z\in G$ such that $\Psi(y,z)=x$. Moreover, let $y^\circ\in h^{-1}(\operatorname{int}C)$ be arbitrary. We now define
\[
y_n = \frac{n-1}{n}y + \frac{1}{n}y^\circ,
\qquad n=1,2,\ldots
\]
By the convexity of $h^{-1}(C)$, we can see that $y_n\in h^{-1}(C)$ for all $n\geq 1$. Moreover, $y^\circ\in h^{-1}(\operatorname{int}C)$ implies that $h(y^\circ)+\epsilon\mathbb{B}_q\subseteq C$ for all sufficiently small $\epsilon>0$, where $\mathbb{B}_q$ denotes the open unit ball in $\mathbb{R}^q$. This further leads to
\begin{align*}
h(y_n) + \frac{\epsilon}{n}\mathbb{B}_{q}
={} &
\frac{n-1}{n}h(y) + \frac{1}{n}\left(
h(y^\circ) + \epsilon \mathbb{B}_q
\right) \subseteq C
\end{align*}
for all sufficiently small $\epsilon$. Therefore $h(y_n)\in\operatorname{int}C$, i.e., $y_n\in h^{-1}(\operatorname{int}C)$.

Now let $x_n=\Psi(y_n,z)\in\Psi(h^{-1}(\operatorname{int}C)\times G)$. By the continuity of $\Psi(\cdot,z)$ on $F_0$, we can infer from $\lim_{n\rightarrow\infty}y_n = y$ that $\lim_{n\rightarrow\infty} x_n = 
\lim_{n\rightarrow\infty}\Phi(y_n,z)
=\Phi(y,z)=x$. By the arbitariness of $x\in\Psi(h^{-1}(C)\times G)$, we get the desired conclusion.
\end{proof}

\subsection{Proofs for LQG Control} \label{appendix:LQG-control}

We here provide the proof details for \Cref{lemma:inclusion-projection-LQG}, which include two parts:
\begin{itemize}
    \item The mapping $\Phi_{\LQG}$ given by \cref{eq:diffeomorphism-LQG} is a $C^2$ diffeomorphism from $\mathcal{L}_{\LQG}$ to $\mathcal{F}_{\LQG}\times\mathrm{GL}_{n}$, and its inverse $\Phi_{\LQG}^{-1}$ is given by $\Psi_{\LQG}$ defined in \cref{eq:Psi-LQG}. This is proved in \cref{appendix:diffeomorphsim-Phi-LQG}. 
    \item The following inclusion relationship holds
    \begin{equation} \label{eq:LQG-inclusion}
    \operatorname{epi}_>(J_{\LQG,n})  \subseteq \pi_{\mK, \gamma} (\mathcal{L}_{\LQG}) \subseteq \operatorname{cl}\, \operatorname{epi}_{\geq}(J_{\LQG,n}). 
    \end{equation}
    This is proved in \cref{Appendix:proof-inclusion-LQG}. 
\end{itemize}

\subsubsection{Diffeomorphism from $\mathcal{L}_{\LQG}$ to $\mathcal{F}_{\LQG}\times\mathrm{GL}_{n}$} \label{appendix:diffeomorphsim-Phi-LQG}

Denote
\begin{align*}
\mathcal{L}^0_{\LQG}
& =\left\{ (\mK, \gamma, P, \Gamma) \left|\,
       \mK\in \mathcal{V}_{n,0}, \;\gamma\in\mathbb{R},\;
         P \in \mathbb{S}^{2n}_{++}, \; P_{12} \in \mathrm{GL}_n,\;
         \Gamma\in\mathbb{S}^{n+m}
        \right.\right\}, \\
\mathcal{F}^0_{\LQG}
& =\left\{
\left(\gamma, \Lambda,X, Y, \Gamma \right) \left|\,
\gamma\in\mathbb{R},\;
\Lambda\in\mathcal{V}_{n,0},\;
       \begin{bmatrix}
        X & I \\
        I & Y
\end{bmatrix}\in\mathbb{S}_{++}^{2n},\;
\Gamma\in\mathbb{S}^{n+m}
        \right.\right\}.
\end{align*}
Evidently $\mathcal{L}^0_{\LQG}$ and $\mathcal{F}^0_{\LQG}$ can be identified with certain open subsets of some Euclidean spaces, and $\mathcal{F}^0_{\LQG}$ is convex. Moreover,
\begin{align*}
\mathcal{L}_{\LQG}
&=\left\{(\mK,\gamma,P,\Gamma)\in\mathcal{L}^0_{\LQG}
\left|\,
\begin{aligned} 
        -\begin{bmatrix} A_{\mathrm{cl}}(\mK)^\tr P+PA_{\mathrm{cl}}(\mK) & PB_{\mathrm{cl}}(\mK) \\ B_{\mathrm{cl}}(\mK)^\tr P & -\gamma I \end{bmatrix} 
        \in\mathbb{S}^{3n+p}_{+}, \\
        \begin{bmatrix} P & C_{\mathrm{cl}}(\mK)^\tr \\ C_{\mathrm{cl}}(\mK) & \Gamma \end{bmatrix} \in\mathbb{S}^{3n+m}_+,\;  \mathrm{tr}(\Gamma) \leq \gamma
        \end{aligned}
\right.\right\}, \\
\mathcal{F}_{\LQG}
&=\left\{\left(\gamma, \Lambda,X, Y, \Gamma \right) \in\mathcal{F}^0_{\LQG}
\left|\,
\begin{aligned}
& -\mathcal{A}(\gamma,\Lambda,X,Y,\Gamma)
\in \mathbb{S}_{+}^{3n+p}, \\
& \mathcal{B}(\gamma,\Lambda,X,Y,\Gamma)
\in \mathbb{S}_{+}^{3n+m},
\; \operatorname{tr}(\Gamma)\leq\gamma
\end{aligned}
\right.\right\},
\end{align*}
where we remind the readers of the notations
\begin{align*}
    {\mathcal{A}}\!\left(\gamma, \!\begin{bmatrix}
    0 & F \\
    H & M
    \end{bmatrix}\!,X,Y,\Gamma\right)& \!=\!
    \begin{bmatrix} AX \!+\! B_2F \!+\! (AX \!+\! B_2F)^\tr  & M^\tr \!+\! A & B_1 \\
            *  & YA \!+\! HC_2 \!+\! (YA \!+\! HC_2)^\tr  & YB_1\!+\!HD_{21}  \\       
            *  & *  & -\gamma I\end{bmatrix}\!\!, \\
    {\mathcal{B}}\!\left(\gamma, \!\begin{bmatrix}
    0 & F \\
    H & M
    \end{bmatrix}\!,X,Y,\Gamma\right)&\!=\!
    \begin{bmatrix} \ \ \ X\ \ \  & \ \ \ I\ \ \  & (C_1X\!+\!D_{12}F)^\tr  \\ * & \ \ \ Y\ \ \  & C_1^\tr \\ * & * &\Gamma \end{bmatrix}.
\end{align*}
We extend the domain of $\Phi_{\LQG}$ so that it is now defined on the larger set $\mathcal{L}^0_{\LQG}$:
\[
\Phi_{\LQG}(\mK,\gamma,P,\Gamma)=\left(\gamma,
\Phi_\Lambda(\mK,P),(P^{-1})_{11},P_{11},\Gamma, P_{12}\right),
\quad
\forall (\mK,\gamma,P,\Gamma)\in\mathcal{L}^0_{\LQG},
\]
where $\Phi_\Lambda(\mK,P)$ is defined in \eqref{eq:Phi_M_V_def}.
We also extend the domain of $\Psi_{\LQG}$ so that it is now defined on $\mathcal{F}^0_{\LQG}\times\mathrm{GL}_n$:
\[
\Psi_{\LQG}(\mZ,\Xi)=\left(\Psi_{\mK}(\Lambda,X,Y,\Xi),\gamma, \Psi_P(X,Y,\Xi),\Gamma\right),
\qquad\forall \mZ = \left(\gamma, 
\Lambda, X,Y,\Gamma \right) \in \mathcal{F}^{0}_{\LQG},\ \Xi\in\mathrm{GL}_n,
\]
where $\Psi_{\mK}(\Lambda,X,Y,\Xi)$ and $\Psi_P(X,Y,\Xi)$ are defined in~\eqref{eq:Psi_K_P_def}.

\vspace{3pt}
Our subsequent proof consists of the following parts:
\begin{enumerate}
\item Showing that $\Phi_{\LQG}$ is a $C^2$ diffeomorphism from $\mathcal{L}^0_{\LQG}$ to $\mathcal{F}^0_{\LQG}\times\mathrm{GL}_n$, with $\Phi_{\LQG}^{-1}=\Psi_{\LQG}$.

\item Showing that for any $(\mK,\gamma,P,\Gamma)\in\mathcal{L}^0_{\LQG}$, we have $(\mK,\gamma,P,\Gamma)\in\mathcal{L}_{\LQG}$ if and only if 
$\Phi_{\LQG}(\mK,\gamma,P,\Gamma)\in\mathcal{F}_{\LQG}\times\mathrm{GL}_n$.
\end{enumerate}

\vspace{3pt}
\noindent\textbf{Part I.} We first show that
\[
\Psi_{\LQG}(\mathcal{F}^0_{\LQG}\times\mathrm{GL}_n)\subseteq\mathcal{L}^0_{\LQG}.
\]
By the definitions of $\mathcal{L}^0_{\LQG}$, $\mathcal{F}^0_{\LQG}$ and $\Psi_{\LQG}$, it is not hard to see that this inclusion holds as long as $\Psi_\mK(\Lambda,X,Y,\Xi)\in\mathcal{V}_{n,0}$ and $\Psi_P(X,Y,\Xi)\in\mathbb{S}^{2n}_{++}$ for any $(\gamma,\Lambda,X,Y,\Gamma)\in\mathcal{F}^0_{\LQG}$ and $\Xi\in\mathrm{GL}_n$, which can be directly justified by~\eqref{eq:Psi_K_expanded} and the first part of \Cref{lemma:Psi_P_basic_properties} respectively.

Next, we show that
\[
\Phi(\mathcal{L}^0_{\LQG})\subseteq\mathcal{F}^0_{\LQG}\times\mathrm{GL}_n.
\]
By definition, we only need to show
\[
\begin{bmatrix}
(P^{-1})_{11} & I \\
I & P_{11}
\end{bmatrix} \succ 0
\qquad\forall P\in\mathbb{S}^{2n}_{++} \text{ with }P_{12}\in\mathrm{GL}_n,
\]
which can be directly justified by the second part of \Cref{lemma:Psi_P_basic_properties}.

We are now allowed to form the compositions $\Psi_{\LQG}\circ\Phi_{\LQG}:\mathcal{L}^0_{\LQG}\rightarrow\mathcal{L}^0_{\LQG}$ and $\Phi\circ\Psi:\mathcal{F}^0_{\LQG}\times\mathrm{GL}_n\rightarrow \mathcal{F}^0_{\LQG}\times\mathrm{GL}_n$. We will now show that these two compositions are the identity functions on their domains:
\begin{enumerate}
\item Let $\mZ=\left(\gamma,\Lambda,X,Y,\Gamma\right)\in\mathcal{F}^0_{\LQG}$ and $\Xi\in\mathrm{GL}_n$ be arbitrary. We then have
\[
\begin{aligned}
\Phi_{\LQG}\circ\Psi_{\LQG}(\mZ,\Xi)
={} & \!\left(
\gamma,
\Phi_\Lambda(\Psi_{\mK}(\Lambda,X,Y,\Xi),\Psi_P(X,Y,\Xi)),
(\Psi_P(X,Y,\Xi)^{-1})_{11},
Y,\Gamma,\Xi
\right) \\
={} &
(\gamma,\Lambda,X,Y,\Gamma,\Xi)
=(\mZ,\Xi),
\end{aligned}
\]
where the second step follows from \Cref{lemma:Phi_Psi_identity}. By the arbitrariness of $(\mZ,\Xi)\in \mathcal{F}^0_{\LQG}\times\mathrm{GL}_n$, we see that $\Phi_{\LQG}\circ\Psi_{\LQG}$ is the identity function on $\mathcal{F}^0_{\LQG}\times\mathrm{GL}_n$.

\item Let $(\mK,\gamma,P,\Gamma)\in\mathcal{L}^0_{\LQG}$ be arbitrary. Then by definition,
\[
\begin{aligned}
\Psi_{\LQG}\circ\Phi_{\LQG}(\mK,\gamma,P,\Gamma)
={} &
\!\left(\Psi_{\mK}\!\left(\Phi_\Lambda(\mK,P),
(P^{-1})_{11},
P_{11},P_{12}
\right),\gamma,
\Psi_P\!\left((P^{-1})_{11},
P_{11},P_{12}\right),\Gamma\right) \\
={} &
\left(\mK,\gamma,P,\Gamma\right),
\end{aligned}
\]
where we used \Cref{lemma:Psi_Phi_identity}. We see that $\Psi_{\LQG}\circ\Phi_{\LQG}$ is indeed the identity function on $\mathcal{L}^0_{\LQG}$.
\end{enumerate}

Summarizing the previous results, we can conclude that $\Phi_{\LQG}$ is a bijection from $\mathcal{L}^0_{\LQG}$ to $\mathcal{F}^0_{\LQG}\times\mathrm{GL}_n$ and $\Psi_{\LQG}$ is its inverse. The fact that $\Phi_{\LQG}$ and $\Psi_{\LQG}$ are both $C^2$ functions follows directly from their explicit expressions.

\vspace{6pt}
\noindent\textbf{Part II.} We next show that for any $(\mK,\gamma,P,\Gamma)\in\mathcal{L}^0_{\LQG}$, we have $(\mK,\gamma,P,\Gamma)\in\mathcal{L}_{\LQG}$ if and only if $\Phi_{\LQG}(\mK,\gamma,P,\Gamma)\in\mathcal{F}_{\LQG}\times\mathrm{GL}_n$. 
We shall prove a stronger result that will be used in later proofs:

\begin{lemma}
\label{lemma:LQG_ECL_basic_lemma}
Suppose $\mathscr{C}_1\subseteq\mathbb{S}^{3n+p}$ and $\mathscr{C}_2\subseteq\mathbb{S}^{3n+m}$ satisfy
\[
\begin{aligned}
& T_1^\tr M_1 T_1 \in \mathscr{C}_1,\qquad\forall M_1\in \mathscr{C}_1,\,T_1\in\mathrm{GL}_{3n+p}, \\
& T_2^\tr M_2 T_2 \in \mathscr{C}_2,\qquad\forall M_2\in \mathscr{C}_2,\,T_2\in\mathrm{GL}_{3n+m}.
\end{aligned}
\]
Let $(\mK,\gamma,P,\Gamma)\in\mathcal{L}^0_{\LQG}$ be arbitrary, and denote $(\gamma,\Lambda,X,Y,\Gamma,\Xi)=\Phi_{\LQG}(\mK,\gamma,P,\Gamma)$. Then
\[
-\begin{bmatrix} A_{\mathrm{cl}}(\mK)^\tr P+PA_{\mathrm{cl}}(\mK) & PB_{\mathrm{cl}}(\mK) \\ B_{\mathrm{cl}}(\mK)^\tr P & -\gamma I \end{bmatrix} 
        \in \mathscr{C}_1
        \quad\text{and}\quad
        \begin{bmatrix} P & C_{\mathrm{cl}}(\mK)^\tr \\ C_{\mathrm{cl}}(\mK) & \Gamma \end{bmatrix} \in \mathscr{C}_2
\]
if and only if
\[
-\mathcal{A}(\gamma,\Lambda,X,Y,\Gamma)
\in \mathscr{C}_1\quad\text{and}\quad
\mathcal{B}(\gamma,\Lambda,X,Y,\Gamma)
\in
\mathscr{C}_2.
\]
\end{lemma}
\begin{proof}[Proof of \Cref{lemma:LQG_ECL_basic_lemma}]
Denote $\Lambda=\begin{bmatrix}
0 & F \\ H & M
\end{bmatrix}$, i.e.,
\[
F = C_\mK(P^{-1})_{21},\quad
H = P_{12}B_\mK,\quad
M = \Phi_M(\mK,P).
\]
Let
\begin{equation}\label{eq:ECL_LQG_Tmat}
T = \begin{bmatrix}
(P^{-1})_{11} & I \\ (P^{-1})_{21} & 0
\end{bmatrix}\in\mathbb{R}^{2n\times 2n}.
\end{equation}
We can infer from $PP^{-1}=I$ that
\[
PT = \begin{bmatrix}
I & P_{11} \\ 0 & P_{12}^\tr
\end{bmatrix}.
\]
Since $P_{12}\in\mathrm{GL}_n$, we know that $PT$ is invertible, which implies that $T$ is also invertible.

Then, we note that
\begin{equation}
\label{eq:proof-diffeomorphism-TPAT}
\begin{aligned}
    & T^\tr  PA_\mathrm{cl}(\mK) T = \begin{bmatrix} I & 0 \\ P_{11} & P_{12} \end{bmatrix} \begin{bmatrix} A & B_2C_\mK \\ B_\mK C_2 & A_\mK \end{bmatrix} \begin{bmatrix} (P^{-1})_{11} & I \\ (P^{-1})_{21}  & 0 \end{bmatrix} \\
    ={} &\begin{bmatrix} A (P^{-1})_{11}+ B_2C_\mK (P^{-1})_{21} & A  \\ P_{11}A (P^{-1})_{11}+P_{12}B_\mK C_2 (P^{-1})_{11}+P_{11}B_2C_\mK (P^{-1})_{21}+P_{12}A _\mK(P^{-1})_{21} & P_{11}A +P_{12}B_{\mK}C_2 \end{bmatrix}\\
    ={} & \begin{bmatrix} AX+B_2F & A \\ M & YA+HC_2\end{bmatrix}.
\end{aligned}
\end{equation}
Similarly, it can be verified that
\begin{align}
    &T^\tr  PB_\mathrm{cl}(\mK)=\begin{bmatrix}
        B_1 \\ YB_1+HD_{21}
    \end{bmatrix}, \quad
    T^\tr C_\mathrm{cl}(\mK)^\tr=\begin{bmatrix}
        (C_1X+D_{12}F)^\tr \\ C_1^\tr
    \end{bmatrix}
    \label{eq:proof-diffeomorphism-TC}
\end{align}
Summarizing these identities \cref{eq:proof-diffeomorphism-TPAT} and \cref{eq:proof-diffeomorphism-TC}, we can show that      
\begin{align*}
    &\mathcal{A}(\gamma,\Lambda,X,Y,\Gamma)=\begin{bmatrix} T & 0 \\ 0 & I_{n+p} \end{bmatrix}^\tr  \begin{bmatrix} A_{\mathrm{cl}}(\mK)^\tr P+PA_{\mathrm{cl}}(\mK) & PB_{\mathrm{cl}}(\mK) \\ B_{\mathrm{cl}}(\mK)^\tr P & -\gamma I \end{bmatrix} \begin{bmatrix} T & 0 \\ 0 & I_{n+p} \end{bmatrix}, \\
    &\mathcal{B}(\gamma,\Lambda,X,Y,\Gamma)=\begin{bmatrix} T & 0 \\ 0 & I_{n+m} \end{bmatrix}^\tr  \begin{bmatrix} P & C_{\mathrm{cl}}(\mK)^\tr \\ C_{\mathrm{cl}}(\mK) & \Gamma \end{bmatrix}
    \begin{bmatrix} T & 0 \\ 0 & I_{n+m} \end{bmatrix}
\end{align*}
Since $\begin{bmatrix} T & 0 \\ 0 & I_{n+p} \end{bmatrix}$ and $\begin{bmatrix} T & 0 \\ 0 & I_{n+m} \end{bmatrix}$ are invertible, by the conditions on $\mathscr{C}_1$ and $\mathscr{C}_2$, we see that
\[
-\mathcal{A}(\gamma,\Lambda,X,Y,\Gamma)\in \mathscr{C}_1\quad\text{and}\quad
\mathcal{B}(\gamma,\Lambda,X,Y,\Gamma)
\in\mathscr{C}_2
\]
is equivalent to
\[
-\begin{bmatrix} A_{\mathrm{cl}}(\mK)^\tr P+PA_{\mathrm{cl}}(\mK) & PB_{\mathrm{cl}}(\mK) \\ B_{\mathrm{cl}}(\mK)^\tr P & -\gamma I \end{bmatrix}\in \mathscr{C}_1\quad\text{and}\quad
\begin{bmatrix} P & C_{\mathrm{cl}}(\mK)^\tr \\ C_{\mathrm{cl}}(\mK) & \Gamma \end{bmatrix}\in \mathscr{C}_2,
\]
which completes the proof.
\end{proof}

As a corollary of \Cref{lemma:LQG_ECL_basic_lemma}, by letting $\mathscr{C}_1=\mathbb{S}_+^{3n+p}$ and $\mathscr{C}_2=\mathbb{S}_+^{3n+m}$ and noting that $\gamma$ and $\Gamma$ remain unchanged under the mapping $\Phi_{\LQG}$, we see that
\[
\Phi_{\LQG}(\mathcal{L}_{\LQG})
=\mathcal{F}_{\LQG}\times\mathrm{GL}_n.
\]
We can now conclude that $\Phi_{\LQG}$ is a $C^2$ diffeomorphism from $\mathcal{L}_{\LQG}$ to $\mathcal{F}_{\LQG}\times\mathrm{GL}_n$.

\subsubsection{Proof of the Inclusion \cref{eq:LQG-inclusion} in LQG} \label{Appendix:proof-inclusion-LQG}

\noindent\textbf{Part I.} We first prove the inclusion
\[
\operatorname{epi}_> (J_{\LQG,n})\subseteq \pi_{\mK,\gamma}(\mathcal{L}_{\LQG}).
\]
Let $(\mK,\gamma)\in\operatorname{epi}_>(J_{\LQG,n})$ be arbitrary. We then have $\mK\in {\mathcal{C}}_{n,0}$ and $\gamma> J_{\LQG,n}(\mK)$. Consequently, $A_{\mathrm{cl}}(\mK)$ is stable, and the transfer matrix
$C_{\mathrm{cl}}(\mK)
(sI-A_{\mathrm{cl}}(\mK))^{-1}B_{\mathrm{cl}}(\mK)$
has an $\mathcal{H}_2$ norm that is strictly upper bounded by $\gamma$. By applying the strict LMI in \Cref{lemma:H2norm}, we see that there exist $P\in\mathbb{S}^{2n}_{++}$ and $\Gamma\in \mathbb{S}^{n+m}_{++}$ such that
\begin{align*}
\begin{bmatrix}
A_{\mathrm{cl}}(\mK)^\tr P + PA_{\mathrm{cl}}(\mK) & PB_{\mathrm{cl}}(\mK) \\
B_{\mathrm{cl}}(\mK)^\tr P & -\gamma I
\end{bmatrix}
\prec 0,
\quad
\begin{bmatrix}
P & C_{\mathrm{cl}}(\mK)^\tr \\
C_{\mathrm{cl}}(\mK) & \Gamma
\end{bmatrix}\succ 0,
\quad \operatorname{trace}(\Gamma) < \gamma.
\end{align*}
Since the above inequalities are all strict, we can add a sufficiently small perturbation on the matrix $P$ to ensure that $P_{12}$ is invertible without violating the above inequalities. By comparing the obtained properties of $\mK$, $\gamma$, $P$ and $\Gamma$ with the definition of $\mathcal{L}_{\LQG}$, we see that $(\mK, \gamma, P, \Gamma)\in \mathcal{L}_{\LQG}$. By the arbitrariness of $(\mK,\gamma)\in \operatorname{epi}_> (J_{\LQG,n})$, we get the inclusion $\operatorname{epi}_> (J_{\LQG,n})\subseteq \pi_{\mK,\gamma}(\mathcal{L}_{\LQG})$.

\vspace{6pt}
\noindent\textbf{Part II.} We next prove the other inclusion 
\[
\pi_{\mK,\gamma}(\mathcal{L}_{\LQG})\subseteq \operatorname{cl}\operatorname{epi}_\geq (J_{\LQG,n}).
\]
This part is more involved compared to Part I. In particular, the proof will rely on \Cref{lemma:ECL_second_inclusion}.
Define the affine function $h:\mathcal{F}^0_{\LQG}\rightarrow\mathbb{S}^{3n+p}\times\mathbb{S}^{3n+m}\times\mathbb{R}$ by
\[
h( \gamma, \Lambda,X, Y, \Gamma)
=
\left(
-{\mathcal{A}}(\gamma, \Lambda,X,Y,\Gamma),
{\mathcal{B}}(\gamma, \Lambda,X,Y,\Gamma),
\gamma-\operatorname{tr}(\Gamma)\right),
\qquad(\gamma,\Lambda,X,Y,\Gamma)\in\mathcal{F}^0_{\LQG}.
\]
Also, define
\[
C= \mathbb{S}^{3n+p}_{+} \times \mathbb{S}^{3n+m}_{+} \times [0,+\infty).
\]
It can be seen from the definition of $\mathcal{F}_{\LQG}$ that $\mathcal{F}_{\LQG}=h^{-1}(C)$, and consequently,
\[
\mathcal{L}_{\LQG}
=\Psi(h^{-1}(C)\times\mathrm{GL}_n).
\]
Now let $(\gamma,\Lambda,X,Y,\Gamma)\in \mathcal{F}^0_{\LQG}$ and $\Xi\in \mathrm{GL}_n$ be arbitrary, and denote $\mK=\Phi_{\mK}(\Lambda,X,Y,\Xi),P=\Phi_{P}(X,Y,\Xi)$. Since
\[
\operatorname{int}C
=\mathbb{S}^{3n+p}_{++}\times\mathbb{S}^{3n+m}_{++}\times (0,+\infty),
\]
by applying \Cref{lemma:LQG_ECL_basic_lemma} with $\mathscr{C}_1=\mathbb{S}^{3n+p}_{++}$, $\mathscr{C}_2=\mathbb{S}^{3n+m}_{++}$, we get
\[
\begin{aligned}
(\gamma,\Lambda,X,Y,\Gamma)
\in h^{-1}(\operatorname{int}C)
\quad\Longleftrightarrow\quad
& \begin{cases}
-\mathcal{A}(\gamma,\Lambda,X,Y,\Gamma)
\in\mathbb{S}^{3n+p}_{++}, \\
\mathcal{B}(\gamma,\Lambda,X,Y,\Gamma)
\in\mathbb{S}^{3n+m}_{++},\;
\operatorname{tr}(\Gamma)<\gamma,
\end{cases} \\
\quad\Longleftrightarrow\quad & \begin{cases}
-\begin{bmatrix} A_{\mathrm{cl}}(\mK)^\tr P+PA_{\mathrm{cl}}(\mK) & PB_{\mathrm{cl}}(\mK) \\ B_{\mathrm{cl}}(\mK)^\tr P & -\gamma I \end{bmatrix}\in \mathbb{S}^{3n+p}_{++}, \\[10pt]
\begin{bmatrix} P & C_{\mathrm{cl}}(\mK)^\tr \\ C_{\mathrm{cl}}(\mK) & \Gamma \end{bmatrix}
\in \mathbb{S}^{3n+m}_{++},\;
\operatorname{tr}(\Gamma) < \gamma,
\end{cases}
\end{aligned}
\]
which implies that
\[
\begin{aligned}
\tilde{\mathcal{L}}_{\LQG}\coloneqq{} & 
\Psi(h^{-1}(\operatorname{int}C)\times\mathrm{GL}_n) \\
={} &
\left\{
(\mK,\gamma,P,\Gamma)\in\mathcal{L}^0_{\LQG}
\left|\,
         \begin{aligned}
        &\begin{bmatrix} A_{\mathrm{cl}}(\mK)^\tr P+PA_{\mathrm{cl}}(\mK) & PB_{\mathrm{cl}}(\mK) \\ B_{\mathrm{cl}}(\mK)^\tr P & -\gamma I \end{bmatrix} \prec 0, \\
        &\begin{bmatrix} P & C_{\mathrm{cl}}(\mK)^\tr \\ C_{\mathrm{cl}}(\mK) & \Gamma \end{bmatrix} \succ 0,\;  \mathrm{tr}(\Gamma) < \gamma
        \end{aligned}
\right.
\right\} \\
={} & 
\left\{
(\mK,\gamma,P,\Gamma)
\left|\begin{array}{c} 
       \mK\in\mathcal{V}_{n,0},\; \gamma\in\mathbb{R},\;
         P \in \mathbb{S}^{2n}_{++}, \; P_{12} \in \mathrm{GL}_n, \;\Gamma\in\mathbb{S}^{n+m}, \\
         \begin{aligned}
        &\begin{bmatrix} A_{\mathrm{cl}}(\mK)^\tr P+PA_{\mathrm{cl}}(\mK) & PB_{\mathrm{cl}}(\mK) \\ B_{\mathrm{cl}}(\mK)^\tr P & -\gamma I \end{bmatrix} \prec 0, \\
        &\begin{bmatrix} P & C_{\mathrm{cl}}(\mK)^\tr \\ C_{\mathrm{cl}}(\mK) & \Gamma \end{bmatrix} \succ 0,\;  \mathrm{tr}(\Gamma) < \gamma
        \end{aligned}
        \end{array}
\right.
\right\}.
\end{aligned}
\]
By using the strict LMI in \Cref{lemma:H2norm}, we see that $(\mK,\gamma,P,\Gamma)\in \tilde{\mathcal{L}}_{\LQG}$ implies $(\mK,\gamma)\in \operatorname{epi}_>(J_{\LQG,n})$, i.e.,
\[
\pi_{\mK,\gamma}(\tilde{\mathcal{L}}_{\LQG})
\subseteq \operatorname{epi}_>(J_{\LQG,n}).
\]
On the other hand, \Cref{lemma:ECL_second_inclusion} shows that
\[
\mathcal{L}_{\LQG}
\subseteq \operatorname{cl}\tilde{\mathcal{L}}_{\LQG}.
\]
As a result, we have
\[
\pi_{\mK,\gamma}(\mathcal{L}_{\LQG})
\subseteq\pi_{\mK,\gamma}
(\operatorname{cl}\tilde{\mathcal{L}}_{\LQG})
\subseteq \operatorname{cl}
\pi_{\mK,\gamma}(\tilde{\mathcal{L}}_{\LQG})
\subseteq\operatorname{cl}\operatorname{epi}_>(J_{\LQG,n})
=\operatorname{cl}\operatorname{epi}_{\geq}(J_{\LQG,n}),
\]
and the proof is complete.

\subsection{Proofs for $\mathcal{H}_\infty$ Output Feedback Control} \label{appendix:Hinf-control}

We here provide the proof details for \Cref{lemma:inclusion-projection-Hinf}, which include two parts:
\begin{itemize}
    \item The mapping $\Phi_{\infty,\mathrm{d}}$ defined in \cref{eq:Hinf-Phi} is a $C^2$ diffeomorphism from $\mathcal{L}_{\infty,\mathrm{d}}$ to $\mathcal{F}_{\infty,\mathrm{d}}\times\mathrm{GL}_n$, and its inverse $\Phi^{-1}_{\infty,\mathrm{d}}$ is given by $\Psi_{\infty,\mathrm{d}}$ defined in \cref{eq:Hinf-Psi}. This is proved in \cref{appendix:diffeomorphsim-Phi-Hinf}. 
    \item The following inclusion relationship holds
    \begin{equation} \label{eq:Hinf-inclusion}
    \operatorname{epi}_>(J_{\infty,n})  \subseteq \pi_{\mK, \gamma} (\mathcal{L}_{\infty,\mathrm{d}}) \subseteq \operatorname{cl}\, \operatorname{epi}_{\geq}(J_{\infty,n}). 
    \end{equation}
    This is proved in \cref{Appendix:proof-inclusion-Hinf}. 
\end{itemize}
For both of them, the proofs follow similarly as those in  
in \Cref{appendix:LQG-control}.

\subsubsection{Diffeomorphism from $\mathcal{L}_{\infty,\mathrm{d}}$ to $\mathcal{F}_{\infty,\mathrm{d}}\times\mathrm{GL}_{n}$} \label{appendix:diffeomorphsim-Phi-Hinf}

Denote
\begin{align*}
\mathcal{L}^0_{\infty,\mathrm{d}}
& =\left\{ (\mK, \gamma, P) \left|\,
       \mK\in \mathbb{R}^{(m+n)\times(p+n)}, \;\gamma\in\mathbb{R},\;
         P \in \mathbb{S}^{2n}_{++}, \; P_{12} \in \mathrm{GL}_n
        \right.\right\}, \\
\mathcal{F}^0_{\infty,\mathrm{d}}
& =\left\{
\left(\gamma, \Lambda,X, Y \right) \left|\,
\gamma\in\mathbb{R},\;
\Lambda\in\mathbb{R}^{(m+n)\times(p+n)},\;
       \begin{bmatrix}
        X & I \\
        I & Y
\end{bmatrix}\in\mathbb{S}_{++}^{2n}
        \right.\right\}.
\end{align*}
Evidently $\mathcal{L}^0_{\infty,\mathrm{d}}$ and $\mathcal{F}^0_{\infty,\mathrm{d}}$ can be identified with certain open subsets of some Euclidean spaces, and $\mathcal{F}^0_{\infty,\mathrm{d}}$ is convex. Moreover,
\begin{align*}
\mathcal{L}_{\infty,\mathrm{d}}
&=\left\{(\mK,\gamma,P,\Gamma)\in\mathcal{L}^0_{\infty,\mathrm{d}}
\left|
\begin{bmatrix}
    A_{\mathrm{cl}}(\mK)^\tr P \!+\! P A_{\mathrm{cl}}(\mK) & PB_{\mathrm{cl}}(\mK) & C_{\mathrm{cl}}(\mK)^\tr \\
    B_{\mathrm{cl}}(\mK)^\tr P & -\gamma I & D_{\mathrm{cl}}(\mK)^\tr \\
    C_{\mathrm{cl}}(\mK) & D_{\mathrm{cl}}(\mK) & -\gamma I
    \end{bmatrix} \preceq 0
\right.\right\}, \\
\mathcal{F}_{\infty,\mathrm{d}}
&=\left\{\left(\gamma, \Lambda,X, Y, \Gamma \right) \in\mathcal{F}^0_{\infty,\mathrm{d}}
\left|\,
-\mathscr{M}(\gamma,\Lambda,X,Y)\in
\mathbb{S}_{+}^{4n+p+m}
\right.\right\},
\end{align*}
where we remind the readers of the notation
\[
    \begin{aligned}
& \mathscr{M}\left(\gamma,\begin{bmatrix}
G & F \\
H & M
\end{bmatrix},X,Y\right) \\
\coloneqq &\!
\begin{bmatrix}
AX + B_2F + (AX + B_2F)^\tr & M^\tr + A + B_2GC_2 & B_1 + B_2GD_{21} & (C_1 X \!+\! D_{12}F)^\tr \\
* & YA \!+\! HC_2 \!+\! (YA \!+\! HC_2)^\tr & YB_1 \!+\! HD_{21} & (C_1 \!+\! D_{12}GC_2)^\tr \\
* & * & -\gamma I & (D_{11} \!+\! D_{12}GD_{21})^\tr \\
* & * & * & -\gamma I
\end{bmatrix}\!. 
\end{aligned}
\]
We extend the domain of $\Phi_{\infty,\mathrm{d}}$ so that it is now defined on the larger set $\mathcal{L}^0_{\infty,\mathrm{d}}$:
\[
\Phi_{\infty,\mathrm{d}}(\mK,\gamma,P)=\left(\gamma,
\Phi_\Lambda(\mK,P),(P^{-1})_{11},P_{11}, P_{12}\right),
\quad
\forall (\mK,\gamma,P)\in\mathcal{L}^0_{\infty,\mathrm{d}},
\]
where $\Phi_\Lambda(\mK,P)$ is defined in \eqref{eq:Phi_M_V_def}.
We also extend the domain of $\Psi_{\infty,\mathrm{d}}$ so that it is now defined on $\mathcal{F}^0_{\infty,\mathrm{d}}\times\mathrm{GL}_n$:
\[
\Psi_{\infty,\mathrm{d}}(\mZ,\Xi)=\left(\Psi_{\mK}(\Lambda,X,Y,\Xi),\gamma, \Psi_P(X,Y,\Xi)\right),
\qquad\forall \mZ = \left(\gamma, 
\Lambda, X,Y \right) \in \mathcal{F}^{0}_{\infty,\mathrm{d}},\ \Xi\in\mathrm{GL}_n,
\]
where $\Psi_{\mK}(\Lambda,X,Y,\Xi)$ and $\Psi_P(X,Y,\Xi)$ are defined in~\eqref{eq:Psi_K_P_def}.

\vspace{3pt}
Our subsequent proof consists of the following parts:
\begin{enumerate}
\item Showing that $\Phi_{\infty,\mathrm{d}}$ is a $C^2$ diffeomorphism from $\mathcal{L}^0_{\infty,\mathrm{d}}$ to $\mathcal{F}^0_{\infty,\mathrm{d}}\times\mathrm{GL}_n$, with $\Phi_{\infty,\mathrm{d}}^{-1}=\Psi_{\infty,\mathrm{d}}$.

\item Showing that for any $(\mK,\gamma,P)\in\mathcal{L}^0_{\infty,\mathrm{d}}$, we have $(\mK,\gamma,P)\in\mathcal{L}_{\infty,\mathrm{d}}$ if and only if 
$\Phi_{\infty,\mathrm{d}}(\mK,\gamma,P)\in\mathcal{F}_{\infty,\mathrm{d}}\times\mathrm{GL}_n$.
\end{enumerate}

\vspace{3pt}
\noindent\textbf{Part I.} We first show that
\[
\Psi_{\infty,\mathrm{d}}(\mathcal{F}^0_{\infty,\mathrm{d}}\times\mathrm{GL}_n)\subseteq\mathcal{L}^0_{\infty,\mathrm{d}}.
\]
By the definitions of $\mathcal{L}^0_{\infty,\mathrm{d}}$, $\mathcal{F}^0_{\infty,\mathrm{d}}$ and $\Psi_{\infty,\mathrm{d}}$, it is not hard to see that this inclusion holds as long as $\Psi_P(X,Y,\Xi)\in\mathbb{S}^{2n}_{++}$ for any $(\gamma,\Lambda,X,Y)\in\mathcal{F}^0_{\infty,\mathrm{d}}$ and $\Xi\in\mathrm{GL}_n$, which can be directly justified by the first part of \Cref{lemma:Psi_P_basic_properties}.

Next, we show that
\[
\Phi(\mathcal{L}^0_{\infty,\mathrm{d}})\subseteq\mathcal{F}^0_{\infty,\mathrm{d}}\times\mathrm{GL}_n.
\]
By definition, we only need to show
\[
\begin{bmatrix}
(P^{-1})_{11} & I \\
I & P_{11}
\end{bmatrix} \succ 0
\qquad\forall P\in\mathbb{S}^{2n}_{++} \text{ with }P_{12}\in\mathrm{GL}_n,
\]
which can be directly justified by the second part of \Cref{lemma:Psi_P_basic_properties}.

We are now allowed to form the compositions $\Psi_{\infty,\mathrm{d}}\circ\Phi_{\infty,\mathrm{d}}:\mathcal{L}^0_{\infty,\mathrm{d}}\rightarrow\mathcal{L}^0_{\infty,\mathrm{d}}$ and $\Phi_{\infty,\mathrm{d}}\circ\Psi_{\infty,\mathrm{d}}:\mathcal{F}^0_{\infty,\mathrm{d}}\times\mathrm{GL}_n\rightarrow \mathcal{F}^0_{\infty,\mathrm{d}}\times\mathrm{GL}_n$. We will now show that these two compositions are the identity functions on their domains:
\begin{enumerate}
\item Let $\mZ=\left(\gamma,\Lambda,X,Y\right)\in\mathcal{F}^0_{\infty,\mathrm{d}}$ and $\Xi\in\mathrm{GL}_n$ be arbitrary. We then have
\[
\begin{aligned}
\Phi_{\infty,\mathrm{d}}\circ\Psi_{\infty,\mathrm{d}}(\mZ,\Xi)
={} &\! \left(
\gamma,
\Phi_\Lambda(\Psi_{\mK}(\Lambda,X,Y,\Xi),\Psi_P(X,Y,\Xi)),
(\Psi_P(X,Y,\Xi)^{-1})_{11},
Y,\Xi
\right) \\
={} &
(\gamma,\Lambda,X,Y,\Xi)
=(\mZ,\Xi),
\end{aligned}
\]
where the second step follows from \Cref{lemma:Phi_Psi_identity}. By the arbitrariness of $(\mZ,\Xi)\in \mathcal{F}^0_{\infty,\mathrm{d}}\times\mathrm{GL}_n$, we see that $\Phi_{\infty,\mathrm{d}}\circ\Psi_{\infty,\mathrm{d}}$ is the identity function on $\mathcal{F}^0_{\infty,\mathrm{d}}\times\mathrm{GL}_n$.

\item Let $(\mK,\gamma,P,\Gamma)\in\mathcal{L}^0_{\infty,\mathrm{d}}$ be arbitrary. Then by definition,
\[
\begin{aligned}
\Psi_{\infty,\mathrm{d}}\circ\Phi_{\infty,\mathrm{d}}(\mK,\gamma,P)
={} &
\!\left(\Psi_{\mK}\!\left(\Phi_\Lambda(\mK,P),
(P^{-1})_{11},
P_{11},P_{12}
\right),\gamma,
\Psi_P\!\left((P^{-1})_{11},
P_{11},P_{12}\right)\right) \\
={} &
\left(\mK,\gamma,P\right),
\end{aligned}
\]
where we used \Cref{lemma:Psi_Phi_identity}. We see that $\Psi_{\infty,\mathrm{d}}\circ\Phi_{\infty,\mathrm{d}}$ is indeed the identity function on $\mathcal{L}^0_{\infty,\mathrm{d}}$.
\end{enumerate}

Summarizing the previous results, we can conclude that $\Phi_{\infty,\mathrm{d}}$ is a bijection from $\mathcal{L}^0_{\infty,\mathrm{d}}$ to $\mathcal{F}^0_{\infty,\mathrm{d}}\times\mathrm{GL}_n$ and $\Psi_{\infty,\mathrm{d}}$ is its inverse. The fact that $\Phi_{\infty,\mathrm{d}}$ and $\Psi_{\infty,\mathrm{d}}$ are both $C^2$ functions follow directly from their explicit expressions.

\vspace{6pt}
\noindent\textbf{Part II.} We next show that for any $(\mK,\gamma,P)\in\mathcal{L}^0_{\infty,\mathrm{d}}$, we have $(\mK,\gamma,P)\in\mathcal{L}_{\infty,\mathrm{d}}$ if and only if $\Phi_{\infty,\mathrm{d}}(\mK,\gamma,P)\in\mathcal{F}_{\infty,\mathrm{d}}\times\mathrm{GL}_n$. 
We shall prove a stronger result that will be used in later proofs:

\begin{lemma}
\label{lemma:Hinf_ECL_basic_lemma}
Suppose $\mathscr{C}\subseteq\mathbb{S}^{4n+p+m}$ satisfy
\[
\begin{aligned}
& T^\tr M T \in \mathscr{C},\qquad\forall M\in \mathscr{C},\,T\in\mathrm{GL}_{4n+p+m}.
\end{aligned}
\]
Let $(\mK,\gamma,P)\in\mathcal{L}^0_{\infty,\mathrm{d}}$ be arbitrary, and denote $(\gamma,\Lambda,X,Y,\Xi)=\Phi_{\infty,\mathrm{d}}(\mK,\gamma,P)$. Then
\[
-\begin{bmatrix}
    A_{\mathrm{cl}}(\mK)^\tr P \!+\! P A_{\mathrm{cl}}(\mK) & PB_{\mathrm{cl}}(\mK) & C_{\mathrm{cl}}(\mK)^\tr \\
    B_{\mathrm{cl}}(\mK)^\tr P & -\gamma I & D_{\mathrm{cl}}(\mK)^\tr \\
    C_{\mathrm{cl}}(\mK) & D_{\mathrm{cl}}(\mK) & -\gamma I
    \end{bmatrix} 
        \in \mathscr{C}
\]
if and only if
\[
-\mathscr{M}(\gamma,\Lambda,X,Y)
\in \mathscr{C}.
\]
\end{lemma}
\begin{proof}[Proof of \Cref{lemma:Hinf_ECL_basic_lemma}]
Denote $\Lambda=\begin{bmatrix}
D_{\mK} & F \\ H & M
\end{bmatrix}$, i.e.,
\[
F = D_\mK C_2(P^{-1})_{11} + C_\mK(P^{-1})_{21},\quad
H = P_{11}B_2D_\mK + P_{12}B_\mK,\quad
M = \Phi_M(\mK,P).
\]
We again introduce the matrix
\begin{equation}
T = \begin{bmatrix}
(P^{-1})_{11} & I \\ (P^{-1})_{21} & 0
\end{bmatrix}\in\mathbb{R}^{2n\times 2n},
\tag{\ref{eq:ECL_LQG_Tmat}}
\end{equation}
and $PP^{-1}=I$ implies that
\[
PT = \begin{bmatrix}
I & P_{11} \\ 0 & P_{12}^\tr
\end{bmatrix}.
\]
Since $P_{12}\in\mathrm{GL}_n$, we know that $PT$ is invertible, which implies that $T$ is also invertible.

Then, it can be checked by straightforward calculation that the following equality holds:
\begin{equation}
\label{eq:proof-diffeomorphism-TPAT_hinf}
\begin{aligned}
    T^\tr  PA_\mathrm{cl}(\mK) T ={} & \begin{bmatrix} I & 0 \\ P_{11} & P_{12} \end{bmatrix} \begin{bmatrix} A + B_2D_\mK C_2 & B_2C_\mK \\ B_\mK C_2 & A_\mK \end{bmatrix} \begin{bmatrix} (P^{-1})_{11} & I \\ (P^{-1})_{21}  & 0 \end{bmatrix} \\
    ={} & \begin{bmatrix} AX+B_2F & A+B_2GC_2 \\ M & YA+HC_2\end{bmatrix}.
\end{aligned}
\end{equation}
Similarly, it can be verified that
\begin{align}
    &T^\tr  PB_\mathrm{cl}(\mK)=\begin{bmatrix}
        B_1 + B_2 GD_{21} \\ YB_1+HD_{21}
    \end{bmatrix}, \quad
    T^\tr C_\mathrm{cl}(\mK)^\tr=\begin{bmatrix}
        (C_1X+D_{12}F)^\tr \\ (C_1+D_{12}GC_2)^\tr
    \end{bmatrix}
    \label{eq:proof-diffeomorphism-TC_hinf}
\end{align}
Summarizing these identities \cref{eq:proof-diffeomorphism-TPAT_hinf} and \cref{eq:proof-diffeomorphism-TC_hinf}, we can show that      
\begin{align*}
\mathscr{M}(\gamma,\Lambda,X,Y) =\!
    \begin{bmatrix}
    T & 0 & 0 \\
    0 & I_{n+p} & 0 \\
    0 & 0 & I_{n+m}
    \end{bmatrix}^{\!\tr}\!\!
    \begin{bmatrix}
    A_{\mathrm{cl}}(\mK)^\tr P \!+\! P A_{\mathrm{cl}}(\mK) & PB_{\mathrm{cl}}(\mK) & C_{\mathrm{cl}}(\mK)^\tr \\
    B_{\mathrm{cl}}(\mK)^\tr P & -\gamma I & D_{\mathrm{cl}}(\mK)^\tr \\
    C_{\mathrm{cl}}(\mK) & D_{\mathrm{cl}}(\mK) & -\gamma I
    \end{bmatrix}\!\!
    \begin{bmatrix}
    T & 0 & 0 \\
    0 & I_{n+p} & 0 \\
    0 & 0 & I_{n+m}
    \end{bmatrix}\!\!.
\end{align*}
Since $\begin{bmatrix}
    T & 0 & 0 \\
    0 & I_{n+p} & 0 \\
    0 & 0 & I_{n+m}
    \end{bmatrix}$ is invertible, by the condition on $\mathscr{C}$, we see that
\[
-\mathscr{M}(\gamma,\Lambda,X,Y)\in \mathscr{C}
\]
is equivalent to
\[
-\begin{bmatrix}
    A_{\mathrm{cl}}(\mK)^\tr P \!+\! P A_{\mathrm{cl}}(\mK) & PB_{\mathrm{cl}}(\mK) & C_{\mathrm{cl}}(\mK)^\tr \\
    B_{\mathrm{cl}}(\mK)^\tr P & -\gamma I & D_{\mathrm{cl}}(\mK)^\tr \\
    C_{\mathrm{cl}}(\mK) & D_{\mathrm{cl}}(\mK) & -\gamma I
    \end{bmatrix}\in \mathscr{C},
\]
which completes the proof.
\end{proof}

As a corollary of \Cref{lemma:Hinf_ECL_basic_lemma}, by letting $\mathscr{C}=\mathbb{S}_+^{4n+p+m}$, we see that
\[
\Phi_{\infty,\mathrm{d}}(\mathcal{L}_{\infty,\mathrm{d}})
=\mathcal{F}_{\infty,\mathrm{d}}\times\mathrm{GL}_n.
\]
We can now conclude that $\Phi_{\infty,\mathrm{d}}$ is a $C^2$ diffeomorphism from $\mathcal{L}_{\infty,\mathrm{d}}$ to $\mathcal{F}_{\infty,\mathrm{d}}\times\mathrm{GL}_n$.

\subsubsection{Proof of the Inclusion \cref{eq:Hinf-inclusion} in $\mathcal{H}_\infty$ Output Feedback Control} \label{Appendix:proof-inclusion-Hinf}

The proof follows almost the same outline as the proof 
in \Cref{Appendix:proof-inclusion-LQG}.

\vspace{12pt}
\noindent\textbf{Part I.} We first prove the inclusion $\operatorname{epi}_> (J_{\infty,n})\subseteq \pi_{\mK,\gamma}(\mathcal{L}_{\infty,\mathrm{d}})$. Let $(\mK,\gamma)\in\operatorname{epi}_>(J_{\infty,n})$ be arbitrary. We then have $\mK\in \mathcal{C}_n$ and $\gamma> J_{\infty,n}(\mK)$. Consequently, $A_{\mathrm{cl}}(\mK)$ is stable, and the transfer matrix
$C_{\mathrm{cl}}(\mK)
(sI-A_{\mathrm{cl}}(\mK))^{-1}B_{\mathrm{cl}}(\mK)+D_{\mathrm{cl}}(\mK)$
has an $\mathcal{H}_\infty$ norm that is strictly upper bounded by $\gamma$. By applying the strict LMI in \Cref{lemma:bounded_real}, we see that there exists a $P\in\mathbb{S}^{2n}_{++}$ such that
\begin{align*}
\begin{bmatrix}
A_{\mathrm{cl}}(\mK)^\tr P + PA_{\mathrm{cl}}(\mK) & PB_{\mathrm{cl}}(\mK) & C_{\mathrm{cl}}(\mK)^\tr \\
B_{\mathrm{cl}}(\mK)^\tr P & -\gamma I & D_{\mathrm{cl}}(\mK)^\tr \\
C_{\mathrm{cl}}(\mK) & D_{\mathrm{cl}}(\mK) & -\gamma I
\end{bmatrix}
\prec 0.
\end{align*}
Since the above inequalities are all strict, we can add a sufficiently small perturbation on the matrix $P$ to ensure that $P_{12}$ is invertible without violating the above inequalities. By comparing the obtained properties of $\mK$, $\gamma$ and $P$ with the definition of $\mathcal{L}_{\infty,\mathrm{d}}$, we see that $(\mK, \gamma, P)\in \mathcal{L}_{\infty,\mathrm{d}}$. By the arbitrariness of $(\mK,\gamma)\in \operatorname{epi}_> (J_{\infty,n})$, we get the inclusion $\operatorname{epi}_> (J_{\infty,n})\subseteq \pi_{\mK,\gamma}(\mathcal{L}_{\infty,\mathrm{d}})$.

\vspace{6pt}
\noindent\textbf{Part II.} We next prove the other inclusion $\pi_{\mK,\gamma}(\mathcal{L}_{\infty,\mathrm{d}})\subseteq \operatorname{cl}\operatorname{epi}_> (J_{\infty,n})$. Define $h:\mathcal{F}^0_{\infty,\mathrm{d}}\rightarrow\mathbb{S}^{4n+m+p}$ by
\[
h(\gamma,\Lambda,X,Y)=-\mathscr{M}(\gamma,\Lambda,X,Y),
\qquad\forall(\gamma,\Lambda,X,Y)\in\mathcal{F}^0_{\infty,\mathrm{d}},
\]
and let
\[
C= \mathbb{S}^{4n+m+p}_{+}.
\]
It can be seen from the definition of $\mathcal{F}_{\infty,\mathrm{d}}$ that $\mathcal{F}_{\infty,\mathrm{d}}=h^{-1}(C)$, and consequently,
\[
\mathcal{L}_{\infty,\mathrm{d}}
=\Psi_{\infty,\mathrm{d}}(h^{-1}(C)\times\mathrm{GL}_n).
\]
Now let $(\gamma,\Lambda,X,Y)\in \mathcal{F}^0_{\infty,\mathrm{d}}$ and $\Xi\in \mathrm{GL}_n$ be arbitrary, and denote $\mK=\Phi_{\mK}(\Lambda,X,Y,\Xi),P=\Phi_{P}(X,Y,\Xi)$. Since
\[
\operatorname{int}C
=\mathbb{S}^{4n+p+m}_{++},
\]
by applying \Cref{lemma:Hinf_ECL_basic_lemma} with $\mathscr{C}=\mathbb{S}^{4n+p+m}_{++}$, we get
\[
\begin{aligned}
(\gamma,\Lambda,X,Y)
\in h^{-1}(\operatorname{int}C)
\quad\Longleftrightarrow\quad
& 
-\mathscr{M}(\gamma,\Lambda,X,Y)
\in\mathbb{S}^{4n+p+m}_{++} \\
\quad\Longleftrightarrow\quad &
-\begin{bmatrix}
A_{\mathrm{cl}}(\mK)^\tr P + PA_{\mathrm{cl}}(\mK) & PB_{\mathrm{cl}}(\mK) & C_{\mathrm{cl}}(\mK)^\tr \\
B_{\mathrm{cl}}(\mK)^\tr P & -\gamma I & D_{\mathrm{cl}}(\mK)^\tr \\
C_{\mathrm{cl}}(\mK) & D_{\mathrm{cl}}(\mK) & -\gamma I
\end{bmatrix}\in \mathbb{S}^{4n+p+m}_{++},
\end{aligned}
\]
which implies that
\[
\begin{aligned}
\tilde{\mathcal{L}}_{\infty,\mathrm{d}}\coloneqq{} & 
\Psi_{\infty,\mathrm{d}}(h^{-1}(\operatorname{int}C)\times\mathrm{GL}_n) \\
={} &
\left\{
(\mK,\gamma,P)\in\mathcal{L}^0_{\infty,\mathrm{d}}
\left|\,
\begin{bmatrix}
A_{\mathrm{cl}}(\mK)^\tr P + PA_{\mathrm{cl}}(\mK) & PB_{\mathrm{cl}}(\mK) & C_{\mathrm{cl}}(\mK)^\tr \\
B_{\mathrm{cl}}(\mK)^\tr P & -\gamma I & D_{\mathrm{cl}}(\mK)^\tr \\
C_{\mathrm{cl}}(\mK) & D_{\mathrm{cl}}(\mK) & -\gamma I
\end{bmatrix}\preceq 0
\right.
\right\} \\
={} & 
\left\{
(\mK,\gamma,P)
\left|\begin{array}{c} 
       \mK\in\mathbb{R}^{(m+n)\times(p+n)},\; \gamma\in\mathbb{R},\;
         P \in \mathbb{S}^{2n}_{++}, \; P_{12} \in \mathrm{GL}_n, \\
        \begin{bmatrix}
A_{\mathrm{cl}}(\mK)^\tr P + PA_{\mathrm{cl}}(\mK) & PB_{\mathrm{cl}}(\mK) & C_{\mathrm{cl}}(\mK)^\tr \\
B_{\mathrm{cl}}(\mK)^\tr P & -\gamma I & D_{\mathrm{cl}}(\mK)^\tr \\
C_{\mathrm{cl}}(\mK) & D_{\mathrm{cl}}(\mK) & -\gamma I
\end{bmatrix}\preceq 0
        \end{array}
\right.
\right\}.
\end{aligned}
\]
By using the strict LMI in \Cref{lemma:bounded_real}, we see that $(\mK,\gamma,P)\in \tilde{\mathcal{L}}_{\infty,\mathrm{d}}$ implies $(\mK,\gamma)\in \operatorname{epi}_>(J_{\infty,n})$, i.e.,
\[
\pi_{\mK,\gamma}(\tilde{\mathcal{L}}_{\infty,\mathrm{d}})
\subseteq \operatorname{epi}_>(J_{\infty,n}).
\]
On the other hand, \Cref{lemma:ECL_second_inclusion} shows that
\[
\mathcal{L}_{\infty,\mathrm{d}}
\subseteq \operatorname{cl}\tilde{\mathcal{L}}_{\infty,\mathrm{d}}.
\]
As a result, we have
\[
\pi_{\mK,\gamma}(\mathcal{L}_{\infty,\mathrm{d}})
\subseteq\pi_{\mK,\gamma}
(\operatorname{cl}\tilde{\mathcal{L}}_{\infty,\mathrm{d}})
\subseteq \operatorname{cl}
\pi_{\mK,\gamma}(\tilde{\mathcal{L}}_{\infty,\mathrm{d}})
\subseteq\operatorname{cl}\operatorname{epi}_>(J_{\infty,n})
=\operatorname{cl}\operatorname{epi}_{\geq}(J_{\infty,n}),
\]
and the proof is complete.